\documentclass{article}

\usepackage[colorlinks,citecolor=blue,linkcolor=blue,urlcolor=blue,backref=page]{hyperref}
\usepackage{amsmath,amssymb,amsthm,asymptote,dsfont,enumerate,float,mathrsfs,stmaryrd,upref,xcolor}
\usepackage[ngerman,main=american]{babel}\addto\extrasamerican{\languageshorthands{ngerman}\useshorthands{"}}
\usepackage[tmargin=2.5cm,lmargin=2.5cm,rmargin=2.5cm,bmargin=2.5cm]{geometry}

\newcommand\1{{\mathds{1}}}
\renewcommand\ae{{a.\@e.\@}}
\newcommand\B{\mathrm{B}}
\newcommand\BV{\mathrm{BV}}
\newcommand\C{\mathrm{C}}
\renewcommand\c{\mathrm{c}}
\newcommand*{\coleq}{\mathrel{\vcenter{\baselineskip0.5ex\lineskiplimit0pt\hbox{\normalsize.}\hbox{\normalsize.}}}=}
\newcommand*{\colequiv}{\mathrel{\vcenter{\baselineskip0.5ex\lineskiplimit0pt\hbox{\normalsize.}\hbox{\normalsize.}}}\equiv}
\newcommand\D{\mathrm{D}}
\renewcommand\d{\mathrm{d}}
\renewcommand\div{\mathrm{div}}
\newcommand\dx{{\,\d x}}
\newcommand\ecke{\mathop{\hbox{\vrule height 7pt width .3pt depth 0pt \vrule height .3pt width 5pt depth 0pt}}\nolimits}
\newcommand\eps{\varepsilon}

\newcommand\ip{\boldsymbol{\cdot}}
\newcommand\ka{\textup{(}}
\newcommand\kz{\textup{)}}
\renewcommand\L{\mathrm{L}}
\newcommand\LN{\mathcal{L}^N}
\newcommand\loc{\mathrm{loc}}
\renewcommand\H{{\mathcal{H}}}
\newcommand\N{{\mathds{N}}}
\renewcommand\P{\mathrm{P}}
\newcommand\p{\varphi}
\newcommand\qq{\qquad}
\newcommand\R{{\mathds{R}}}
\newcommand{\restr}{\,\rule[-1ex]{.1ex}{2.2ex}\,\rule[-.9ex]{0ex}{1ex}}
\newcommand\ts{\textstyle}
\DeclareMathOperator\TV{TV}
\newcommand\W{\mathrm{W}}

\newcommand{\cupdot}{\ensuremath{\,\mathaccent\cdot\cup\,}}
\newcommand{\WPhione}{\ol{\Phi}}
\newcommand{\WPhitwo}{\widetilde\Phi}
\newcommand{\ol}{\overline}
\newcommand\dt{{\,\d t}}
\newcommand\F{\mathrm{F}}
\newcommand\I{\mathrm{I}}

\DeclareMathOperator*{\dist}{dist}

\newcommand{\inn}{\mathrm{T}^{\mathrm{int}}_{\partial \Omega}}
\newcommand{\innk}{\mathrm{T}^{\mathrm{int}}_{\partial^\ast \Omega_k}}
\newcommand{\ext}{\mathrm{T}^{\mathrm{ext}}_{\partial \Omega}} 
\newcommand\mres{\mathop{\hbox{\vrule height 7pt width .3pt depth 0pt \vrule height .3pt width 5pt depth 0pt}}\nolimits}
\renewcommand\div{\mathrm{div}}
\newcommand{\Int}{\mathrm{Int}}

\newtheorem{thm}{Theorem}[section]
\newtheorem{assum}[thm]{Assumption}
\newtheorem{defi}[thm]{Definition}
\newtheorem{examp}[thm]{Example}
\newtheorem{lem}[thm]{Lemma}
\newtheorem{prop}[thm]{Proposition}

\newtheorem{rem}[thm]{Remark}
\newtheorem{cor}[thm]{Corollary}

\numberwithin{equation}{section}

\begin{document}

    \title{\vspace{-5ex}Lower semicontinuity and existence results\\for anisotropic $\TV$
		functionals with signed measure data}
	
	\author{
		Eleonora Ficola\footnote{Fachbereich Mathematik,
			Universit\"at Hamburg, Bundesstr. 55, 20146 Hamburg, Germany.\newline
			Email address: \href{mailto:eleonora.ficola@uni-hamburg.de}
			{\tt eleonora.ficola@uni-hamburg.de}.
		}
		\qq\qq
		Thomas Schmidt\footnote{Fachbereich Mathematik,
			Universit\"at Hamburg, Bundesstr. 55, 20146 Hamburg, Germany.\newline
			Email address: \href{mailto:thomas.schmidt.math@uni-hamburg.de}
			{\tt thomas.schmidt.math@uni-hamburg.de}.
			URL: \href{http://www.math.uni-hamburg.de/home/schmidt/}
			{\tt http:/\!/www.math.uni-hamburg.de/home/schmidt/}.
		}
	}

        \date{January 27, 2026}
        
	\maketitle

    \vspace{-4ex}

	\begin{abstract}
        \noindent We study the minimization of anisotropic total variation functionals with additional measure terms among functions of bounded variation subject to a Dirichlet boundary condition. More specifically, we identify and characterize certain isoperimetric conditions, which prove to be sharp assumptions on the signed measure data in connection with semicontinuity, existence, and relaxation results. Furthermore, we present a variety of examples which elucidate our assumptions and results.
	\end{abstract}

    \medskip


    \vspace{-1ex}

	\tableofcontents

    \pagebreak
	
	\section{Introduction}
	
	Throughout this paper, we consider $N\in\N$ and a bounded open set
	$\Omega\subseteq\R^N$ with Lipschitz boundary. Our main interest is in developing
	an existence theory for the minimization of functionals of type
	\begin{equation}\label{eq:P-pre-functional}
		\Phi[w]\coleq\int_\Omega\p(\,\cdot\,,\nabla w)\,\dx
		+\int_\Omega w^\ast\,\d\mu
		\qq\text{for }w\in\W^{1,1}_{u_0}(\Omega)\,,
	\end{equation}
	where the class $\W^{1,1}_{u_0}(\Omega)\coleq u_0\restr_\Omega{+}\W^{1,1}_0(\Omega)$
	with given $u_0\in\W^{1,1}(\R^N)$ specifies a Dirichlet boundary
	condition. The first term in \eqref{eq:P-pre-functional} is an anisotropic total
	variation with given integrand
	$\p\colon\overline\Omega\times\R^N\to{[0,\infty)}$ positively homogeneous of
	degree $1$ in its second argument $\xi\in\R^N$, while the second term in
	\eqref{eq:P-pre-functional} involves a given finite signed Radon measure $\mu$
	on $\Omega$ and the precise representative $w^\ast$ of $w$. On a formal level,
	where the non-differentiability of $\xi\mapsto\p(x,\xi)$ at $0$ is disregarded
	for an instant, the Euler equation for this minimization problem is the
	anisotropic $1$-Laplace equation
	\begin{equation}\label{eq:Euler}
		\div\big[\nabla_\xi\p(\,\cdot\,,\nabla u)\big]=\mu
		\qq\text{on }\Omega\,.
	\end{equation}
	
	In previous literature, minimization problems and equations of this type have
	been studied mostly in case of a weighted Lebesgue measure $\mu=H\LN$, in which
	the function $H\in\L^1(\Omega)$ can replace $\mu$ on the right-hand side of
	\eqref{eq:Euler}. In this function case, the existence theory in the natural
	setting of the space $\BV(\Omega)$ is nowadays mostly understood: Minimizers of
	$\Phi$ can be obtained from the direct method, while solutions to
	\eqref{eq:Euler} are typically defined via Anzellotti's pairing
	\cite{Anzellotti83a} and can be produced via convex duality or $p$-Laplacian
	approximation; among numerous contributions to the topic compare, for instance,
	\cite{CicTro03,Demengel04,MerSegTro08,MerSegTro09} for $\p(x,\xi)=|\xi|$ and
	either $H\in\L^N(\Omega)$, certain $H\in\W^{-1,\infty}(\Omega)$, or
	$H\in\L^1(\Omega)$, and \cite{AmaBel94,Moll05,BecSch15,Mazon16,JerMorNac18}
        for general $\p$ and $H\equiv0$. Closely related is the case with the first
        term in \eqref{eq:P-pre-functional} replaced by the non-parametric area integral
	$\int_\Omega\sqrt{1{+}|\nabla w|^2}\dx$ and correspondingly \eqref{eq:Euler}
	replaced by the non-parametric prescribed-mean-curvature (PMC) equation
	$\div\frac{\nabla u}{\sqrt{1{+}|\nabla u|^2}}=\mu$ on $\Omega$. This case
	with $\mu=H\LN$, $H\in\L^N(\Omega)$ is in fact even more classical and has been
	studied already in
	\cite{Miranda74,Giaquinta74a,Giaquinta74b,Gerhardt74,Giusti78a}. Moreover, in both
	the (anisotropic) total variation case and the area case, also right-hand side
	functions with more general $u$-dependent structure have been considered in some
	of the references already mentioned.
	
	In the case of the area, measures $\mu$ on the right-hand side have been
	called mean curvature measures and have been considered in a few instances at
	least: A variational approach to $\BV$ solutions is discussed already by Ziemer
	\cite{Ziemer95} (but under restrictive assumptions on $\mu$, which rule out
	interaction between the different terms of the functional), while
	Dai--Trudinger--Wang \cite{DaiTruWan12} and Dai--Wang--Zhou \cite{DaiWanZho15}
	develop an existence theory for the PMC equation on basis of a priori estimates.
	Recently, a general variational existence theory in the different setting of the
	\emph{parametric} area functional with measures $\mu$ has been developed by the
	second author \cite{Schmidt25}, and around the same time Leonardi--Comi
	\cite{LeoCom24} have obtained a variety of existence, duality, and comparison
	results for the \emph{non-parametric} area functional and the PMC equation with
	measures $\mu$. Here, some technical devices for both the last-mentioned papers
	(choices of representatives, refined Anzellotti pairing, duality-based notions
	of solutions) have been coined previously in \cite{CarLeaPas86,CarLeaPas87,CarDalLeaPas88,SchSch15,SchSch16,SchSch18,Treinov21}, mostly in connection with
	obstacle problems. Anyway, while the area case does
	\emph{not} fall into the framework we consider here (due the
	non-homogeneity of the area integrand $\xi\mapsto\sqrt{1{+}|\xi|^2}$), in our
	forthcoming paper \cite{FicSch25} we actually develop a similar existence
	theory even for general non-homogeneous integrands of linear growth (which then
	include the area as a special case) and for comparably general measures $\mu$.
	
	In this paper, however, we stick to anisotropic total variations with
	homogeneous $\p$. Moreover, we do \emph{not} directly investigate the equation
	\eqref{eq:Euler}, but rather attempt to study the minimization problem on the level of semicontinuity of functionals and existence of minimizers. In
	fact, due to the lack of compactness in $\W^{1,1}(\Omega)$ we work in the
	natural setting of the space $\BV(\Omega)$, where in view of the failure of
	weak-$\ast$ closedness of Dirichlet classes the boundary condition cannot be
	kept in its original sense. These effects are coped with in well-known
	manner by extending the anisotropic total variation term in
	\eqref{eq:P-pre-functional} to a functional $\TV_\p^{u_0}$ on $\BV(\Omega)$
	which includes boundary penalization (see \eqref{eq:TV-u0} for the precise
	formula). The full functional \eqref{eq:P-pre-functional} is then extended by
	setting
	\begin{equation}\label{eq:P-functional-intro}
		\widehat\Phi[w]
		\coleq\TV_\p^{u_0}[w]+\int_\Omega w^-\,\d\mu_+-\int_\Omega w^+\,\d\mu_-
		\qq\text{for }w\in\BV(\Omega)\,,
	\end{equation}
	where $\mu=\mu_+{-}\mu_-$ is the Jordan decomposition of $\mu$, while $w^+$ and
	$w^-$ denote $\H^{N-1}$-\ae{} defined approximate upper and
	lower limits of $w$ on $\Omega$ (see Section \ref{sec:prelim} for the precise
	definitions). We remark that in general it is only $w_1^-{+}w_2^-
	\le(w_1{+}w_2)^-\le(w_1{+}w_2)^+\leq w_1^+{+}w_2^+$, so that
	the terms $\int_\Omega w^-\,\d\mu_+{-}\int_\Omega w^+\,\d\mu_-$ in
	\eqref{eq:P-functional-intro} are no longer linear, but merely positively
	homogeneous of degree $1$ and concave in $w\in\BV(\Omega)$.
	
	We now aim at  a semicontinuity and existence theory for
	$\widehat\Phi$ under standard convexity and continuity assumptions on $\p$ and
	for $\mu$ \emph{not} necessarily absolutely continuous with respect to $\LN$.
	A central case of interest, which will also be supported by several examples,
	is when (a part of) the measure $\mu$ is $(N{-}1)$-dimensional, and then the
	decomposition $\mu=\mu_+{-}\mu_-$ and the choice of the representatives $w^\pm$
	in \eqref{eq:P-functional-intro} are essentially inevitable in obtaining
	reasonable lower semicontinuity results for $\widehat\Phi$ on $\BV(\Omega)$.
	Indeed, this becomes apparent when considering smooth approximations $w_k$,
	which converge to $w\in\BV(\Omega)\setminus\W^{1,1}(\Omega)$ strictly in
	$\BV(\Omega)$ and come either from above or below. Then on a pointwise
	$\H^{N-1}$-\ae{} level one can hope at best for $w_k\to w^+$ (when coming from
	above) or $w_k\to w^-$ (when coming from below), and in view of
	$\mu(\{w^+\neq w^-\})>0$ one can expect lower semicontinuity only when using as
	above $\min\{w^-,w^+\}=w^-$ in the $\mu_+$-term and (due to the negative sign in
	front) $\max\{w^-,w^+\}=w^+$ in the $\mu_-$-term. However, for obtaining
	existence results one needs lower semicontinuity of $\widehat\Phi$ on
	$\BV(\Omega)$ \emph{not} w.\@r.\@t.\@ strict convergence but rather w.\@r.\@t.\@
	weaker convergence such as $\L^1(\Omega)$-convergence, and indeed right this
	$\L^1(\Omega)$ lower semicontinuity is obtained in our first main result
	(Theorem \ref{thm:lsc}) under a certain isoperimetric condition (IC) to be
	discussed in more detail below. We stress, in any case, that our
	$\L^1(\Omega)$ lower semicontinuity --- in contrast to the much simpler strict
	lower semicontinuity mentioned before --- does not apply separately to the
	different terms in \eqref{eq:P-functional-intro}, but rather holds only for the
	full functional in \eqref{eq:P-functional-intro} and allows for some controlled
	interaction, governed by the IC, between the total variation and measure terms. Coping
	with such interaction effects, which are partially present already in \cite{CarLeaPas85,CarLeaPas86,CarLeaPas87,Pallara91} and have been recorded more explicitly for the area case in \cite{Schmidt25,LeoCom24}, is one of the main points in our semicontinuity proofs. 
    Despite the above-mentioned concavity of the $\mu_\pm$-part of the functional, one may clearly wonder whether the full functional $\widehat\Phi$ remains convex and whether this is a decisive background reason for the validity of semicontinuity. However, while we believe that our IC (at least in the signed-measure case of this introduction) implies convexity of $\widehat\Phi$, it seems to us that this is \emph{not} a point for semicontinuity. In fact, in the later Example \ref{exp:failure-lsc} we exhibit a basic case, in which one can still check convexity of $\widehat\Phi$, but $\widehat\Phi$ is \emph{not} anymore $\L^1(\Omega)$ lower semicontinuous.

	Our second main result (Theorem \ref{thm:exist}) is an existence result for
	minimizers of $\widehat\Phi$ in $\BV(\Omega)$ and is obtained by the standard
	direct method in the calculus of variations. Indeed, also this result builds on
	the IC already mentioned, which is needed not solely for lower semicontinuity of
	$\widehat\Phi$, but is essentially inevitable for coercivity of $\widehat\Phi$
	as well. Indeed, the connection with coercivity is valid already in the function
	case $\mu=H\LN$ and is presumably the most classical reason for requiring ICs in
	the literature. As an interesting subtlety, for boundary data
	$u_0\in\L^\infty(\R^N)$, we can extend our existence result to a certain
	limit case of the ICs, while, for $u_0\notin\L^\infty(\R^N)$, we demonstrate with
	Example \ref{exp:non-exist} that existence may indeed fail in this case.
	
	Our third main result (Theorem \ref{thm:recovery}) consistently completes the picture by identifying the functional $\widehat\Phi$ in \eqref{eq:P-functional-intro} as a natural extension of the initial functional $\Phi$ in \eqref{eq:P-pre-functional}. More precisely, the result provides, for every $u\in\BV(\Omega)$, a suitable recovery sequence from $\W^{1,1}_{u_0}(\Omega)$.
	
	Finally, we wish to shed some more light on the IC, which is our central
	hypothesis in the tradition of similar assumptions in \cite{BomGiu73,Giaquinta74a,Giaquinta74b,Gerhardt74,DaiTruWan12,Schmidt25,LeoCom24},
	for instance (see also the discussion after Definition \ref{defi:IC} for a more
	extensive list of references). Indeed, our formulation essentially takes the
	form
	\begin{equation}\label{eq:IC-intro}
		|\mu(A)|\leq\P_\p(A)
		\qq\qq\text{for all measurable }A\Subset\Omega
	\end{equation}
	with the $\p$-anisotropic perimeter $\P_\p(A)$ of $A$, where for the purposes of
	this introduction we have slightly simplified by assuming that $\p(x,\xi)$ is
	even in $\xi$ and by disregarding precise choices of representative for $A$. Let
	us point out two features of \eqref{eq:IC-intro}: On one hand we include the
	general anisotropy $\p$, while previously only the case $\p(x,\xi)=|\xi|$ with the standard perimeter
	$\P(A)$ on the right-hand side of \eqref{eq:IC-intro} was considered. On the
	other hand our condition is the most general one on the signed measure $\mu$ in
	the sense that we can deal with merely $|\mu(A)|$ on the left-hand side, while
	in previous literature ICs were often imposed on functions or non-negative
	measures and also the recent related results in \cite{Schmidt25,LeoCom24}
	(partially) require separate control on $\mu_+(A)$ and $\mu_-(A)$. Though the precise
	form of the left-hand side of \eqref{eq:IC-intro} is a technical
	detail, we believe that this brings an improvement of some interest even in the
	standard isotropic case $\p(x,\xi)=|\xi|$, and we will underline the
	corresponding gain in generality by the latter Example \ref{exp:signed-IC}.
	
	The role of our IC as a reasonable assumption is further supported, beside its
	necessity for lower semicontinuity and coercivity (cf. Proposition \ref{prop:anis_CN_coercivity}), by the fact that
	under suitable smoothness assumptions it is a basic necessary criterion for the
	existence of a solution $u$ to the Euler equation \eqref{eq:Euler}. In the
	isotropic total variation and area cases this observation essentially goes back
	to \cite{Giusti78a}, and here we take the occasion to briefly point out the
	analogous necessity of our IC \eqref{eq:IC-intro} in general. Indeed, for a
	$\C^2$ solution $u$ of \eqref{eq:Euler} with $\nabla u\neq0$ on $\Omega$, a
	$\C^1$ domain $A\Subset\Omega$ with inward unit normal $\nu_A$, and for $\p(x,\xi)$ convex in $\xi$ and
	differentiable in $\xi\neq0$, the divergence theorem yields
	\[
	{\pm}\mu(A)
	=\int_{\partial A}\nabla_\xi\p(\,\cdot\,,\nabla u)\ip({\mp}\nu_A)\,\d\H^{N-1}
	\leq\int_{\partial A}\p(\,\cdot\,,{\mp}\nu_A)\,\d\H^{N-1}
	=\P_\p(A)\,,
	\]
	where we exploited that
$\nabla_\xi\p(x,\xi)\ip\nu=\p(x,\xi)+\nabla_\xi\p(x,\xi)\ip(\nu{-}\xi)\leq\p(x,\nu)$
	for all $x\in\overline\Omega$, $\xi,\nu\in\R^N{\setminus}\{0\}$ by homogeneity
	and convexity of $\p(x,\xi)$ in $\xi$. Thus, \eqref{eq:IC-intro} is indeed valid
	in this case. A slightly wider perspective on this reasoning will in fact be
	opened up with the alternative characterizations of measures $\mu$ with IC in the
	later Theorems \ref{thm:anis_equivalence_measure_TOGETHER} and  \ref{thm:anis_equivalence_gen_IC_singular}. Indeed, one equivalent characterization of \eqref{eq:IC-intro} will
	be the representation $\mu=\div\,\sigma$ for some
	$\sigma\in\L^\infty(\Omega,\R^N)$ such that $\p^\circ(\,\cdot\,,\sigma)\le1$ \ae{} on $\Omega$ with the polar $\p^\circ$ of $\p$, and this entails the necessity
	of \eqref{eq:IC-intro} in larger generality. Here, the last-mentioned inequality
	may also be read as the sub-unit condition
	$\|\sigma\|_{\p^\circ;\L^\infty(\Omega)}\le1$ for the $\p^\circ$-anisotropic
	$\L^\infty$ norm
	$\|\sigma\|_{\p^\circ;\L^\infty(\Omega)}\coleq\|\p^\circ(\,\cdot\,,\sigma)\|_{\L^\infty(\Omega)}$
	and thus this gives an interpretation of the IC \eqref{eq:IC-intro} as the
	sub-unit condition $\|\mu\|_{\p^\circ;\W^{-1,\infty}(\Omega)}\le1$ for the corresponding
	$\p^\circ$-anisotropic $\W^{-1,\infty}$ norm in the version homogeneous of
	degree ${-}1$.
	
	To complete this introduction let us add brief comments on the reasons for
	actually restricting ourselves to measures $\mu$ with IC \eqref{eq:IC-intro}
	rather than allowing ideally all sub-unit (in the sense just mentioned)
	distributions $\kappa\in\W^{-1,\infty}(\Omega)\cong(\W^{1,1}_0(\Omega))^\ast$, which
	form the natural class of right-hand sides for $\BV$ solutions of
	\eqref{eq:Euler}. Indeed, it seems to us that, in order to generally access the
	case of such distributions by the direct method, one would need to explain
	evaluations $\llfloor\kappa\,;w{-}u_0\rrfloor$ of $\kappa$ for $w\in\BV(\Omega)$, and could then
	potentially proceed towards lower semicontinuity and existence results for
	\begin{equation}\label{eq:TV-T-intro}
		\TV_\p^{u_0}[w]+\llfloor \kappa\,;w{-}u_0\rrfloor
		\qq\qq\text{on }w\in\BV(\Omega)\,.
	\end{equation}
	However, an apparent obstacle is that one cannot extend from
	$\langle\kappa\,;w{-}u_0\rangle$ for $w\in\W^{1,1}_{u_0}(\Omega)$ to
	$\llfloor\kappa\,;w{-}u_0\rrfloor$ for $w\in\BV(\Omega)$ by continuity, since, for strict approximations
	$w_k\in\W^{1,1}_{u_0}(\Omega)$ of $w\in\BV(\Omega)\setminus\W^{1,1}(\Omega)$,
	the limit value $\lim_{k\to\infty}\langle\kappa\,;w_k{-}u_0\rangle$ may depend on
	the choice of $w_k$. This problem persists already in the case of measures,
	where it was discussed above and overcome essentially by employing the
	representatives $w^\pm$. As an abstraction of this proceeding we consider using
	a strict-convergence lower semicontinuous relaxation rather than an extension by continuity. Therefore,
	mimicking our approach abstractly, for $\kappa\in\W^{-1,\infty}(\Omega)$ and
	$w\in\BV(\Omega)$, one could let
	\begin{equation}\label{eq:dist-on-BV}
		\llfloor\kappa\,;w{-}u_0\rrfloor\coleq\inf\Big\{\liminf_{k\to\infty}\,\langle\kappa\,;w_k{-}u_0\rangle\,:\,
		\W^{1,1}_{u_0}(\Omega)\ni w_k\to w\text{ in }\L^1(\Omega)\,,\,\TV^{u_0}[w_k]\to\TV^{u_0}[w]\Big\}\,,
	\end{equation}
	where the isotropic version $\TV^{u_0}$ of $\TV_\p^{u_0}$ was used in expressing an $u_0$-extended strict convergence. But then one would still face the principal difficulty of
	proving $\L^1(\Omega)$ lower semicontinuity of the full functional
	\eqref{eq:TV-T-intro} with possible interaction between its two terms. It is
	this decisive property which we do not know how to approach for general sub-unit
	$\kappa\in\W^{-1,\infty}(\Omega)$, since presently our arguments rely on the
	possibility of decomposing $\mu=\mu_+{-}\mu_-$ and exploiting the concrete form
	of the measure terms through additivity and continuity properties at several
	points. All in all, at this stage we leave it as an issue for further study to
	explore if one can possibly turn \eqref{eq:dist-on-BV} into more concrete
	representation formulas and can make progress in the variational existence
	theory also for other reasonable classes of distributions
	$\kappa\in\W^{-1,\infty}(\Omega)$ than measures.
	
	\medskip
	
	The plan of this paper is now as follows. We first collect preliminaries in
	Section \ref{sec:prelim} before stating our main results along with several
	illustrative examples in Section \ref{sec:statements}. In Section
	\ref{sec:charact} we provide various equivalent reformulations of our ICs (which
	so far were only touched upon rather briefly, but may also be of some interest),
	and in Section \ref{sec:exp} we complete the discussion of various examples
	stated before. In Section \ref{sec:par-lsc} we extend the parametric semicontinuity results of \cite{Schmidt25} to general signed-measure cases. Finally, in Sections \ref{sec:exist} and
	\ref{sec:recovery} we provide a few additional statements, but mostly focus on
	the proofs of the three main results.

    \medskip

    \noindent\textbf{Acknowledgements.} The authors are grateful to G.E.~Comi and G.P.~Leonardi for a discussion on the relevance of signed ICs of type \eqref{eq:IC-intro}. The figures in this article have been created in the vector graphics language `Asymptote'.

	\section{Preliminaries} \label{sec:prelim}
	
	\subsection{Generalities}
	
	Throughout the paper we work in the Euclidean space $\R^N$ with arbitrary $N \in \N$, unless differently specified. 
	With $\LN$ we denote the $N$-dimensional Lebesgue measure, and we say that a property holds \ae{} if it holds except for a set of zero $\LN$ measure. In accordance we call a set measurable or negligible if it is measurable or negligible with respect to $\LN$, and we write $|A| \coleq \LN(A)$ for the volume of any measurable $A \subseteq \R^N$. The $M$-dimensional Hausdorff measure on $\R^N$ for $M \in [0,N]$ is denoted by $\H^M$, though in the following we will primarily work in the codimension 1 case $M=N-1$. We write $\B_r(x)$ for the open ball centered at $x \in \R^N$ and of radius $r >0$, and we abbreviate $\B_r\coleq\B_r(0)$ whenever the ball is centered at the origin. The volume of the unit ball in $\R^N$ is denoted by $\omega_N$. The Euclidean distances between a point $a$ and set $A$, $B$ in $\R^N$ are
	$\dist(a,B) \coleq \inf_{b \in B} |a-b|$ and $\dist(A,B) \coleq \inf_{a \in A} \dist(a,B)$. 
	With $\ol{A}$, $\Int(A)$ we refer to the closure and the interior of a set $A$, respectively, and $\1_A$ denotes 
    the indicator function of $A$. The symmetric difference of two sets $A$, $B$ is $A \triangle B$,
	whereas we use $A \Subset B$ whenever $A$ is compactly contained in $B$.

    \medskip
    
    We recall that a Lipschitz function $u\colon U \to \R$ on open $U \subseteq \R^N$ is \ae{} classically differentiable by Rademacher's theorem (see e.\@g.\@ \cite[Theorem 7.8]{Maggi12}), and we then denote the \ae{} defined gradient by $\nabla u$. Moreover, we recall (see e.\@g.\@ \cite[Theorem 18.1]{Maggi12}):
    
	\begin{thm}[{coarea formula for Lipschitz functions}]
		\label{thm:coarea_Lip}
		Given a Lipschitz function $u\colon U \to \R$ on open $U \subseteq \R^N$, it holds
		\[
          \int_{A}|\nabla u|\,\dx 
          = \int_{- \infty}^{\infty} \H^{N-1} \left( A \cap \left\{ u = t \right\} \right) \dt
        \]
		for every Borel set $A\subseteq U$. 
	\end{thm}

    Our conventions for general measures mostly follow \cite[Sections 1.1, 1.3]{AFP00}, and specifically we use non-negative, signed, and $\R^M$-valued Radon measures on locally compact $U \subseteq \R^N$ in the sense specified there. In particular, $\mu \mres S$ is the restriction of a measure $\mu$ to a $\mu$-measurable set $S \subseteq \R^N$, defined by $\mu \mres S (A) \coleq \mu(S \cap A)$ for all $\mu$-measurable $A$, and $f\mu$ denotes the weighting of a measure $\mu$ with a $\mu$-integrable density $f$, defined by $f\mu(A)\coleq\int_Af\,\d\mu$ for all $\mu$-measurable $A$. According to the Jordan decomposition theorem, for any signed Radon measure $\mu$ on $U$, there exists a unique pair of non-negative, mutually singular measures $\mu_+$, $\mu_-$ on $U$ (the positive and negative parts of $\mu$) such that $\mu=\mu_+ - \mu_-$ and $|\mu|=\mu_+ + \mu_-$ for the total variation measure $|\mu|$ of $\mu$. Finally, if, for a non-negative Radon measure $\mu$ on $U$ and an $\R^M$-valued Radon measure $\nu$ on $U$, all $\mu$-negligible sets are also $|\nu|$-negligible, the Radon-Nikod\'ym theorem asserts the existence of a $\mu$-integrable $\R^M$-valued density $f$ such that $\nu=f\mu$. This density is known as Radon-Nikod\'ym density, is $\mu$-\ae{} uniquely determined, and is sometimes denoted by $\frac{\d\nu}{\d\mu}$.    

    \medskip

    For measurable $A \subseteq \R^N$ and $\theta \in {[0,1]}$, we introduce the sets
	\[
      A^{\theta}\coleq\left\lbrace x \in \R^N : \  \lim_{r \to 0} \frac{|\B_r(x) \cap A|}{|\B_r|} = \theta \right\rbrace \qq \text{ and } \qq 
      A^+\coleq\left\lbrace x \in \R^N : \ \limsup_{r \to 0} \frac{|\B_r(x) \cap A|}{|\B_r|} > 0 \right\rbrace\,,
    \]
	respectively, of density-$\theta$ points and positive-upper-density points of $A$. We refer to $A^1$ as measure-theoretic interior and to $A^+$ as measure-theoretic closure of $A$. 
 
	\begin{rem}\label{rem:inclusions_positive_density}
		If $A \subseteq \R^N$ is measurable and $U \subseteq \R^N$ is open, then we have $A^\theta \cap U \subseteq (A \cap U)^\theta$ for any $\theta \in {[0,1]}$ and $A^+ \cap U \subseteq (A \cap U)^+$. 
	\end{rem}	

    \subsection[\texorpdfstring{Basic $\BV$ theory}{}]{\boldmath Basic $\BV$ theory}
 
    We employ standard notation (see again \cite{AFP00}) for the function spaces $\L^p$, $\W^{k,p}_{(0)}$, $\C^{k}$, $\C^{k}_{\c}$, $\BV$ and their localized versions. Specifically, the space $\BV_\loc(U)$ over open $U\subseteq\R^N$ is the collection of all $u\in\L^1_\loc(U)$ such that the distributional gradient of $u$ exists as an $\R^N$-valued Radon measure $\D u$. For the total variation of a measurable function $u\colon U\to\R$ in a Borel set $B\subseteq U$ we sometimes use the notation $\TV(u,B)\coleq|\D u|(B)$ if $u$ is $\BV_\loc$ in some open neighborhood of $B$, otherwise we understand $\TV(u,B)\coleq\infty$. The space $\BV(U)$ of functions of bounded variation contains all $u\in\BV_\loc(U)$, for which the norm $\|u\|_{\BV(U)} \coleq \|u\|_{\L^1(U)} + |\D u|(U)$ is finite, and is a Banach space with this norm. Since norm-convergence in $\BV(U)$ is too strong for many purposes, we often deal with strict convergence of a sequence $(u_k)_k$ in $\BV(U)$ to a limit $u\in\BV(U)$, which means both $\lim_{k\to\infty}\|u_k-u\|_{\L^1(U)}=0$ and $\lim_{k\to\infty}|\D u_k|(U)=|\D u|(U)$. 
    We follow the convention of indicating positive and negative parts $u_\pm\coleq\max\{\pm u,0\}$ of a function $u$ by \emph{lower} indices ${}_\pm$, which need to be distinguished from \emph{upper} indices ${}^\pm$. In fact, given $U \subseteq \R^N$ and a measurable $u \colon U \to \R^N$, we denote by $u^+(x)$ and $u^-(x)$ the approximate upper and lower limit of $u$ at $x \in U$, respectively, that is,
    \[
      u^+(x)\coleq\sup\big\{t\in\R:x\in\{u>t\}^+\big\}\,,
      \qq
      u^-(x)\coleq\sup\big\{t\in\R:x\in\{u>t\}^1\big\}\,,
    \]
    and additionally we set  
    $u^\ast(x)\coleq(u^+(x)+u^-(x))/2$\,. In particular, we make wide use of the resulting equalities $(\1_A)^+ = \1_{A^+}$ and $(\1_A)^- = \1_{A^1}$ for any measurable $A\subseteq\R^N$. For functions $u \in \L^{1}_{\loc}(U)$, we observe $u^+=u^-=\tilde{u}$ in Lebesgue points, while $u^+ > u^-$ are the two jump values in jump points, compare \cite[Section 3.6]{AFP00}. 
    Moreover, by the Federer-Volpert theorem \cite[Theorem 3.78]{AFP00}, for $u \in \BV_\loc(U)$, the representatives 
    $u^+,u^-,u^\ast$ coincide $\H^{N-1}$-\ae{} outside the jump set of $u$. Specifically, for $u\in\W^{1,1}_\loc(U)$, the coincidence of the representatives holds $\H^{N-1}$-\ae{} on all of $U$. We also put on record that generally the estimates
    \begin{equation}\label{eq:sub-super-add-representatives}
      u^-+v^-\leq(u+v)^-\leq(u+v)^+\leq u^++v^+
      \qq\text{hold in }U\,,
    \end{equation}
    while, for $u,v\in\BV_{\loc}(U)$, we have the $\H^{N-1}$-\ae \ equality $(u+v)^\ast=u^\ast + v^\ast$. Some more properties of the representatives follow.

	\begin{rem}\label{rem:equivalence_sets_wrt_mu_ext}
		If $\mu$ is a non-negative Radon measure on $\R^N$ such that $\mu$ vanishes on $\H^{N-1}$-negligible sets, then, for every $u \in \BV(\R^N)$, there hold
		\begin{equation} \label{FFF}
			\left\{ u^+ > t \right\} 
			= \left\{ u > t \right\}^+ 
			\quad \text{and} \quad \left\{ u^- > t \right\} 
			= \left\{ u > t \right\}^1\quad\text{up to } \mu\text{-negligible sets,}\quad\text{for }\mathcal{L}^1 \mbox{-\ae{} } t \in \R\,.
		\end{equation}
		\begin{proof}
			Given $u \in \BV(\R^N)$, for every $t\in\R$, we have 
			$\left\{ u^+ > t \right\} \subseteq \left\{ u > t \right\} ^+ \subseteq \left\{ u^+ \geq t \right\}$ and 
			$\left\{ u^- > t \right\} \subseteq \left\{ u > t \right\} ^1 \subseteq \left\{ u^- \geq t \right\}$ up to $\H^{N-1}$-negligible sets and thus also up to $\mu$-negligible sets. Moreover, a standard argument ensures
			\[
			\mu \left( \{ u^\pm  \geq t \} \setminus \{  u^\pm > t \} \right) 
			=\mu \left( \{ u^\pm = t \} \right) =0
			\]
			for all but countably many $t\in\R$, in particular for $\mathcal{L}^1$-\ae{} $t\in\R$. In conclusion we arrive at \eqref{FFF}. 
		\end{proof}
	\end{rem}

    From the $\H^{N-1}$-\ae{} characterization of $u^+$ and $u^-$ as jump values, one straightforwardly reads off:

	\begin{lem} \label{lem:dec_upper_limit}
		For each $u \in \BV(U)$ on open $U \subseteq\R^N$, we have a precise decomposition into positive and negative parts in the sense that
		\begin{enumerate}[{\rm(i)}]
			\item $u^+ = (u_+)^+ - (u_-)^-
              \quad\text{holds }\H^{N-1}$-\ae{} in $U$\,,
              \label{item:dec_upper_limit_i}
			\item $u^- = (u_+)^- - (u_-)^+
              \quad\text{holds }\H^{N-1}$-\ae{} in $U$\,.
              \label{item:dec_upper_limit_ii}
		\end{enumerate}
	\end{lem}
	
	The next result deals with $\H^{N-1}$-\ae{} convergence properties of strictly convergent sequences in $\BV$. The general case has been established in \cite[Theorem 3.2]{Lahti17}, while particular cases such as the one of strongly convergent sequences in $\W^{1,1}$ have been treated already in \cite[Sections 4 and 10]{FedZie72}.
	
	\begin{thm}[{pointwise convergence of $\BV$ functions}] 
		\label{strict_Hausdorff_repr}
		Let $U \subseteq \R^N$ be an open set and consider a sequence $(u_k)_k$ in $\BV(U)$ and $u\in\BV(U)$ such that $u_k \to u$ strictly in $\BV(U)$. Then, there exists a subsequence $(u_{k_\ell})_\ell$ such that there holds
		\[
		u^-(x) \leq \liminf_{\ell \to \infty} u_{k_\ell}^-(x) \leq  \limsup_{\ell \to \infty} u_{k_\ell}^+(x) \leq u^+(x)
		\qq\text{for }\H^{N-1}\text{-\ae{} }x \in U\,.
		\]
		In particular, in case of\/ $u\in\W^{1,1}(U)$ this implies $\lim_{\ell\to\infty}u_{k_\ell}^\ast(x)=u^\ast(x)$ for $\H^{N-1}$-\ae{} $x\in U$.
	\end{thm}
   
    \bigskip

    \subsection{Sets of finite perimeter and trace theory}

    We say that a measurable set $E \subseteq \R^N$ has locally finite perimeter in open $U \subseteq \R^N$ if $\1_E \in \BV_{\loc}(U)$ holds. The perimeter of measurable $E \subseteq \R^N$ in a Borel set $B \subseteq \R^N$ is given by $\P(E,B) \coleq |\D \1_E|(B)$ whenever $E$ has locally finite perimeter in some open neighborhood of $B$, otherwise we understand $\P(E,B) \coleq \infty$. As usual we abbreviate $\P(E)\coleq\P(E,\R^N)$. A set $E \subseteq \R^N$ has finite perimeter in open $U \subseteq \R^N$ if we have $\1_E \in \BV_{\loc}(U)$ and $\P(E,U) < \infty$. Whenever no domain $U$ is specified, a set of (locally) finite perimeter is intended to be of (locally) finite perimeter in $\R^N$.

    If $E\subseteq\R^N$ has locally finite perimeter in an open $U\subseteq\R^N$, the Radon-Nikod\'ym theorem yields
    \[
      \D\1_E = \nu_E |\D\1_E|
      \qq\text{as measures on }U\,,
    \]
    where the Radon-Nikod\'ym density $\nu_E\coleq\frac{\d \D\1_E}{\d |\D\1_E|}\in\L^\infty(U;|\D\1_E|)$ is characterized by the Lebesgue-Besicovitch differentiation theorem (see \cite[Theorem 2.22]{AFP00} or \cite[Theorem 5.8]{Maggi12}) as
	\[
      \nu_E(x)
      =\lim_{r \to 0} \frac{\D\1_E(\B_r(x))}{|\D\1_E|(\B_r(x))}
      \qq\text{for }|\D\1_E|\text{-a.e. } x \in U\,.
    \]
    The reduced boundary $U\cap\partial^\ast\!E$ of $E$ in $U$ is defined as the set of points $x\in U$ such that $|\D\1_E|(\B_r(x))>0$ holds for all $r>0$ and such that $\nu_E(x)$ exists in the sense of the preceding limit with $|\nu_E(x)|=1$ (compare e.\@g.\@ \cite[Definition 3.54]{AFP00}, \cite[Definition 5.4]{EvGar15}, or \cite[Section 15]{Maggi12}). The unit vector $\nu_E(x)$ is then called the generalized inward unit normal to $E$ at $x\in\partial^\ast\!E$. A fundamental geometric characterization of $|\D\1_E|$ and $\D\1_E$ in terms of $\partial^\ast\!E$ and $\nu_E$ can be retrieved from \cite[Theorem 3.59]{AFP00} or \cite[Theorem 15.9]{Maggi12}, for instance, and is now partially restated as follows.
 
	\begin{thm}[{De Giorgi's structure theorem; partial statement}]
		\label{thm:DG}
		If $E$ is a set of locally finite perimeter in open $U\subseteq\R^N$, then	one has
		\[
          |\D \1_E|=\H^{N-1} \mres (U\cap\partial^\ast E)
          \qq\text{as non-negative measures on }U\,.
        \]
	\end{thm}
	
	As a consequence of Theorem \ref{thm:DG}, one further has $\D \1_E =  \nu_E |\D\1_E| = \nu_E \H^{N-1} \mres \partial^\ast E$ as $\R^N$-valued measures on $U$ and $\P(E,B)=|\D \1_E|(B)
	=\H^{N-1}(B\cap\partial^\ast\!E)$ 
	whenever $B$ is a Borel subset of $U$.

    \medskip
	
	For the detailed definition of traces of $\BV$ functions, we refer to \cite[Section 3]{AFP00}. We briefly recall that $u\in\BV(U)$ has an interior trace and an exterior trace on each oriented countably $\H^{N-1}$-rectifiable set $\Gamma\subset U$, see \cite[Theorem 3.77]{AFP00}. The traces are defined and are finite $\H^{N-1}$-\ae{} on $\Gamma$, depend linearly on $u$, and are here denoted by $\mathrm{T}^{\mathrm{int}}_{\Gamma} u$ and $\mathrm{T}^{\mathrm{ext}}_{\Gamma} u$, respectively. If the context allows it, we sometimes omit the trace symbol and simply write $u$ instead of its (interior or exterior) trace. The decisive  property of the traces is recorded in the following restatement of \cite[Theorem 3.84]{AFP00}.
 	
	\begin{thm}[{pasting $\BV$ functions across reduced boundaries}] \label{thm:dec_BV}
		Consider an open set $U \subseteq \R^N$ and $u, v \in \BV(U)$, a set $E$ of finite perimeter in $U$, with reduced boundary $U\cap\partial^\ast\!E$
		oriented by $\nu_E$. Then, for $w\coleq u \1_E + v \1_{U \setminus E}$, it holds\textup{:} 
		\[
          w \in \BV(U) \iff 
		   \int_{U\cap\partial^\ast\!E}  \left| \mathrm{T}^{\mathrm{int}}_{U \cap \partial^\ast\!E} u - \mathrm{T}^{\mathrm{ext}}_{U \cap\partial^\ast\!E} 
		   v 
		   \right|\,\d\H^{N-1} < \infty\,.
        \]
		Moreover, if $w \in \BV(U)$, it is
		\begin{equation} \label{eq:dec_BV}
			\D w = \D u \mres E^1 + \left(  \mathrm{T}^{\mathrm{int}}_{U \cap \partial^\ast\!E} u - \mathrm{T}^{\mathrm{ext}}_{U \cap \partial^\ast\!E} v 
			\right) \otimes \nu_E \H^{N-1} \mres (U \cap \partial^\ast E) + \D v \mres E^0\,. 
		\end{equation}
	\end{thm}

    If $\Omega \subseteq \R^N$ is open and bounded with Lipschitz boundary, each $u\in\BV(\Omega)$ extends to a function in $\BV(\R^N)$ with value $0$ outside $\Omega$, compare \cite[Theorem 3.87]{AFP00}.
    In this case, the trace $\mathrm{T}_{\partial\Omega}u$ of $u$ on the boundary $\partial\Omega$ is the interior trace of the extension when $\partial\Omega$ is oriented by the inward normal. For this type of trace, we recall (see e.\@g.\@ \cite[Theorem 3.88]{AFP00}):

    \begin{thm}[continuity of the boundary trace operator] \label{thm:trace_cont}
    Consider an open and bounded $\Omega \subseteq \mathbb{R}^N$ with Lipschitz boundary. Then we have 
    $\mathrm{T}_{\partial\Omega}u \in \L^1(\partial \Omega; \H^{N-1})$ for each 
    $u \in \BV(\Omega)$. Moreover, the trace operator $u \mapsto \mathrm{T}_{\partial\Omega}u$ is continuous from $\BV(\Omega)$ with strict convergence to $\L^1(\partial \Omega;\H^{N-1})$ with norm convergence.
    \end{thm}

    Another useful result for our purposes is the isoperimetric inequality in the framework of sets of finite perimeter, as treated for instance in \cite[Theorem 14.1]{Maggi12}:
	
	\begin{thm}[{isoperimetric inequality}] \label{thm:isop_ineq}
        Given $E$ a measurable set in $\R^N$ with $0<|E|<\infty$, one has
		\begin{equation} \label{eq:isop_ineq}
			N \omega_N^\frac1N |E|^\frac{N-1}N\leq \P(E)\,,
		\end{equation}
		where equality holds if and only if\/ $|E \triangle \B_r(x)|=0$ for some $x \in \R^N$ and $r>0$. 
	\end{thm}

    The next result is a standard version of Poincar\'e's inequality. However, we cannot resist providing below the one-line estimation, which identifies the sharp constant at least in case the domain is a ball.

	\begin{thm}[Poincar\'{e} inequality]
		\label{thm:Poinc_BV_Lip}
      Let $\Omega \subseteq \mathbb{R}^N$ be open and bounded with Lipschitz boundary. Then, for any $u \in \BV(\Omega)$, it holds
		\begin{equation} \label{Poinc_BV}
			\|u\|_{\L^1(\Omega)} 
			\leq\frac{r}N\left( |\D u|(\Omega) + \int_{\partial \Omega} |u|\,\d\H^{N-1} \right)\,,
		\end{equation}
        where $r$ is the smallest possible radius of a ball that contains $\Omega$.
	\end{thm}

    \begin{proof}
      One may reduce to $0\le u\in\C^\infty_\c(\Omega)$, for which $\P(\{u>t\})=\H^{N-1}(\{u=t\})$ holds for $\mathcal{L}^1$-\ae{} $t>0$. Then one finds (with understanding $|E|^\frac{N-1}N\coleq0$ for negligible $E$ in case $N=1$)
      \[
        \int_\Omega u\dx
        =\int_0^\infty|\{u>t\}|\,\d t
        \le\big(\omega_Nr^N\big)^\frac1N\int_0^\infty|\{u>t\}|^\frac{N-1}N\,\d t
        \le\frac rN\int_0^\infty\H^{N-1}(\{u=t\})\,\d t
        =\frac rN\int_\Omega|\nabla u|\dx
      \]
      by the layer-cake formula, the isoperimetric inequality \eqref{eq:isop_ineq}, and the coarea formula of Theorem \ref{thm:coarea_Lip}.
    \end{proof}

    Specifically, Theorem \ref{thm:Poinc_BV_Lip} implies
		\begin{equation} \label{eq:Poinc_BV3}
			\|u\|_{\BV(\Omega)} 
			\leq C\left( |\D u|(\Omega) + \int_{\partial \Omega} |u|\,\d\H^{N-1} \right)
		\end{equation}
	for all $u\in\BV(\Omega)$ with $C\coleq\frac rN+1$ (for $r$ as in the theorem).
	
	\subsection{Anisotropic total variations and anisotropic perimeters}
	
	We now collect typical assumptions on the integrand $\p$ of an anisotropic
	total variation. Here, for convenience in working with auxiliary functionals
	over arbitrary open $U\subseteq\R^N$ or enlarged domains
	$\overline\Omega\subseteq\Omega'\subseteq\R^N$, we directly consider $\p$ as defined on all of $\R^N\times\R^N$ rather than
	on $\overline\Omega\times\R^N$. We remark, however, that it is indeed possible to
	extend a given $\p$ from $\overline\Omega\times\R^N$ to $\R^N\times\R^N$ such
	that the relevant properties are preserved.
	
	\begin{assum}[admissible integrands]\label{assum:phi}
		For a Borel function $\p\colon\R^N\times\R^N\to{[0,\infty)}$, we generally assume that $\xi\mapsto\p(x,\xi)$ is positively homogeneous of degree $1$, i.\@e.\@
		\[
		\p(x,t\xi)=t\p(x,\xi)\qq\text{for all }t\in{[0,\infty)}\text{ and }x,\xi\in\R^N\,.
		\]
		Moreover, we require that $\p$ is comparable to the Euclidean norm of\/ $\R^N$ \textup{(}in other words, comparable to the standard isotropic total variation integrand\textup{)} in the sense of
		\begin{equation}\label{eq:comp-p0}
			\alpha|\xi|\leq\p(x,\xi)\leq\beta|\xi|
			\qq\text{for all }x,\xi\in\R^N
		\end{equation}
		with fixed constants $0<\alpha\leq\beta<\infty$. In addition, we often
		require\textup{:}
		\begin{enumerate}[{\rm(a)}]
			\item that $\xi\mapsto\p(x,\xi)$ is convex for a.\@e.\@ $x\in\R^N$ and\/
			$(x,\xi)\mapsto\p(x,\xi)$ is lower semicontinuous,\label{case:conv-lsc}
			\item that $(x,\xi)\mapsto\p(x,\xi)$ is continuous.\label{case:cont}
		\end{enumerate}
	\end{assum}

    In the sequel, by Assumption \ref{assum:phi} we refer to the plain version     
    without any of \eqref{case:conv-lsc}, \eqref{case:cont}, while the addition of 
    \eqref{case:conv-lsc} or \eqref{case:cont} or both will be specified whenever needed.
 
	As a consequence of positive homogeneity and convexity, an admissible integrand $\p$ satisfies the triangle inequalities
	\begin{equation} \label{anis_eq:subadd_phi}
		\p(x,\xi+\tau) \leq \p(x,\xi) +\p(x,\tau)
		\quad\text{and}\quad
		\p(x,\xi-\tau) \geq \p(x,\xi) - \p(x,\tau)
        \qq\text{for all } x,\xi,\tau \in\R^N\,.
	\end{equation}
    For later usage we also put on record that in the special case of linearly dependent $\xi$ and $\tau$ the inequalities in \eqref{anis_eq:subadd_phi} remain valid when merely positive homogeneity and non-negativity (but not necessarily convexity) of $\xi\mapsto\p(x,\xi)$ are assumed. The reason is that \emph{on each $1$-dimensional subspace} of $\R^N$ each function $\xi\mapsto\p(x,\xi)$ is then nothing but a maximum of two linear functions and as such is automatically convex.

    \smallskip
	
	We also fix the following piece of notation.
	
	\begin{defi}[mirrored integrand]
		For an integrand $\p\colon\R^N\times\R^N\to\R$, we define the mirrored
		integrand $\widetilde\p\colon\R^N\times\R^N\to\R$ by setting
		\[
		\widetilde\p(x,\xi)\coleq\p(x,{-}\xi)
		\qq\text{for all }x,\xi\in\R^N\,.
		\]
	\end{defi}
	
	Clearly, in many basic cases the integrand  $\p$ is even in $\xi$ in the sense of $\widetilde\p=\p$, and then we need not distinguish between $\widetilde\p$ and $\p$ at all. However, our results can in fact be stated without an evenness assumption.

    \medskip
	
	Next we recall the definition of the polar function. Given $f\colon \R^N \to{[0,\infty)}$ with $f(\xi)>0=f(0)$ for all $\xi\in\R^N\setminus\{0\}$, the polar $f^\circ \colon \R^N \to{[0,\infty]}$ of $f$ is defined by 
    \[
      f^\circ(\xi^\ast)
	   = \sup_{\xi \in \R^N\setminus\{0\}} \frac{ \xi^* \ip \xi }{f(\xi)}\,.
    \]
	We record that the polar $f^\circ$ is positively homogeneous of degree $1$ and convex. In the following we are interested in the polar function of the anisotropic integrand $\p\colon \R^N \times \R^N \to {[0,\infty)}$ with respect to the second variable only. In fact, we impose Assumption \ref{assum:phi} and, for any fixed $x \in \R^N$, write $\p^\circ(x,\,\cdot\,)$ for the polar of the mapping $\xi \mapsto \p(x,\xi)$. Specifically, the definition of polar directly implies 
	$\p^\circ(x,\xi^*)\,\p(x,\xi) \geq \xi^* \ip \xi$ for all $\xi,\xi^\ast\in\R^N$. Moreover, the same bound returns 
	\begin{equation} \label{anis_eq:polar}
		\p^\circ(x,\xi^*) 
		= \sup_{ \substack{\xi \in \R^N \setminus \left\{ 0 \right\} } }
		\frac{\xi^\ast \ip \xi}{\p(x,\xi)}
		= \sup_{ \substack{\xi \in \R^N \\ \p(x,\xi) \leq 1 }  } \xi^\ast \ip \xi \qq \text{ for } x, \xi^\ast \in \R^N\,,
	\end{equation}
	where the last equality results from the homogeneity of $\p$ in $\xi$.

    \smallskip
	
    At this stage we define the anisotropic total variation and the anisotropic perimeter.
	
	\begin{defi}[anisotropic total variation and anisotropic perimeter]
		We consider a Borel function $\p\colon\R^N\times\R^N\to{[0,\infty]}$ such that $\xi\mapsto\p(x,\xi)$ is positively homogeneous of degree $1$ and a Borel set\/ $B\subseteq\R^N$. Whenever we have $w\in\BV_\loc(U)$ for some open $U\subseteq\R^N$ with $B\subseteq U$, the $\p$-anisotropic total variation of\/ $w$ on $B$ is defined as
		\[
		\TV_\p(w,B)\coleq|\D w|_\p(B)\coleq\int_B\p(\,\cdot\,,\nu_w)\,\d|\D w|
		\]
		with the Radon-Nikod\'ym derivative $\nu_w\coleq\frac{\d\D w}{\d|\D w|}$.
		Specifically, whenever a measurable $E\subseteq\R^N$ has locally finite perimeter in some open $U\subseteq\R^N$ with $B\subseteq U$, the $\p$-anisotropic perimeter of\/ $E$ in $B$ is
		\[
		\P_\p(E,B)\coleq|\D\1_E|_\p(B)=\int_B\p(\,\cdot\,,\nu_E)\,\d|\D\1_E|
		\]
		with the generalized inward normal $\nu_E=\frac{\d\D\1_E}{\d|\D\1_E|}$. If $\p$ fulfills the lower bound in \eqref{eq:comp-p0}, this is reasonably complemented with the convention $\P_\p(E,B)\coleq\infty$ in case of $\P(E,B)=\infty$.
		Finally, we abbreviate occasionally $\P_\p(E)\coleq\P_\p(E,\R^N)$.
	\end{defi}

    Clearly, in case of $\p(x,\xi)=|\xi|$ the definitions reduce to the standard isotropic ones mentioned earlier. Moreover, recalling $|\D\1_E|=\H^{N-1}\ecke(U\cap\partial^\ast\!E)$, we may recast the definition of anisotropic perimeter as $\P_\p(E,B)=\int_{B\cap\partial^\ast\!E}\p(\,\cdot\,,\nu_E)\,\d\H^{N-1}$. 
 
	\begin{lem}
		Suppose that $\p$ satisfies Assumption \ref{assum:phi}. Then, for any $w \in \BV(U)$ on open $U \subseteq \R^N$, we have 
		\begin{equation} \label{anis_eq:bd_TV}
			\alpha |\D w| \leq  |\D w|_{\p} \leq \beta |\D w|
            \qq\text{as measures on }U\,.
		\end{equation}
    \end{lem}
    
	\begin{proof}
		With the help of \eqref{eq:comp-p0}, for any Borel set $B\subseteq U$, we directly compute
		\begin{equation*}
			|\D w|_{\p}(B)
			=\int_B \p(\,\cdot\,,\nu_w)\ \d|\D w|
			\leq \beta \int_B |\nu_w| \ \d|\D w|
			= \beta |\D w|(B)\,.
		\end{equation*}
		Similarly, one checks $|\D w|_{\p}(B) \geq \alpha |\D w|(B)$. 
	\end{proof}
	
	We point out that (simply by evenness of the Euclidean norm) also the mirrored integrand $\widetilde\p$ satisfies the bound in \eqref{eq:comp-p0}. Therefore, \eqref{anis_eq:bd_TV} holds with $\widetilde\p$ in place of $\p$ as well.   Specifically, \eqref{anis_eq:bd_TV} implies
	\begin{equation} \label{anis_eq:bd_PER}
		\alpha \, \P(E,B) 
		\leq  \P_\p(E,B) 
		\leq \beta \, \P(E,B) \qq \text{ for all measurable } E \subseteq \R^N \text{ and Borel } B \subseteq \R^N\,. 
	\end{equation}

    Next follows a decomposition result for anisotropic total variations.
	
	\begin{lem} \label{anis_lem:add_parts}
		Suppose that $\p$ satisfies Assumption \ref{assum:phi}, and consider an open $U \subseteq \R^N$ and $u \in \BV(U)$. Then the positive part\/ $u_+$ and the negative part\/ $u_-$ of\/ $u$ are in $\BV(U)$, and the $\p$-variation of\/ $u$ decomposes into the $\p$-variation of\/ $u_+$ and the $\widetilde\p$-variation of\/ $u_-$ in the sense that
		\begin{equation} \label{anis_eq:add_parts}
			|\D u|_{\p}
			=|\D u_+|_{\p} + |\D u_-|_{\widetilde{\p}}
			\qq\text{as measures on } U \,. 
		\end{equation}
	\end{lem}

    When additionally requiring \eqref{case:conv-lsc} and \eqref{case:cont} of Assumption \ref{assum:phi}, Lemma \ref{anis_lem:add_parts} can be proved by strict approximation. Here, however, we prefer the subsequent proof which avoids both these extra requirements.
	
    \begin{proof}
        We initially record $u_\pm\in\BV(U)$; compare the first part of the proof of \cite[Theorem 3.96]{AFP00} or directly \cite[Theorem 3.99]{AFP00} for this conclusion. Then we decompose
        \[
          U=J\cupdot J^\mathrm{c}
          \qq\text{with }J\coleq\{u^+>u^-\}\text{ and }J^\c=\{u^+=u^-\}
        \]
        and consider the jump part $\D u\ecke J$ and the diffusive part $\D u\ecke J^\c$ of $\D u$ separately.

        In treating the diffusive part, with $u^+=u^-=u^\ast$ on $J^\c$ in mind we employ solely $u^\ast$. In fact, from \cite[Proposition 3.92]{AFP00} and the chain rule of \cite[Theorem 3.99]{AFP00}, we have
        \[
          \D u\ecke\big(J^\c\cap\{u^\ast=0\}\big)\equiv0\,,\qq
          \D u_+\ecke J^\c=\D u\ecke\big(J^\c\cap\{u^\ast>0\}\big)\,,\qq
          \D u_-\ecke J^\c={-}\D u\ecke\big(J^\c\cap\{u^\ast<0\}\big)
        \]
        in $U$. By the definition of the $\p$-variation, the preceding formulas, and the homogeneity of $\p$, for any Borel set $B\subseteq J^\c\subseteq U$, we deduce
        \[\begin{aligned}
          |\D u|_\p(B)
          &=\int_{B\cap\{u^\ast>0\}}\p\bigg(\,\cdot\,,\frac{\d{\D u}}{\d|\D u|}\bigg)\,\d|\D u|
          +\int_{B\cap\{u^\ast<0\}}\p\bigg(\,\cdot\,,\frac{\d{\D u}}{\d|\D u|}\bigg)\,\d|\D u|\\
          &=\int_B\p\bigg(\,\cdot\,,\frac{\d{\D u_+}}{\d|\D u|}\bigg)\,\d|\D u|
          +\int_B\p\bigg(\,\cdot\,,{-}\frac{\d{\D u_-}}{\d|\D u|}\bigg)\,\d|\D u|
          =|\D u_+|_\p(B)+|\D u_-|_{\widetilde\p}(B)\,.
        \end{aligned}\]

        In treating the jump part, we exploit the fact that
        \begin{equation}\label{eq:decomp-u-pm}
          u^+-u^-=\big[(u_+)^+-(u_+)^-\big]+\big[(u_-)^+-(u_-)^-\big]
          \qq\text{in }U
        \end{equation}
        by Lemma \ref{lem:dec_upper_limit}, and in addition we now verify the auxiliary $\H^{N-1}$-\ae{} equalities
        \begin{equation}\label{eq:nu-u-pm}
          \nu_{u_+}=\nu_u\text{ on }J_+\coleq\{(u_+)^+>(u_+)^-\}\,,\qq\qq
          \nu_{u_-}=-\nu_u\text{ on }J_-\coleq\{(u_-)^+>(u_-)^-\}\,.
        \end{equation}
        Indeed, \cite[Theorem 3.78]{AFP00} gives $|\D u|\ecke J=(u^+-u^-)\H^{N-1}\ecke J$, which implies $\D u\ecke J=(u^+-u^-)\nu_u\H^{N-1}\ecke J$ with the Radon-Nikod\'ym density $\nu_u=\frac{\d\D u}{\d|\D u|}$. Similarly, but also taking into account that $J_\pm\subseteq J$ and $\D u_\pm\ecke\big(J\setminus J_\pm\big)\equiv0$ by \cite[Lemma 3.76]{AFP00}, we find $\D u_\pm\ecke J=\big[(u_\pm)^+-(u_\pm)^-\big]\nu_{u_\pm}\H^{N-1}\ecke J$. We exploit these observations to rewrite both sides of $\D u\ecke J=\D u_+\ecke J-\D u_-\ecke J$ and subsequently compare the $\H^{N-1}$-densities to conclude that
        \[
          \big(u^+-u^-\big)\nu_u
          =\big[(u_+)^+-(u_+)^-\big]\nu_{u_+}-\big[(u_-)^+-(u_-)^-\big]\nu_{u_-}
          \qq\text{holds }\H^{N-1}\text{-\ae{} on }J\,.
        \]
        Then, comparing with \eqref{eq:decomp-u-pm} and taking into account $|\nu_u|=1=|\nu_{u_\pm}|$, for $V_+\coleq\big[(u_+)^+-(u_+)^-\big]\nu_{u_+}$ and $V_-\coleq{-}\big[(u_-)^+-(u_-)^-\big]\nu_{u_-}$, we arrive at $|V_++V_-|=|V_+|+|V_-|$. The last equality can occur with $V_+\neq0$ only if $V_++V_-$ and $V_+$ point in the same direction, in other words we have shown $\nu_{u_+}=\nu_u$ on $J_+$. In the same way we deduce ${-}\nu_{u_-}=\nu_u$ on $J_-$ and thus have finally deduced the validity of \eqref{eq:nu-u-pm}. At this stage, via the definition of $\p$-variation, the preceding representations of $\D u\ecke J$ and $\D u_\pm\ecke J$, the homogeneity of $\p$, \eqref{eq:decomp-u-pm}, and \eqref{eq:nu-u-pm}, for any Borel set $B\subseteq J\subseteq U$, we arrive at
        \[\begin{aligned}
          |\D u|_\p(B)
          &=\int_B(u^+-u^-)\p(\,\cdot\,,\nu_u)\,\d\H^{N-1}\\
          &=\int_B((u_+)^+-(u_+)^-)\p(\,\cdot\,,\nu_u)\,\d\H^{N-1}
          +\int_B((u_-)^+-(u_-)^-)\p(\,\cdot\,,\nu_u)\,\d\H^{N-1}\\
          &=\int_B((u_+)^+-(u_+)^-)\p(\,\cdot\,,\nu_{u_+})\,\d\H^{N-1}
          +\int_B((u_-)^+-(u_-)^-)\p(\,\cdot\,,{-}\nu_{u_-})\,\d\H^{N-1}\\
          &=|\D u_+|_\p(B)+|\D u_-|_{\widetilde\p}(B)\,.
        \end{aligned}\]
        
        All in all, with the equalities obtained for both $B\subseteq J^\c$ and $B\subseteq J$, the proof is complete.
    \end{proof}
    
    A fundamental property of the total variation measure is its lower semicontinuity. As a corollary of \cite[Theorem 2.38]{AFP00} we obtain that this property extends from the isotropic to the anisotropic case as follows.
	
	\begin{thm}[Reshetnyak lower semicontinuity theorem]
		\label{anis_thm:LSC_TV_1}
		Let $U$ be an open set in $\R^N$. Consider $(u_k)_k$ and $u$ in $\BV(U)$ 
		such that $u_k \to u$ in $\L^1(U)$.
		Suppose $\p\colon \R^N \times \R^N \to {[0,\infty]}$ satisfies $\p(x,\xi)\ge\alpha|\xi|$ for all $x,\xi\in\R^N$ and some $\alpha>0$, that $\xi\mapsto\p(x,\xi)$ is positively $1$-homogeneous of degree $1$ and convex for all $\xi\in\R^N$, and that\/  $(x,\xi)\mapsto\p(x,\xi)$ is lower semicontinuous. Then it holds
		\[
          \liminf_{k \to \infty}|\D u_k|_{\p}(U)
          \geq|\D u|_{\p}(U)\,.
        \]
	\end{thm}

    By \cite[Theorem 2.39]{AFP00} anisotropic total variations also inherit continuity from the standard total variation:
	
	\begin{thm}[{Reshetnyak continuity theorem}]
		\label{anis_thm:Reshetnyak_cont}
		Let $U$ be an open set in $\R^N$. Consider $(u_k)_k$ and $u$ in $\BV(U)$ such that $u_k\to u$ strictly in $\BV(U)$. Suppose that $\p\colon \R^N \times \R^N \to {[0,\infty)}$ satisfies $\p(x,\xi)\leq\beta|\xi|$ for all $x,\xi\in\R^N$ and some $0\leq\beta<\infty$ and that $(x,\xi)\mapsto\p(x,\xi)$ is continuous. Then it holds
		\[
          \lim_{k \to \infty}|\D u_k|_{\p}(U)
          = |\D u|_{\p}(U)\,.
        \]
	\end{thm}

    In fact, we may summarize Theorem \ref{anis_thm:Reshetnyak_cont} by saying that (under the specified assumptions on $\p$) strict convergence implies $\p$-strict convergence in the sense of the next definition.

 	\begin{defi}[$\p$-strict convergence] Suppose that $\p$ satisfies Assumption \ref{assum:phi}
		and consider a sequence of functions $(u_k)_k$ in $\BV(U)$, where $U$ is open in $\R^N$. 
		We say that $(u_k)_k$ converges $\p$-strictly in $\BV(U)$ to $u \in \BV(U)$ if we have $\lim_{k\to\infty}\|u_k-u\|_{\L^1(U)}=0$ and\/ $\lim_{k\to\infty}|\D u_k|_\p(U)=|\D u|_\p(U)$.
	\end{defi}

    Further we recall:
	
	\begin{thm}[anisotropic Fleming-Rishel coarea formula] \label{anis_thm:coarea} 
		Consider $\p\colon \R^N \times \R^N \to {[0,\infty)}$ such that Assumption \ref{assum:phi} holds, and let\/ $U \subseteq \R^N$ be open. Then, for each $u \in \BV(U)$, the superlevel sets $E_t \coleq \left\{x \in U  \, : \, u(x) > t \right\}$ satisfy $\P(E_t,U)<\infty$ for \ae{} $t \in \R$, and the following coarea formula holds\textup{:}
		\begin{equation} \label{anis_eq:coarea}
			|\D u|_{\p}(U) 
			= \int_{-\infty}^{\infty}\P_{\p}(E_t,U)\dt\,.
		\end{equation}
	\end{thm}
	
	Though \eqref{anis_eq:coarea} is basically known (compare \cite[Remark 4.4]{AmaBel94}, for instance), previous recordings seem to be based on slightly different definitions of $\p$-perimeter. Therefore, we briefly indicate how the formula can be deduced from the standard isotropic case in \cite[Theorem 3.40]{AFP00}.
	
	\begin{proof}[Sketch of proof for Theorem \ref{anis_thm:coarea}]
		Since \cite[Theorem 3.40]{AFP00} directly confirms $\P(E_t,U)<\infty$ for \ae{} $t \in \R$, it remains to establish \eqref{anis_eq:coarea}. To this end, we rewrite the formulas of \cite[Theorem 3.40]{AFP00} with the help of $\D u=\nu_u|\D u|$ and $\D\1_{E_t}=\nu_{E_t}\H^{N-1}\ecke\partial^\ast\!E_t$ as
		\begin{equation}\label{is_eq:coarea}
			\int_U g\,\d|\D u|
			=\int_{-\infty}^\infty\int_{U\cap\partial^\ast\!E_t}g\,\d\H^{N-1}\dt\,,
			\qq
			\int_U V\cdot\nu_u\,\d|\D u|
			=\int_{-\infty}^\infty\int_{U\cap\partial^\ast\!E_t}V\cdot\nu_{E_t}\,\d\H^{N-1}\dt\,,
		\end{equation}
		where we have generalized in a standard way from the sets $B$ of \cite[Theorem 3.40]{AFP00} to arbitrary bounded Borel functions $g\colon U\to\R$ and bounded Borel vector fields $V\colon U\to\R^N$, respectively. Combining both formulas in \eqref{is_eq:coarea}, we may also assert
		\begin{equation}\label{eq:int-nu=nu}
			\int_{-\infty}^\infty\int_{U\cap\partial^\ast\!E_t}V\cdot\nu_u\,\d\H^{N-1}\dt
			=\int_{-\infty}^\infty\int_{U\cap\partial^\ast\!E_t}V\cdot\nu_{E_t}\,\d\H^{N-1}\dt\,.
		\end{equation}
		Next we fix a countable dense subset $\mathscr{C}$ of $\C^0_\c(U,\R^N)$ and let $V\coleq\eta(u)\,\Psi$ with arbitrary $\eta\in\C^0_\c(\R)$ and $\Psi\in\mathscr{C}$. We plug (a Borel representative of) this $V$ into \eqref{eq:int-nu=nu}, exploit $u(x)=t$ for $\H^{N-1}$-\ae{} $x\in U\cap\partial^\ast\!E_t$ for \ae{} $t\in\R$, and then use the fundamental lemma of the calculus of variations to infer
		\begin{equation}\label{eq:nu=nu-integral}
          \int_{U\cap\partial^\ast\!E_t}\Psi\cdot\nu_u\,\d\H^{N-1}
		   =\int_{U\cap\partial^\ast\!E_t}\Psi\cdot\nu_{E_t}\,\d\H^{N-1}
          \qq\text{for all }\Psi\in\mathscr{C}\text{ and }t\in\R\setminus Z\,,
        \end{equation}
        where $Z\subseteq\R$ is negligible and due to countability of $\mathscr{C}$ can indeed be taken independent of $\Psi\in\mathscr{C}$. By density of $\mathscr{C}$ we generalize \eqref{eq:nu=nu-integral} from $\Psi\in\mathscr{C}$ to arbitrary $\Psi\in\C^0_\c(U,\R^N)$ with unchanged negligible $Z$. Then, we apply the fundamental lemma of the calculus of variations once more to deduce that
		\begin{equation}\label{eq:nu=nu}
			\nu_u=\nu_{E_t}\text{ holds }\H^{N-1}\text{-\ae{} on }U\cap\partial^\ast\!E_t\text{ for \ae{} }t\in\R\,.
		\end{equation}
		Via the definition of the $\p$-variation, the first formula in \eqref{is_eq:coarea}, the observation \eqref{eq:nu=nu}, and the definition of the $\p$-perimeter, we then conclude
		\[\begin{aligned}
			|\D u|_\p(U)&=\int_U\p(\,\cdot\,,\nu_u)\,\d|\D u|
			=\int_{-\infty}^\infty\int_{U\cap\partial^\ast\!E_t}\p(\,\cdot\,,\nu_u)\,\d\H^{N-1}\dt\\
			&=\int_{-\infty}^\infty\int_{U\cap\partial^\ast\!E_t}\p(\,\cdot\,,\nu_{E_t})\,\d\H^{N-1}\dt
			=\int_{-\infty}^\infty\P_\p(E_t,U)\dt
		\end{aligned}\]
		and arrive at the claim \eqref{anis_eq:coarea}.
	\end{proof}

    The next lemma adapts \cite[Lemma 2.9]{Schmidt25} to the anisotropic case and will be useful in dealing with parametric semicontinuity.
 
	\begin{lem}\label{lem:P(AcapB),P(A-S)}
		We consider an open $U\subseteq\R^N$ and a Borel function $\p\colon\R^N\times\R^N\to{[0,\infty)}$ such that
        $\p(x,\xi)\le\beta|\xi|$ for all $x,\xi\in\R^N$ and some $0\leq\beta<\infty$. Then, for sets $A,R,S\subseteq\R^N$ with $\P(A,U){+}\P(R,U){+}\P(S,U)<\infty$, there hold
		\[
		   \P_\p(A\cap R,U)\leq\P_\p(A,R^1\cap U)+\P_\p(R,A^+\cap U)
          \quad\text{and}\quad
		   \P_\p(A\setminus S,U)\leq\P_\p(A,S^0\cap U)+\P_{\widetilde\p}(S,A^+\cap U)\,,
		\]
		where the conditions $\H^{N-1}(\partial^\ast\!A\cap\partial^\ast\!R\cap U)=0$ and
		$\H^{N-1}(\partial^\ast\!A\cap\partial^\ast\!S\cap U)=0$, respectively, are sufficient
		for equality.
	\end{lem}
	
	\begin{proof}
		From \cite[eq.\@ (3.10)]{AFP00} we get
		$\P(A\cap R,U)=|\D(\1_A\1_R)|(U)<\infty$, and from Federer's structure theorem \cite[Theorem 3.61]{AFP00} we read off  $\partial^\ast(A\cap R)\cap U\subseteq(\partial^\ast\!A\cap R^1\cap U)\cupdot(\partial^\ast\!R\cap A^+\cap U)$ up to $\H^{N-1}$-negligible sets. In addition, the full statement of De Giorgi's structure theorem as provided in \cite[Theorem 3.59]{AFP00} implies, for sets $E\subseteq F$ of finite perimeter and $x\in\partial^\ast\!E\cap\partial^\ast\!F$, that necessarily $\nu_E(x)=\nu_F(x)$ holds. Therefore, we may even assert that, for $\H^{N-1}$-\ae{} $x\in\partial^\ast(A\cap R)\cap U$, there hold
		either $x\in\partial^\ast\!A\cap R^1\cap U$, $\nu_{A\cap R}(x)=\nu_A(x)$ or
		$x\in\partial^\ast\!R\cap A^+\cap U$, $\nu_{A\cap R}(x)=\nu_R(x)$. As a
		consequence, we obtain
		\[\begin{aligned}
			\P_\p(A\cap R,U)
			&=\int_{\partial^\ast(A\cap R)\cap U}\p(\,\cdot\,,\nu_{A\cap R})\,\d\H^{N-1}\\
			&\leq\int_{\partial^\ast\!A\cap R^1\cap U}\p(\,\cdot\,,\nu_A)\,\d\H^{N-1}
			+\int_{\partial^\ast\!R\cap A^+\cap U}\p(\,\cdot\,,\nu_R)\,\d\H^{N-1}\\
			&=\P_\p(A,R^1\cap U)+\P_\p(R,A^+\cap U)\,.
		\end{aligned}\]
		In case of $\H^{N-1}(\partial^\ast\!A\cap\partial^\ast\!R\cap U)=0$, the reasoning with Federer's structure theorem even yields the partition $\partial^\ast(A\cap R)\cap U=(\partial^\ast\!A\cap R^1\cap U)\cupdot(\partial^\ast\!R\cap A^+\cap U)$ up to
		$\H^{N-1}$-negligible sets, and thus the preceding inequality becomes an
		equality. This proves the claims for $A\cap R$. The claims for $A\setminus S$
		follow by plugging in $R=S^\c$ and exploiting
		$\P_\p(S^\c,A^+\cap U)=\P_{\widetilde\p}(S,A^+\cap U)$.
	\end{proof}
		
	\subsection{Approximation and continuity results}

    The next result on strict approximation of $\BV$ functions with arbitrarily prescribed boundary datum can be retrieved from \cite[Lemma B.2]{Bildhauer00} or \cite[Theorem 1.2]{Schmidt15}. In fact, the sources assert even area-strict convergence, which is stronger than the strict convergence recorded here, but the latter suffices for our purposes.
   
	\begin{thm}[strict approximation with prescribed boundary datum]
        \label{strict_int_appr_BV}
		Suppose that $\Omega$ is an open, bounded set in $\R^N$ with Lipschitz boundary. 
		Then, for every $u_0 \in \W^{1,1}(\Omega)$ and every $u \in \BV(\Omega)$, there exists a sequence $(u_k)_k$ in $u_0 + \C^{\infty}_\c(\Omega)$ such that, if we set $\ol{u_k}\coleq u_k\1_\Omega+u_0\1_{\R^N\setminus\ol{\Omega}}$ and $\ol{u}\coleq u\1_\Omega+u_0\1_{\R^N\setminus\ol{\Omega}}$, then $(\ol{u_k})_k$ converges to $\ol{u}$ strictly in $\BV(\R^N)$.
	\end{thm}

    We will also make use of one-sided strict approximations, essentially obtained in \cite[Theorem 3.3]{CarDalLeaPas88}. The only marginal extra features recorded in our subsequent restatement are that we allow for unbounded $U$ and add the inequalities $u_1\ge u_k$ and $v_1\le v_k$. However, at least these extra inequalities are not really new, since it has been shown in \cite[Theorem 4.5, Corollary 4.7]{Treinov21} by a slightly involved truncation argument that even the whole sequences $(u_k)_k$ and $(v_k)_k$ can be taken monotone. Furthermore, we remark that the sources provide even area-strict convergence, and we could do the same, but for simplicity we restrict ourselves once more to the usual strict convergence, which is enough for our purposes. In any case, below we include a brief deduction of precisely our statement on the basis of solely \cite[Theorem 3.3]{CarDalLeaPas88}.

	\begin{prop}[one-sided strict approximation]\label{prop:mon_area_strict_appr}
		Consider an open $U \subseteq \R^N$ and $u \in \BV(U)$. Then there exists a sequence $(u_k)_k$ in $\W^{1,1}(U)$ such that $(u_k)_k$ converges to $u$ strictly in $\BV(U)$ and such that $u_1\ge u_k\ge u$ \ae{} in $U$ for all $k$. Analogously, there exists a sequence $(v_k)_k$ in $\W^{1,1}(U)$ such that $(v_k)_k$ converges to $u$ strictly in $\BV(U)$ and such that $v_1\le v_k\le u$ \ae{} in $U$ for all $k$. If $u$ has compact support in $U$, we may even take $u_k\in\W^{1,1}_0(U)$ and $v_k\in\W^{1,1}_0(U)$, respectively.
	\end{prop}

    \begin{proof}
      We fix $u_1\in\W^{1,1}(U)$ such that $u_1\ge u$ \ae{} in $U$ and such that $u_1-u$ is bounded away from zero \ae{} in $U\cap\B_r$ for each $r>0$. Moreover, for each $k$, we choose $R_k>0$ such that $\|u_1\|_{\L^1(U\setminus\B_{R_k})}+\|u\|_{\L^1(U\setminus\B_{R_k})}<\frac1k$, $|\D u_1|(U\setminus\B_{R_k})<\frac1k$, and $\|\mathrm{T}^\mathrm{ext}_{\partial\B_{R_k}}\!\!u_1\|_{\L^1(U\cap\partial\B_{R_k};\H^{N-1})}+\|\mathrm{T}^\mathrm{int}_{\partial\B_{R_k}}\!\!u\|_{\L^1(U\cap\partial\B_{R_k};\H^{N-1})}<\frac1k$, and we fix $\eps_k>0$ such that $u_1-u\ge\eps_k$ on $U\cap\B_{R_k}$. Then we apply \cite[Theorem 3.3]{CarDalLeaPas88} to $\widetilde u_k\coleq u\1_{U\cap\B_{R_k}}+u_1\1_{U\setminus\B_{R_k}}\in\BV(U)$ on the bounded open set $U\cap\B_{R_k+1}$ to obtain a sequence $(w_{k,\ell})_\ell$ in $\W^{1,1}(U\cap\B_{R_k+1})$ such that $(w_{k,\ell})_\ell$ converges to $\widetilde u_k$ strictly in $\BV(U\cap\B_{R_k+1})$ and such that $w_{k,\ell}\ge\widetilde u_k\ge u$ \ae{} in $U\cap\B_{R_k+1}$ for all $\ell$. Now we consider $\min\{w_{k,\ell},u_1\}\in\W^{1,1}(U)$, which coincides with $u_1$ on $U\cap(\B_{R_k+1}\setminus\B_{R_k})$ and is extended by the values of $u_1$ also outside $\B_{R_k+1}$. We record
      \[
        \lim_{\ell\to\infty}\|{\min}\{w_{k,\ell},u_1\}-u\|_{\L^1(U)}
        =\|\widetilde u_k-u\|_{\L^1(U)}
        =\|u_1-u\|_{\L^1(U\setminus\B_{R_k})}
        <{\ts\frac1k}\,,
      \]
      and we observe $|\{w_{k,\ell}\ge u_1\}\cap\B_{R_k}|\le|\{w_{k,\ell}-\widetilde u_k\ge\eps_k\}|\le\frac1{\eps_k}\|w_{k,\ell}-\widetilde u_k\|_{\L^1(U)}\to0$ and, as a consequence, also $|\D u_1|(\{w_{k,\ell}\ge u_1\}\cap\B_{R_k})\to0$ for $\ell\to\infty$. All in all, we infer
      \[\begin{aligned} 
        &\limsup_{\ell\to\infty}\big|\D\min\{w_{k,\ell},u_1\}\big|(U)\\
          &\le\limsup_{\ell\to\infty}\big[|\D w_{k,\ell}|(U\cap\B_{R_k})
          +|\D u_1|(\{w_{k,\ell}\ge u_1\}\cap\B_{R_k})\big]
          +|\D u_1|(U\setminus\B_{R_k})\\
          &\le|\D\widetilde u_k|(U)
          +|\D u_1|(U\setminus\B_{R_k})\\
          &=|\D u|(U\cap\B_{R_k})
          +\|\mathrm{T}^\mathrm{ext}_{\partial\B_{R_k}}\!\!u_1-\mathrm{T}^\mathrm{int}_{\partial\B_{R_k}}\!\!u\|_{\L^1(U\cap\partial\B_{R_k};\H^{N-1})}
          +2|\D u_1|(U\setminus\B_{R_k})\\
          &<|\D u|(U)+\ts{\frac3k}\,.
      \end{aligned}\]
      If now, for $k\ge2$, we choose
      $u_k\coleq\min\{w_{k,\ell_k},u_1\}\in\W^{1,1}(U)$ with suitably large $\ell_k$, then $(u_k)_k$ still converges to $u$ strictly in $\BV(U)$, and we have ensured $u_1\ge u_k\ge u$ \ae{} in $U$. This completes the proof for the main $(u_k)_k$ case. (As a side remark we record that the preceding argument reduces quite a bit in case of bounded $U$, since then one can take $u_1\coleq u+1$, need not introduce $R_k$, $\eps_k$, $\widetilde u_k$, and can apply \cite[Theorem 3.3]{CarDalLeaPas88} directly to $u$ itself to get just one sequence $(w_k)_k$.)

      Now the $(v_k)_k$ case follows from the $(u_k)_k$ one by changing signs.      

      Finally, if $u$ has compact support in $U$, we additionally multiply all $u_k$ and $v_k$ with a fixed cut-off function in order to ensure even $u_k\in\W^{1,1}_0(U)$ and $v_k\in\W^{1,1}_0(U)$, respectively; compare e.\@g.\@ with \cite[Lemma 2.22]{Schmidt25}.
    \end{proof}

    For our purposes, the decisive feature of Proposition \ref{prop:mon_area_strict_appr} is that the approximations obtained are good enough for applying both the Reshetnyak continuity theorem and the following basic continuity lemma.
    
    \begin{lem}[fine one-sided continuity properties of $\mu$-integrals]\label{lem:mon_area_strict_appr}
        Consider a non-negative Radon measure $\mu$ on an open set $U\subseteq\R^N$ such that $\mu(Z)=0$ for every $\H^{N-1}$-negligible Borel set $Z\subseteq U$ and $\int_U w^+\,\d\mu<\infty$ for every non-negative $w\in\BV(U)$. If a sequence $(u_k)_k$ in $\W^{1,1}(U)$ converges to $u\in\BV(U)$ strictly in $\BV(U)$ and satisfies $u_1\ge u_k\ge u$ \ae{} in $U$ for all $k$, then there exists a subsequence $(u_{k_\ell})_\ell$ such that
        \begin{equation}\label{eq:mu-conv-plus}
          \lim_{\ell \to \infty} \int_{U} u_{k_\ell}^\ast\,\d\mu
          = \int_{U} u^+ \,\d\mu\,.
        \end{equation}
        Analogously, if a sequence $(v_k)_k$ in $\W^{1,1}(U)$ converges to $u\in\BV(U)$ strictly in $\BV(U)$ and now satisfies $v_1\le v_k\le u$ \ae{} in $U$ for all $k$, then there exists a subsequence $(v_{k_\ell})_\ell$ such that
        \begin{equation}\label{eq:mu-conv-minus}
          \lim_{\ell \to \infty} \int_{U} v_{k_\ell}^\ast\,\d\mu
          = \int_{U} u^- \,\d\mu\,.
        \end{equation}
    \end{lem}

    \begin{proof}
        In case of the sequence $(u_k)_k$, the \ae{} inequality $u_1\ge u_k\ge u$ induces the $\H^{N-1}$-\ae{} inequality $u_1^\ast\ge u_k^\ast\ge u^+$ in $U$, and hence Theorem \ref{strict_Hausdorff_repr} gives the subsequence such that $\lim_{\ell\to\infty}u_{k_\ell}^\ast=u^+$ holds $\H^{N-1}$-\ae{} in $U$. By assumption on $\mu$, the latter convergence remains valid $\mu$-\ae{} in $U$. Since the assumptions on $\mu$ ensure also $u_1^\ast\in\L^1(U;\mu)$ and $u^+\in\L^1(U;\mu)$, we can apply the dominated convergence theorem to deduce \eqref{eq:mu-conv-plus}.

        In case of the sequence $(v_k)_k$, we apply \eqref{eq:mu-conv-plus} with ${-}v_k$ in place of $u_k$ and ${-}u$ in place of $u$. Then we deduce \eqref{eq:mu-conv-minus} by reversing the sign on both sides and observing $-({-}u)^+=u^-$.
    \end{proof}
	
	\begin{lem}[basic properties of truncations] \label{lem:cut_M}
		Given an open set $U \subseteq \R^N$ and a positive constant $M$, for every $u \in \BV(U)$ we introduce the truncation $u^M$ of $u$ at level $M$ defined through
		\begin{equation} \label{eq:cutoff_M}
			u^M\coleq\max\left\{\min\{u,M\},{-}M\right\}\,.
		\end{equation}
		We record\textup{:}
		\begin{enumerate}[{\rm(i)}]
			\item It is $u^M \in \BV(U) \cap \L^{\infty}(U)$.\label{item:cut_M_i}
			\item\label{item:cut_M_ii}
			If $\p$ satisfies Assumption \ref{assum:phi}, we have additivity of anisotropic $\TV$ on truncations in the sense that
			\[
			|\D u|_{\p} = |\D(u-u^M)|_{\p}+|\D u^M|_{\p} \qq \text{ as measures on } U.
			\]
			\item There hold $u^M \to u$ and $(u^M)_\pm \to u_\pm$ strongly in $\BV(U)$ as $M \to \infty$. \label{item:cut_M_iii}
            \item If $\p$ satisfies Assumption \ref{assum:phi}, we have $\p$-strict convergence $u^M \to u$ in $\BV(U)$ for $M \to \infty$. \label{item:cut_M_iv}         
			\item We have the $\H^{N-1}$-\ae{} convergences $(u^M)^\pm\to u^\pm$ on $U$ for $M\to\infty$. \label{item:cut_M_v}
		\end{enumerate}
  
		\begin{proof}
			\eqref{item:cut_M_i} We exploit that minimum and maximum of $\BV$ functions are still $\BV$ and that evidently $|u^M|\leq M$.
			
			\smallskip
			
			\noindent\eqref{item:cut_M_ii} Suppose first $u \in \BV(U)$ is non-negative almost everywhere. Then one checks that, for any $M>0$, there hold $u-u^M=(u-M)_+$ and $u^M=M-(u-M)_-$ \ae{} in $U$. Thus, applying \eqref{anis_eq:add_parts} we find
			\begin{align*}
				|\D(u-u^M)|_{\p}+|\D u^M|_{\p}
				&=|\D(u-M)_+|_{\p} + |\D(M-(u-M)_-)|_{\p}\\
				&=|\D(u-M)_+|_{\p} + |\D(u-M)_-|_{\widetilde\p}
				=|\D(u-M)|_{\p}=|\D u|_{\p},
			\end{align*}
			as required. For $u \in \BV(U)$ of arbitrary sign, we apply the previous result to $u_+$ and $u_-$ and employ $(u-u^M)_\pm = u_\pm - (u_\pm)^M$ and $(u_\pm)^M=(u^M)_\pm$ \ae{} on $U$. Using \eqref{anis_eq:add_parts} again, we get
			\begin{align*}
				|\D u|_{\p}
				=|\D u_+|_{\p} + |\D u_-|_{\widetilde{\p}}
				&= |\D(u_+-(u_+)^M)|_{\p}+|\D (u_+)^M|_{\p} + |\D(u_--(u_-)^M)|_{\widetilde\p}+|\D (u_-)^M|_{\widetilde\p} \\
				&=|\D((u-u^M)_+))|_{\p}+|\D(u^M)_+|_{\p} + |\D((u-u^M)_-))|_{\widetilde\p}+|\D(u^M)_-|_{\widetilde\p} \\
				&=|\D(u-u^M)|_{\p}+|\D u^M|_{\p}
			\end{align*}
			as required. 
			
			\smallskip
			
			\noindent\eqref{item:cut_M_iii}
            Since the function $t\mapsto\max\{\min\{t,M\},{-}M\}$ is a contraction, a standard argument (see e.\@g.\@ the first part of the proof of \cite[Theorem 3.96]{AFP00}) yields
			\[
			|\D u^M|(U) \leq |\D u|(U) \qq \text{ for all } M > 0\,.
			\]
			In addition, as one easily checks $u^M \to u$ in $\L^1(U)$ as $M\to\infty$, the lower semicontinuity of the total variation yields
			\[
			|\D u|(U) 
			\leq \liminf_{M \to \infty} |\D u^M|(U)\,.
			\]
            Combining the previous observations, strict convergence is achieved. Moreover, by \eqref{item:cut_M_ii} for the case $\p(x,\xi)=|\xi|$, strict convergence of the truncated sequence implies its strong convergence. The same works for $(u_\pm)^M$.
			
			\smallskip

            \noindent\eqref{item:cut_M_iv} From \eqref{item:cut_M_iii} and \eqref{eq:comp-p0}, we have $|\D(u-u^M)|_{\p}(U) \to 0$ for $M \to \infty$, and thus the claim follows by \eqref{item:cut_M_ii}. (If even Assumption \ref{assum:phi}\eqref{case:cont} applies, the claim also follows from \eqref{item:cut_M_iii} and Theorem \ref{anis_thm:Reshetnyak_cont}.)

            \smallskip
   
			\noindent\eqref{item:cut_M_v} For each fixed $x\in U$ such that $u^\pm(x)$ is finite, the convergence is immediate from the observation that $(u^M)^\pm(x)=u^\pm(x)$ for all $M>|u^\pm(x)|$.
		\end{proof}
	\end{lem}
	
    \begin{lem}[fine continuity property of $\mu$-integrals along sequences of truncations]\label{lem:int-uM-to-int-u}
        Consider a non-negative Radon measure $\mu$ on open $U\subseteq\R^N$ such that $\mu(Z)=0$ for every $\H^{N-1}$-negligible Borel set $Z\subseteq U$. Then, for every $u\in\BV(U)$ with $\int_U|u|^+\d\mu<\infty$ and its truncations $u^M$ from Lemma \ref{lem:cut_M}, there hold
        \begin{equation}\label{eq:int-uM-to-int-u}
          \lim_{M \to \infty} \int_U \big(u^M\big)^+ \,\d\mu
          = \int_U u^+\ \d \mu
          \qq\text{and}\qq
          \lim_{M \to \infty} \int_U \big(u^M\big)^- \,\d\mu
          = \int_U u^-\ \d \mu\,.
        \end{equation}
    \end{lem}

    \begin{proof}
         Suppose first that $u\in\BV(U)$ is non-negative. Then we have $u^{M_1}\leq u^{M_2}$ for $M_1\leq M_2$, and from Lemma \ref{lem:cut_M}\eqref{item:cut_M_v} we infer $(u^M)^\pm\to u^\pm$ for $M\to\infty$ first $\H^{N-1}$-\ae{} and by assumption on $\mu$ also $\mu$-\ae{} on $U$. Hence, the monotone convergence theorem confirms the claim \eqref{eq:int-uM-to-int-u} in this case. 

         For arbitrary $u\in\BV(U)$, Lemmas \ref{lem:cut_M}\eqref{item:cut_M_i} and \ref{lem:dec_upper_limit} yield $\H^{N-1}$-\ae{} decompositions $u^\pm=(u_+)^\pm-(u_-)^\mp$ and
         $(u^M)^\pm=((u^M)_+)^\pm-((u^M)_-)^\mp=((u_+)^M)^\pm-((u_-)^M)^\mp$, and these remain valid $\mu$-\ae{} as well. We can thus apply the claim already established to the non-negative functions $u_+\in\BV(U)$ and $u_-\in\BV(U)$ and combine the resulting integral convergences to reach \eqref{eq:int-uM-to-int-u} in the general case. Here, in combining we exploit that thanks to the assumption $\int_U|u|^+\d\mu<\infty$ all relevant integrals remain well-defined and finite.
    \end{proof}

	\begin{rem} \label{rem:additivity_upper_lower_limit}
		For a \textbf{non-negative} measurable function $u\colon U\to{[0,\infty)}$ on open $U \subseteq \R^N$ and a positive constant $M$, there hold
		\begin{equation} \label{eq:additivity_upper_lower_limit}
			u^+ = (u-u^M)^+ + (u^M)^+ \ \text{ and } \ 	u^- = (u-u^M)^- + (u^M)^-
            \qq\text{in }U\,.  
		\end{equation}
    \end{rem}

    \begin{proof}
        From the definition of the approximate upper and lower limits, one can check by case distinction between $u^\pm(x)>M$ and $u^\pm(x)\leq M$ that $(u-u^M)^\pm=(u^\pm-M)_+$ and $(u^M)^\pm=\min\{u^\pm,M\}$ hold. The claims then follow when taking into account $u^\pm=(u^\pm-M)_++\min\{u^\pm,M\}$.
    \end{proof}

	\subsection[\texorpdfstring{$1$-capacity}{}]{\boldmath$1$-capacity}
	
	In the following we use $1$-capacity in the sense of \cite{CarDalLeaPas88}. Indeed, one checks easily that the following definition is equivalent to the one introduced in \cite[Definition 2.1]{CarDalLeaPas88}.
	
	\begin{defi}[$1$-capacity]
		The $1$-capacity \ka or $\BV$-capacity\kz{} of an arbitrary set $E \subseteq \R^N$ is defined as
		\[
          \mathrm{Cap}_1(E) \coleq \inf \left\{ \int_{\R^N} |\nabla u|\,\dx \, : \, u \in \W^{1,1}(\R^N), \ u \geq 1\text{ a.\@e.\@ on } U,\, \ U\text{ open, }\, E \subseteq U \right\}
        \]
		\ka with the convention $\mathrm{Cap}_1(E) \coleq \infty$ in case that the preceding infimum runs over the empty set\kz.
	\end{defi}

    It is not difficult to check that $1$-capacity is an outer measure in the sense that for $E_k,E\subseteq\R^N$ there holds
	\begin{equation} \label{eq:subadd_cap}
		E \subseteq \bigcup_{k =1}^{\infty} E_k
        \implies
		\mathrm{Cap}_1(E) \leq \sum_{k =1}^{\infty} \mathrm{Cap}_1(E_k)\,.
	\end{equation}
	
    A consequence of \cite[Theorem 2.1]{CarDalLeaPas88} is the following alternative characterization of $1$-capacity. 
	
	\begin{prop}[perimeter characterization of $1$-capacity] \label{prop:per_cap}
		For every $E \subseteq \R^N$, it is 
		\begin{equation*}
			\mathrm{Cap}_1(E) 
			= \inf \left\{\P(H) \, : \, H \subseteq \R^N \text{ measurable}, \, |H|<\infty, \, E \subseteq H^+ \right\}\,.
		\end{equation*}
	\end{prop}
	
	Finally, the subsequent result of \cite[Section 4]{FedZie72} (compare also \cite[Theorem 5.12]{EvGar15}, for instance) shows that $\mathrm{Cap}_1$ has the same negligible sets of $\H^{N-1}$. 
	
	\begin{prop}[negligible sets] \label{prop:negl_cap}
		For every Borel set $E \subseteq \R^N$, we have
		\begin{equation*}
			\mathrm{Cap}_1(E)=0
            \iff \H^{N-1}(E)=0\,.
		\end{equation*}
	\end{prop}

	\section{Statement and contextualization of the main results}\label{sec:statements}
	
	At this stage we precisely introduce the isoperimetric conditions which will be our main hypotheses.
	
	\begin{defi}[$\p$-anisotropic IC]\label{defi:IC}
		We say that a pair $(\mu,\nu)$ of finite non-negative Radon measures $\mu$ and
		$\nu$ on open $U\subseteq\R^N$ satisfies the $\p$-anisotropic isoperimetric
		condition \textup{(}in brief\textup{:} the $\p$-IC\textup{)} in $U$ with
		constant $C\in{[0,\infty)}$ if it holds
		\begin{equation}\label{eq:signed-IC}
			\mu(A^+)-\nu(A^1)\leq C\P_\p(A)
			\qq\text{for all measurable }A\Subset U\,.
		\end{equation}
		By the $\p$-IC for $\mu$ alone we mean the $\p$-IC for $(\mu,0)$, that is, the
		validity of the preceding with $\nu\equiv0$. In case of the standard isotropic
		integrand $\p(x,\xi)=|\xi|$ we speak of the isotropic IC or standard IC.
	\end{defi}
	
	In the isotropic case, this type of the condition for $\mu$ alone (or
	also for $\mu=H\LN$) has already occurred in several papers such as \cite{BomGiu73,Giaquinta74a,Giaquinta74b,Gerhardt74,Steffen76a,Steffen76b,MeyZie77,Giusti78a,DuzSte96,PhuTor08,DaiTruWan12,BoeDuzSch13,DuzSch15,DaiWanZho15,PhuTor17,Schmidt25,LeoCom24},
	for instance. Equivalent reformulations have been recorded in \cite{MeyZie77,DuzSte96,PhuTor08,PhuTor17,Schmidt25} and may be taken in the
	spirit of expressing the IC for $\mu$ as $\|\mu\|_{\BV(U)^\ast}\leq C$ or
	$\|\mu\|_{\W^{-1,\infty}(U)}\leq C$ or $\mu\leq C\mathrm{Cap}_1$ for certain
	natural norms (homogeneous of degree ${-}1$) on the dual space $\BV(U)^\ast$ of
	$\BV(U)$ and the negative Sobolev space $\W^{-1,\infty}(U)\cong(\W_0^{1,1}(U))^\ast$
	and for a relative $1$-capacity $\mathrm{Cap}_1$ on $U$. Extensions of the
	characterization results to our anisotropic setting with pairs $(\mu,\nu)$ are
	addressed in the later Theorems \ref{thm:anis_equivalence_measure_TOGETHER} and \ref{thm:anis_equivalence_gen_IC_singular}, are very convenient in the sequel, and may be of independent interest.
	
	For what concerns examples of measures with IC, it has been recorded in
	\cite[Proposition 8.1]{Schmidt25} that, for each convex or
	pseudoconvex/outward-minimizing
	$K\subseteq\R^N$ with $|K|{+}\P(K,\R^N)<\infty$, the perimeter measure
	$\P(K,\,\cdot\,)=\H^{N-1}\ecke\partial^\ast\!K$ satisfies the
	isotropic IC in $\R^N$ with constant $1$ --- where now we have normalized to
	$C=1$, the case most important for the purposes of this paper. As a
	straightforward consequence, at least $(\alpha/\beta)\P_\p(K,\cdot\,)$ with
	$\alpha$ and $\beta$ from \eqref{eq:comp-p0} satisfies also the $\p$-IC with
	constant $1$. However, for $x$-independent, convex, smooth $\p$, every bounded,
	convex, smooth $K\subseteq\R^N$ is in fact $\P_\p$-outward-minimizing, as
	discussed in \cite[Section 2.1]{ChaNov22}, and thus at least in this case we
	expect the validity of the $\p$-IC with constant $1$ also for $\P_\p(K,\cdot\,)$
	itself.
	
	Since any IC for $\mu$ alone implies by definition the same IC for $(\mu,\nu)$,
	we obtain examples of pairs $(\mu,\nu)$ with IC by choosing $\mu$ as in the
	preceding discussion and $\nu$ arbitrary. However, this basic ansatz without
	interaction between $\mu$ and $\nu$ does not yet clarify why we actually
	incorporate the additional measure $\nu$ with negative sign in the
	``signed IC'' \eqref{eq:signed-IC}. Rather the reason emerges only with the
	observation that an IC for $(\mu,\nu)$ can be based on cancellation effects and
	can thus be valid even though the corresponding IC for $\mu$ alone fails. This
        phenomenon has been observed in \cite[Proposition 5.1]{PhuTor17} and
        \cite[Remark 4.12]{ComLeo25}, for instance, and is here underpinned with the
        following pivotal example, whose proof is postponed to Section
        \ref{subsec:signed-IC}.

	\begin{examp}[a non-trivial signed IC]\label{exp:signed-IC}
		Fix $\theta\in{(0,1]}$ and the Radon measures
		$\mu\coleq(1{+}\theta/2)\H^1\ecke\partial\B_2$ and
		$\nu\coleq\theta\H^1\ecke\partial\B_1$ on $\R^2$. Then $(\mu,\nu)$
		satisfies the isotropic IC in $\R^2$ with constant $1$, while $\mu$ alone does
		not.
	\end{examp}
	
	There is, however, a price for the additional generality of the ``signed
        IC'' of type \eqref{eq:signed-IC}: While an IC for $\mu$ alone on $U$
        does imply $\int_Uv^+\,\d\mu<\infty$ for all non-negative
        $v\in\BV(U)$ (compare \cite[Theorem 4.7]{MeyZie77} or our later
        Proposition \ref{prop:admissible}, for instance), somewhat surprisingly
        an IC for $(\mu,\nu)$ on $U$ does not anymore allow this conclusion and
        does not even yield finiteness of $\int_Uv^\ast\,\d\mu$ for all
        non-negative $v\in\W^{1,1}_0(U)$. In fact, we have the following
        example, whose proof is partially deferred to Section \ref{subsec:non-finite}.
	
	\begin{examp}[the signed IC does not enforce finiteness on $\W^{1,1}$]
		\label{exp:non-finite}
		For $i\in\N$, consider the circles $S_i\coleq\partial\B_{1/i^2}$ in
		$\R^2$, and take $\mu\coleq\H^1\ecke\bigcup_{k=1}^\infty S_{2k-1}$ and
		$\nu\coleq\H^1\ecke\bigcup_{k=1}^\infty S_{2k}$. Then both $(\mu,\nu)$ and $(\nu,\mu)$ satisfy
		the isotropic IC in $\R^2$ with constant $1$, but for the compactly supported\/
		$v\in\W^{1,1}(\R^2)$ obtained by setting
		$v(x)\coleq\big(\frac1{|x|^{1/2}}{-}1\big)_+$ for $x\in\R^2$, it occurs
		$\int_{\R^2}v^\ast\,\d\mu=2\pi\sum_{k=2}^\infty\frac{2k{-}2}{(2k{-}1)^2}=\infty$ and
		$\int_{\R^2}v^\ast\,\d\nu=2\pi\sum_{k=1}^\infty\frac{2k{-}1}{(2k)^2}=\infty$.
	\end{examp}

	From Example \ref{exp:non-finite} it is clear that, when working in the sequel with the ``signed IC'' \eqref{eq:signed-IC}, the finiteness of integrals of the preceding type has to be added as an extra requirement. Therefore, we impose:
	
	\begin{assum}[admissible measures]\label{assum:mu}
		We consider non-negative Radon measures $\mu_+$ and $\mu_-$ on the bounded open set $\Omega\subseteq\R^N$ with Lipschitz boundary, and we assume
		\begin{equation}\label{eq:negligible}
			\mu_\pm(Z)=0
			\qq\text{for every }\H^{N-1}\text{-negligible Borel set }Z\subseteq\Omega
		\end{equation}
		and
		\begin{equation}\label{eq:finite-integral}
			\int_\Omega v^+\,\d\mu_\pm<\infty
			\qq\text{for every non-negative }v\in\BV(\Omega)\,.
		\end{equation}
	\end{assum}

	As the choice $v\equiv1$ is possible, Assumption \ref{assum:mu} includes the requirement that $\mu_+$ and $\mu_-$ are finite.
	
	Assumption \ref{assum:mu} seems inevitable and imposes essentially the minimal conditions to keep the signed term $\int_\Omega w^-\,\d\mu_+-\int_\Omega w^+\,\d\mu_-$ in \eqref{eq:P-functional} below well-defined for all $w\in\BV(\Omega)$. As explained in the introduction, our main interest is still in the case that the non-negative measures $\mu_+$ and $\mu_-$ result from the Jordan decomposition of a signed measure $\mu$ or, equivalently, satisfy $\mu_+\perp\mu_-$, i.\@e.\@ are singular to each other. However, since a good portion of our results (essentially semicontinuity, existence, and a part of the characterization results in Section \ref{sec:charact}) remains valid even in case $\mu_+\not\perp\mu_-$, we have 
    formulated our overarching Assumption \ref{assum:mu} for a general pair of measures $\mu_+$ and $\mu_-$ and have not included the requirement $\mu_+\perp\mu_-$ there.
	
	We also record that the combination of \eqref{eq:negligible} and
	\eqref{eq:finite-integral} is in fact equivalent with having the isotropic IC
	(or alternatively any $\p$-IC with \eqref{eq:comp-p0} valid) for both $\mu_+$ and $\mu_-$ with
	some finite constant $C<\infty$; see Proposition \ref{prop:admissible}. This means
	in particular that, when specializing our following results to the case of just one non-negative measure or when imposing separate ICs on $\mu_+$ and $\mu_-$, then
	\eqref{eq:negligible} and \eqref{eq:finite-integral} will be direct consequences
	of the more decisive ICs with normalized constant $1$, which are inevitable in
	the statements anyway. In this sense, only the ``fully signed case''
	illustrated by Example \ref{exp:non-finite} truly needs \eqref{eq:negligible}
	and \eqref{eq:finite-integral} as a weak complement to the main IC assumptions
	imposed below.
	
	Before stating our main results in Theorems \ref{thm:lsc}, \ref{thm:exist},
	\ref{thm:recovery} below, we recall that the $\BV$ version of our functional
	takes the form
	\begin{equation}\label{eq:P-functional}
		\widehat\Phi[w]
		\coleq\TV_\p^{u_0}[w]
        +\int_\Omega w^-\,\d\mu_+
        -\int_\Omega w^+\,\d\mu_-
	\end{equation}
	for $w\in\BV(\Omega)$. Here, the $\p$-anisotropic total variation
	\begin{equation}\label{eq:TV-u0}
		\TV_\p^{u_0}[w]
		\coleq|\D\overline w|_\p\big(\overline\Omega\big)
		=|\D w|_\p(\Omega)+\int_{\partial\Omega}\p(\,\cdot\,,(w{-}u_0)\nu_\Omega)\,\d\H^{N-1}
	\end{equation}
	of the extended function
	$\ol{w}\coleq w\1_\Omega+u_0\1_{\R^N\setminus\ol{\Omega}}$ on $\ol{\Omega}$
	incorporates the boundary values $u_0$ through the latter penalization term, in
	which $w$ and $u_0$ are evaluated in the sense of traces and $\nu_\Omega$
	denotes the $\H^{N-1}$-\ae{} defined inward unit normal on $\partial\Omega$.
	This said, we turn to our main semicontinuity result, which --- as we recall --- allows for possibly non-even $\p$ with mirrored integrand $\widetilde\p$, but with regard to the precise assumptions on $\mu_\pm$ is new even in the standard isotropic case $\p(x,\xi)=|\xi|$. In fact, the result reads:
	
	\begin{thm}[semicontinuity]\label{thm:lsc}
		We consider $u_0\in\W^{1,1}(\R^N)$ and impose Assumptions \ref{assum:phi}\eqref{case:conv-lsc},\eqref{case:cont} and \ref{assum:mu}. If $(\mu_-,\mu_+)$ satisfies the $\p$-IC in $\Omega$ with constant $1$ and $(\mu_+,\mu_-)$ satisfies the $\widetilde\p$-IC in $\Omega$ with constant $1$, then the functional $\widehat\Phi$ in \eqref{eq:P-functional} is lower semicontinuous on $\BV(\Omega)$ with respect to $\L^1(\Omega)$-convergence.
	\end{thm}
	
	Theorem \ref{thm:lsc} crucially relies on considerations of Section \ref{sec:par-lsc}, where we first adapt the parametric semicontinuity results of \cite{Schmidt25} to the $\p$-anisotropic framework and to more general conditions on the measures. On a rough level, the proof of Theorem \ref{thm:lsc} in Section \ref{subsec:npar-lsc} then uses the coarea and layer-cake formulas and exploits the ICs in truncation arguments in order to finally deduce the conclusion from the parametric semicontinuity.
	
	Though the ICs imposed in Theorem \ref{thm:lsc} seem just right for our purposes, for comparison and later auxiliary usage let us next define also some more general ICs of a type introduced in \cite{Schmidt25}.
	
	\begin{defi}[small-volume $\p$-IC] \label{anis_defi:SVIC}
		We say that a pair $(\mu,\nu)$ of finite non-negative Radon measures $\mu$ and
		$\nu$ on open $U\subseteq\R^N$ satisfies the small-volume $\p$-anisotropic isoperimetric
		condition \textup{(}in brief\textup{:} the small-volume $\p$-IC\textup{)} in $U$ with
		constant $C\in{[0,\infty)}$ if, for every $\eps>0$, there exists $\delta>0$ such that 
		\begin{equation}\label{eq:signed-SVIC}
			\mu(A^+)-\nu(A^1)\leq C\P_\p(A)+\eps 
			\qq\text{for all measurable }A\Subset U\, \text{ with } \ |A| < \delta. 
		\end{equation}
		As before, by the small-volume $\p$-IC for $\mu$ we mean the small-volume $\p$-IC for $(\mu,0)$, and in case of the standard integrand $\p(x,\xi)=|\xi|$ we speak of the small-volume isotropic IC.
	\end{defi}
	
	We emphasize, however, that the ICs imposed in Theorem \ref{thm:lsc} cannot be weakened to small-volume ICs.
	This contrasts with the parametric situation of \cite{Schmidt25} and is in fact demonstrated by the following example.
	
	\begin{examp}[failure of lower semicontinuity without IC]\label{exp:failure-lsc}
		For $N\ge2$, consider any measurable $A_k\Subset\Omega$, $k\in\N$, with disjoint closures and with $p_k\coleq\P(A_k)>0$ such that $\sum_{k=1}^\infty p_k<\infty$ \ka for instance, balls $A_k$ with $\P(A_k)=\delta/k^2$, where $\delta>0$ is
		chosen suitably small that they fit into $\Omega$\kz. Then, $\mu_-\coleq2\H^{N-1}\ecke\bigcup_{k=1}^\infty\partial^\ast\!A_k$ does not satisfy the isotropic IC in $\Omega$ with constant $1$ \ka since $\mu_-(A_k^+)=2\H^{N-1}(\partial^\ast\!A_k)=2\P(A_k)>\P(A_k)$ for all\/ $k\in\N$\kz, but does by \textup{\cite[Theorem 8.2]{Schmidt25}} still satisfy the small-volume isotropic IC in $\R^N$ with constant $1$. Moreover, $p_k^{-1}\1_{A_k}\in\BV(\Omega)$ converge in $\L^1(\Omega)$ to $0$
		\ka as $\|p_k^{-1}\1_{A_k}\|_{\L^1(\Omega)}=\P(A_k)^{-1}|A_k|
		\leq\mathrm{const}(N)\P(A_k)^\frac1{N-1}$ by the isoperimetric inequality \eqref{eq:isop_ineq}\kz, but for $\mu_+\colequiv0$, $u_0\colequiv0$, and the isotropic integrand $\p(x,\xi)=|\xi|$ lower semicontinuity of\/ $\widehat\Phi$ fails along this sequence \ka as $\widehat\Phi[p_k^{-1}\1_{A_k}]=p_k^{-1}\big[\P(A_k){-}2\H^{N-1}(\partial^\ast\!A_k)\big]={-}1$ for $k\in\N$ and\/ $\widehat\Phi[0]=0$\kz.
	\end{examp}
	
	It is interesting to compare Theorem \ref{thm:lsc} and Example \ref{exp:failure-lsc} in some detail with the results of Carriero--Leaci--Pascali \cite{CarLeaPas85,CarLeaPas86,CarLeaPas87}
	on functionals which are the sum of a standard first-order integral on $\W^{1,p}(\Omega)$, $p\in{[1,\infty)}$, and a term $\int_\Omega g(\,\cdot\,,w^\ast)\,\d\mu$ with the normal integrand $g$ and the Radon measure $\mu$ both non-negative. Specifically, \cite[Theorem 5.2]{CarLeaPas87} guarantees lower semicontinuity of such sum functionals on $\W^{1,p}(\Omega)$ under assumptions, which in case $p=1$, $g(\,\cdot\,,w^\ast)=w_\pm^\ast$ resemble \emph{not} the $\p$-IC of Theorem \ref{thm:lsc}, but rather the \emph{small-volume} $\p$-IC for $\mu$ with constant $1$. At first, this seems to be in remarkable contrast to Example \ref{exp:failure-lsc} (which can be adapted to the $\W^{1,1}$ case), but in fact the result depends strongly on the coincidence of the signs of $g$ and $\mu$ and can be applied to our present framework only in cases where the effect of Example \ref{exp:failure-lsc} cannot occur, namely for $\mu_-\equiv0$ and sequences $w_k$ uniformly bounded from below, or alternatively for $\mu_+\equiv0$ and sequences $w_k$ uniformly bounded from above. While it has already been observed in the parametric setting of \cite{Schmidt25} that the sign of the measure term indeed matters for the precise IC needed, we can now put on record that the same phenomenon manifests in the present non-parametric case as well. In any case, we conclude the comparison by emphasizing again that all results of \cite{CarLeaPas85,CarLeaPas86,CarLeaPas87} remain restricted to $\W^{1,p}(\Omega)$, while both the natural extension to $\BV(\Omega)$ and the treatment of non-parametric measure terms with arbitrary signs seem to be a novelty of the present work.

	In fact, the ICs imposed in Theorem \ref{thm:lsc} also prove to be sharp hypotheses for
	deducing coercivity of $\widehat\Phi$ and existence minimizers of
	$\widehat\Phi$. The main result in this direction follows.
	
	\begin{thm}[existence]\label{thm:exist}
          We consider $u_0\in\W^{1,1}(\R^N)$ and impose Assumptions
          \ref{assum:phi}\eqref{case:conv-lsc},\eqref{case:cont} and
          \ref{assum:mu}. If $(\mu_-,\mu_+)$ satisfies the $\p$-IC in
          $\Omega$ with constant $C<1$ and $(\mu_+,\mu_-)$ satisfies the
          $\widetilde\p$-IC in $\Omega$ with constant $C<1$, then the
          functional $\widehat\Phi$ in \eqref{eq:P-functional} has at
          least one minimizer in $\BV(\Omega)$. Moreover, if we additionally
          assume $u_0\in\L^\infty(\R^N)$, even in the extreme case with constant
          $C=1$ there exists at least one minimizer $u$ in
          $\BV(\Omega)\cap\L^\infty(\Omega)$ such that
          $\|u\|_{\L^\infty(\Omega)}\le\|u_0\|_{\L^\infty(\R^n)}$.
	\end{thm}
	
	The proof of Theorem \ref{thm:exist} is given in Section \ref{subsec:exist}.

    As a side remark we record that the continuity hypothesis of Assumption \ref{assum:phi}\eqref{case:cont} is not essential for Theorems \ref{thm:lsc} and \ref{thm:exist}. Indeed, this requirement is truly needed only in passing from $A\Subset\Omega$ as in \eqref{eq:signed-IC} to general $A\subseteq\Omega$; compare the later Theorem \ref{thm:anis_equivalence_measure_TOGETHER}. Accordingly,
    if we would directly allow all measurable subsets $A$ in \eqref{eq:signed-IC} instead of just the relatively compact ones, then all occurrences of Assumption \ref{assum:phi}\eqref{case:cont} in Section \ref{sec:exist} would drop out and Theorems \ref{thm:lsc}, \ref{thm:exist} could be obtained under Assumption \ref{assum:phi}\eqref{case:conv-lsc} only.
	
	We observe that the extreme case $C=1$ is also included in the results of
	\cite{MerSegTro08,MerSegTro09} for the isotropic case with zero boundary datum
	$u_0\equiv0$ (where $C=1$ is actually the most interesting case) and in
	comparison is included in Theorem \ref{thm:exist} in the wider generality of
	arbitrary anisotropies and boundary data
	$u_0\in\L^\infty(\R^N)$. Interestingly, the following example demonstrates that
	the inclusion of the case $C=1$ is no longer possible and one truly falls back to
	$C<1$ when the $\L^\infty$ assumption on $u_0$ is dropped.
	
	\begin{examp}[non-existence in extreme case with unbounded boundary datum]
		\label{exp:non-exist}
		Consider specifically the\linebreak two"~""di\-men\-sio\-nal Lipschitz domain $\Omega\coleq\{x\in\B_2:x_2>-1\}$, a boundary datum
		$u_0\in\W^{1,1}(\R^2)$ which extends $u_0(x)\coleq(|x|{-}1)^{-\alpha}$ for
		$x\in\B_3\setminus\overline\Omega$ with fixed $\alpha\in{(0,1/2)}$, and the
		measures $\mu_+\equiv0$ and $\mu_-=H\mathcal{L}^2$ on $\Omega$ with
		$H\in\bigcap_{p\in{[1,2)}}\L^p(\Omega)$ defined by $H(x)\coleq|x|^{-1}$ for
		$x\in\Omega\setminus\{(0,0)\}$. Then $\mu_-$ satisfies the isotropic IC in
		$\R^2$ \ka exactly\kz{} with constant $1$, but\/ $\widehat\Phi$
		with isotropic integrand $\p(x,\xi)=|\xi|$ has no minimizer in $\BV(\Omega)$.
		
		Furthermore, the preceding claims stay valid if in order to achieve even
		$H\in\L^\infty(\Omega)$ one keeps the choice $H(x)\coleq|x|^{-1}$ only for
		$x\in\Omega\setminus\overline{\B_1}$, but uses constant $H\colequiv2$ on $\B_1$.
	\end{examp}
	
	\begin{figure}[ht]
		\centering
		\includegraphics{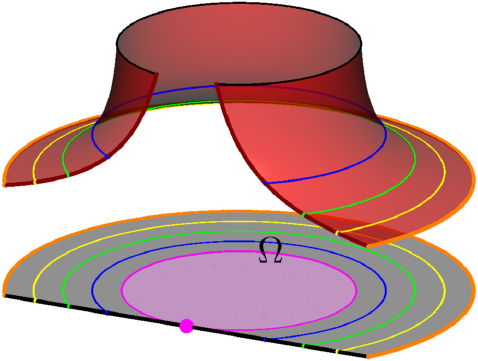}
		\vspace{-1.5ex}
		\parbox{10.3cm}{\caption{\centering The domain $\Omega$ of Example \ref{exp:non-exist} and the graph of the potential minimizer of $\widehat\Phi$ (red surface) with some level sets (colored arcs and pink disc) and with the position of the pole (pink dot).}\label{fig:non-ex}}
	\end{figure}
	
	We postpone the detailed treatment of Example \ref{exp:non-exist} to Section \ref{subsec:non-exist}, 
        but provide a rough explanation on the background idea already at this stage. 
        Indeed, the boundary values specified in
	Example \ref{exp:non-exist} are constant on the circular-arc portion of
	$\partial\Omega$, while on the flat portion of $\partial\Omega$ they are
	symmetric in $x_1$, but are unbounded from above and strictly increase towards a
	pole at $(0,-1)\in\partial\Omega$; cf.\@ Figure \ref{fig:non-ex}. Since one expects
	level curves of a (potentially existing) minimizer $u$ of $\widehat\Phi$ to have
	curvature $H$, from the rotational symmetry of $H$ around $(0,0)$ it is
	plausible that such $u$ should be constant on $\Omega\cap\partial\B_r$ for every
	$r\in{(1,2)}$, with the constant value prescribed at the two endpoints in
	$\partial\Omega$ of the arc $\Omega\cap\partial\B_r$; see Figure \ref{fig:non-ex}
	again. But then, since $u_0$ gets unbounded at $(0,-1)$ and $u$ should increase
	towards $(0,0)$, one is lead to having $u\equiv\infty$ on all of $\B_1$, and
	this is impossible for $u\in\BV(\Omega)$. We believe that this contradiction on
	the basis of heuristic expectations provides a fairly good background idea why
	existence fails in the situation of the example, but clearly we eventually turn
	this idea into a precise proof and then also verify suitable integrability of
	the boundary values $u_0$ (the latter in fact the reason for the restriction
	$\alpha<1/2$ in the above statement). \\

	Finally, our last main result shows that the values of $\widehat\Phi$ on $\BV(\Omega)$ can be recovered from those of the functional $\Phi$, defined in the introduction, on the Dirichlet class $\W^{1,1}_{u_0}(\Omega)$.

    \begin{thm}[existence of recovery sequences]\label{thm:recovery}
        We consider $u_0\in\W^{1,1}(\R^N)$ and impose Assumptions \ref{assum:phi}\eqref{case:conv-lsc},\eqref{case:cont} and \ref{assum:mu} with additionally $\mu_+\perp\mu_-$. Then, for every $u\in\BV(\Omega)$, there exists a recovery sequence $(u_k)_k$ in $\W^{1,1}_{u_0}(\Omega)$ such that $(\ol{u_k})_k$ converges to $\ol{u}$ strictly in $\BV(\R^N)$ and such that we have
        \[
          \lim_{k\to\infty}\Phi[u_k]=\widehat\Phi[u]\,,
        \]
		where $\Phi$ denotes the functional defined in \eqref{eq:P-pre-functional} with $\mu\coleq\mu_+-\mu_-$.
    \end{thm}
    
	The proof of Theorem \ref{thm:recovery} will be carried out in Section \ref{sec:recovery}.

    A direct consequence of Theorem \ref{thm:recovery} is the following coincidence of infimum values.

	\begin{cor}[consistency]\label{cor:consist}
		We consider $u_0\in\W^{1,1}(\R^N)$ and impose Assumptions \ref{assum:phi}\eqref{case:conv-lsc},\eqref{case:cont} and \ref{assum:mu} with additionally $\mu_+\perp\mu_-$. Then we have
		\[
          \inf_{\BV(\Omega)}\widehat\Phi
          =\inf_{\W^{1,1}_{u_0}(\Omega)}\Phi\,.
		\]
	\end{cor}
 
    In particular, Corollary \ref{cor:consist} shows that a minimizer of $\Phi$ in $\W^{1,1}_{u_0}(\Omega)$, if it should happen to exist, also minimizes $\widehat\Phi$ in $\BV(\Omega)$ --- a conclusion which in our opinion very consistently completes the existence theory.

    Another standard conclusion from Theorems \ref{thm:lsc} and \ref{thm:recovery} concerns the $\L^1(\Omega)$-relaxation $\Phi_{\mathrm{rel}}$ of $\Phi$, defined by
    \[
      \Phi_{\mathrm{rel}}[w]
      \coleq\inf\Big\{\liminf_{k\to\infty}\Phi[w_k]\,:\,\W^{1,1}_{u_0}(\Omega)\ni w_k\to w\text{ in }\L^1(\Omega)\Big\}
      \qq\text{for }w\in\BV(\Omega)\,,
    \]
    and in fact reads as follows.
    
    \begin{cor}[relaxation]\label{cor:relax}
      We consider $u_0\in\W^{1,1}(\R^N)$ and impose Assumptions \ref{assum:phi}\eqref{case:conv-lsc},\eqref{case:cont} and \ref{assum:mu} with additionally $\mu_+\perp\mu_-$. If $(\mu_-,\mu_+)$ satisfies the $\p$-IC in $\Omega$ with constant\/ $1$ and $(\mu_+,\mu_-)$ satisfies the $\widetilde\p$-IC in $\Omega$ with constant\/ $1$, then $\Phi_\mathrm{rel}$ coincides with $\widehat\Phi$ on $\BV(\Omega)$.
    \end{cor}

    We remark that the coincidence with $\widehat\Phi$ stays true for the relaxation of $\Phi$ with regard to any convergence ``in between'' the $u_0$-extended strict convergence of Theorem \ref{thm:recovery} and the $\L^1$-convergence of Theorem \ref{thm:lsc}. In particular, one may also use weak-$\ast$ convergence in $\BV(\Omega)$.
		
	The following example, whose details are addressed in Section \ref{subsec:non-consist}, shows that the assumption $\mu_+\perp\mu_-$ in Theorem \ref{thm:recovery} and Corollaries \ref{cor:consist}, \ref{cor:relax} cannot be dropped. The main reason behind is that, in the case considered, $\widehat\Phi$ ``sees'' the single measures $\mu_+$ and $\mu_-$, while $\Phi$ involves their difference $\mu=\mu_+-\mu_-$ only.
		
	\begin{examp}[failure of consistency for $\mu_+$ and $\mu_-$ not singular to each other]\label{exp:non-consist}
	   We consider the unit disc $\Omega\coleq\B_1\subseteq\R^2$, a boundary datum $u_0\in\W^{1,1}(\R^2)$ with trace $u_0(x)=\mathrm{sgn}(x_1)$ for $x\in\partial\B_1$, and the measures $\mu_+\coleq\mu_-\coleq\H^1\ecke(\{0\}{\times}{(-1,1)})$. Then, for the isotropic integrand $\p(x,\xi)=|\xi|$, we have $\inf_{\W^{1,1}_{u_0}(\B_1)}\Phi=4>\min_{\BV(\B_1)}\widehat\Phi=0$ and\/ $\Phi_{\mathrm{rel}}[w]=\TV^{u_0}[w]=\widehat\Phi[w]+|\D w|(\{0\}{\times}{({-}1,1)})$ for $w\in\BV(\Omega)$. Additionally, we record that $\mu\coleq\mu_+=\mu_-$ considered here satisfies the isotropic IC in $\Omega$ with constant $\frac12$ \ka and this in particular implies the same IC for the pair $(\mu,\mu)$ and the validity of Assumption \ref{assum:mu}\kz.
	\end{examp}

    However, the case $\mu_+\not\perp\mu_-$ may also come with cancellation effects, which are captured neither by $\Phi$ nor by $\widehat\Phi$ and which thus do not conflict with validity of the above results. Though after all our main interest is still in the opposite case $\mu_+\perp\mu_-$ of a signed measure, at least an extreme case with $\mu_+=\mu_-$ and complete cancellation still seems worth a brief recording:

    \begin{examp}[cancellation of purely unrectifiable measures]\label{examp:unrectifiable}
        Consider measures $\mu_+\coleq\mu_-\coleq\theta\H^{N-1}\ecke P$ with any purely $\H^{n-1}$-unrectifiable $P\subset\Omega$ and any density $\theta\in\L^1(P;\H^{N-1})$.
        Then, for each measurable $A\Subset\Omega$ with $\P(A)<\infty$, countable $\H^{N-1}$-rectifiability of\/ $A^+\!\setminus\!A^1$ \textup{(}see e.\@g.\@ \textup{\cite[Theorems 3.59, 3.61]{AFP00})} enforces $\H^{N-1}(P\cap(A^+\!\setminus\!A^1))=0$ and $\mu_\mp(A^+)-\mu_\pm(A^1)=\int_{P\cap(A^+\setminus A^1)}\theta\,\d\H^{N-1}=0$. Analogously, for $w\in\BV(\Omega)$, countable $\H^{N-1}$-rectifiability of\/ $\{w^+\!\neq\!w^-\}$ \textup{(}see e.\@g.\@ \textup{\cite[Theorem 3.78]{AFP00})} enforces $\H^{N-1}(P\cap\{w^+\!\neq\!w^-\})=0$ and $\int_\Omega w^-\,\d\mu_+-\int_\Omega w^+\,\d\mu_-=\int_{P\cap\{w^+\neq w^-\}}(w^-{-}w^+)\theta\,\d\H^{N-1}=0$. Thus, in such cases, the ICs for $(\mu_-,\mu_+)$ and $(\mu_+,\mu_-)$ are trivially satisfied, and the $\mu_\pm$-terms in $\widehat\Phi$ entirely cancel out. Moreover, we point out that there exist non-zero unrectifiable measures $\theta\H^{N-1}\ecke P$ which also meet the requirements of Assumption \ref{assum:mu}, for instance the measure of Example \ref{exp:til2} below.
    \end{examp}
	
	Finally, we return to the discussion of the $\p$-ICs used above and address
    a technical subtlety of our framework with possibly non-even $\p$. Indeed, we
	demonstrate by the following example that there is a point in distinguishing
	between the $\p$-IC with constant $1$ and the $\widetilde\p$-IC with constant
	$1$ in the above results, since the one does not necessarily imply the other.
	
	\begin{examp}[the $\p$-IC does not imply the $\widetilde\p$-IC]\label{exp:til1}
		Consider $\p\colon\R^2\times\R^2\to{[0,\infty)}$ defined by setting
		\[
		\p(x,\xi)\coleq\begin{cases}|\xi|&\text{if }\xi_2\ge0\\|\xi_1|{+}|\xi_2|&\text{if }\xi_2\le0\end{cases}
		\]
		and the unit triangle $\Delta\coleq\{x\in{[0,\infty)}^2:x_1{+}x_2\le1\}$ with
		catheti $C_1$, $C_2$ and hypotenuse $H$, for which we have
		$\p(\,\cdot\,,\nu_\Delta)=1$ on $C_1\cup C_2$,
		$\p(\,\cdot\,,\nu_\Delta)=\sqrt2$ on $H$,
		$\widetilde\p(\,\cdot\,,\nu_\Delta)=1$ on $\partial\Delta=C_1\cup C_2\cup H$.
		Then the $\p$-perimeter measure
		$\P_\p(\Delta,\,\cdot\,)=\H^1\ecke(C_1\cup C_2)+\sqrt2\,\H^1\ecke H$
		satisfies the $\p$-IC in $\R^2$ with constant $1$ \ka compare the earlier discussion and see Section \ref{subsec:til1}\kz, but does \emph{not} satisfy the $\widetilde\p$-IC in $\R^2$ with constant $1$ \ka as shown by
		$\P_\p(\Delta,\Delta^+)=\P_\p(\Delta)=4>2{+}\sqrt2=\P_{\widetilde\p}(\Delta)$\kz.
	\end{examp}
	
	Beyond that, it turns out that the $\p$-IC with constant $1$ does not even imply the small-volume $\widetilde\p$-IC with constant $C$. In fact, we have the following refined example, whose full details will be taken up in Section \ref{subsec:til2}.
	
	\begin{examp}[the $\p$-IC does not even imply the small-volume
		$\widetilde\p$-IC]\label{exp:til2}
		Consider $\p$ and $\Delta_0\coleq\Delta$ from Example \ref{exp:til1}.
		The three similarities $T_i\colon\R^2\to\R^2$ given by
		$T_1(x)\coleq(x{+}(0,2))/3$, $T_2(x)\coleq x/3$, $T_3(x)\coleq(x{+}(2,0))/3$
		generate the iterates $\Delta_k\coleq\bigcup_{i=1}^3T_i(\Delta_{k-1})$,
		$k\in\N$, and the fractal $\Delta_\infty\coleq\bigcap_{k=0}^\infty\Delta_k$
		with $\H^1(\Delta_\infty)=\sqrt2$\textup{;} cf.\@ Figure \ref{fig:til2}. Then
		the Radon measure $\mu\coleq 2\sqrt2\,\H^1\ecke{\Delta_\infty}$ satisfies the
		$\p$-IC in $\R^2$ with constant $1$, but in view of
		$\mu(\Delta_k^+)=4>2{+}\sqrt2=\P_{\widetilde\p}(\Delta_k)$ for all\/
		$k\in\N_0$ the $\widetilde\p$-IC with constant\/ $1$ falls short uniformly for
		the sets $\Delta_k$, which can realize arbitrarily small volumes
		$|\Delta_k|=\frac124^{-k}$.
	\end{examp}
	
	\begin{figure}[H]
		\centering
		\includegraphics{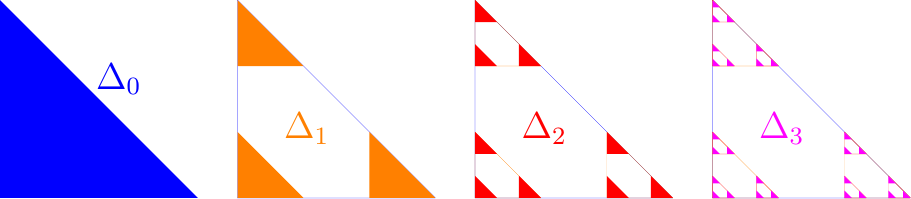}
		\caption{The iterates \textcolor{blue}{$\Delta_0$}, \textcolor{orange}{$\Delta_1$},
			\textcolor{red}{$\Delta_2$}, and \textcolor{magenta}{$\Delta_3$} of Example
			\ref{exp:til2}\label{fig:til2}}
	\end{figure}
	
	We believe that the effect found in Example \ref{exp:til2} indeed occurs
	solely on fractals in the sense that it can be excluded for measures $\mu_\pm=\theta_\pm\H^{N-1}\ecke S$ with countably $\H^{N-1}$-rectifiable $S$ and $\theta_\pm\in\L^1(S;\H^{N-1})$. This is supported by the expectation that in
	analogy with \cite[Section 8]{Schmidt25} the small-volume $\p$-IC with constant $1$ for such $(\mu_+,\mu_-)$ and suitably good $\p$ should be satisfied precisely if
	$\theta_+\leq\p(\,\cdot\,,\nu_S){+}\widetilde\p(\,\cdot\,,\nu_S)$ holds
	$\H^{N-1}$-\ae{} on $S$, whereas the passage from $\p$ to $\widetilde\p$
	plays no role. However, we have not explored such issues in full detail.

	\section{Reformulations of the ICs}\label{sec:charact}
	
	We recall that $\Omega$ generally denotes a bounded open set with Lipschitz boundary in $\R^N$, $N\in\N$. In this section we collect various characterizations of ICs in $\Omega$.
	
	\subsection{IC characterization of admissible non-negative measures}
	
	We first verify the claim made in Section \ref{sec:statements} that the requirements of Assumption \ref{assum:mu} can be read as ICs for $\mu_\pm$ with arbitrarily large constant.
	
	\begin{prop}[characterization of admissible measures]
		\label{prop:admissible}
		A non-negative Radon measure $\mu$ on $\Omega$ is admissible in the sense of \eqref{eq:negligible} and \eqref{eq:finite-integral} from Assumption \ref{assum:mu} \ka clearly with $\mu$ in place of $\mu_\pm$\kz{} if and only if $\mu$ satisfies the isotropic IC \ka or equivalently any $\p$-IC with \eqref{eq:comp-p0} in force\kz{} with some constant $C\in{[0,\infty)}$ in $\Omega$. 
	\end{prop}
	
	\begin{proof}
		We first suppose that $\mu$ fulfills \eqref{eq:negligible} and \eqref{eq:finite-integral}, and we assume for contradiction that the isotropic IC does \emph{not} hold for $\mu$ with any finite constant in $\Omega$. Since we know from \cite[Theorem 7.5]{Schmidt25} that the IC can be equivalently phrased with $A^1$ instead of $A^+$, we can then select a sequence of sets $(A_k)_k$ with $A_k\Subset\Omega$, $\P(A_k) > 0$, and
		\begin{equation} \label{eq:ass_contr}
			\mu(A_k^1) > k \P(A_k) \qq \text{ for all } k \in \N\,.
		\end{equation}
		Specifically, by \eqref{eq:ass_contr} and the finiteness of $\mu$, we find some $\ol{k} \in \N$ such that $\P(A_k) < 1$ for all $k \geq \ol{k}$. In the case $N=1$, this directly contradicts the standard observation that sets $A\Subset\R$ of positive perimeter have even $\P(A)\ge2$. In the main case $N\ge2$, instead, we reason as follows. For any $k$, we fix $\alpha_k \coleq 1/(k^2 \P(A_k))$ and consider $u \coleq \sum_{k =1}^{\infty} \alpha_k \1_{A_k}$. Then, from the isoperimetric inequality \eqref{eq:isop_ineq} we infer
		\begin{align*}
			\int_{\Omega} \alpha_k \1_{A_k}\,\dx 
			= \frac{|A_k|}{k^2 \P(A_k)}
			\leq c_N \frac{\P(A_k)^{1/(N-1)}}{k^2} 
			\leq c_N \frac{1}{k^2} \qq \text{for all } k \geq \ol{k}
		\end{align*}
        with $c_N \coleq (N^N \omega_N)^{-1/(N-1)}$.
		Thus, $u$ is well-defined and $\|u\|_{\L^1(\Omega)}=\sum_{k =1}^{\infty} \int_{\Omega} \alpha_k \1_{A_k}\,\dx < \infty$. Furthermore, we compute
		\begin{align*}
			|\D u|(\Omega)
			\leq \sum_{k =1}^{\infty} \alpha_k |\D \1_{A_k}|(\Omega)
			= \sum_{k =1}^{\infty} \alpha_k \P(A_k) 
			= \sum_{k =1}^{\infty} \frac{1}{k^2} <\infty\,,
		\end{align*}
		and hence we have $u \in \BV(\Omega)$. Finally, taking into account that $u^+\ge u^-\ge\sum_{k=1}^\infty\alpha_k(\1_{A_k})^-$ by \eqref{eq:sub-super-add-representatives} and applying \eqref{eq:ass_contr} we get
		\begin{align*}
			\int_{\Omega} u^+\,\d\mu 
			\ge \sum_{k =1}^{\infty} \alpha_k \int_{\Omega} (\1_{A_k})^- \,\d\mu 
			= \sum_{k =1}^{\infty} \alpha_k \mu(A_k^1) 
			\geq \sum_{k =1}^{\infty} \frac{k \P(A_k)}{k^2 \P(A_k)} 
			= \sum_{k=1}^\infty\frac1k
			= \infty\,,
		\end{align*}
		and thus we reach a contradiction to the validity of \eqref{eq:finite-integral} for non-negative $\BV$ functions.
		
		\smallskip
		
		For the proof of the opposite implication, suppose that $\mu$ satisfies the isotropic IC with some constant $C$ on $\Omega$. Then, from \cite[Theorem 7.5]{Schmidt25} we obtain \eqref{eq:negligible} and $\int_\Omega|v^\ast|\,\d\mu\leq C\int_\Omega|\nabla v|\dx$ for all $v\in\W^{1,1}_0(\Omega)$. Next, even for arbitrary $v\in\W^{1,1}(\Omega)$, we conclude $\int_\Omega|v^\ast|\,\d\mu<\infty$ by using strict approximations $v_k$ in $\W^{1,1}_0(\Omega)$ of $v$ of the type from Theorem \ref{strict_int_appr_BV}, reading off $\H^{N-1}$-\ae{} convergence $v_{k_\ell}^\ast\to v^\ast$ on $\Omega$ of a subsequence from Theorem \ref{strict_Hausdorff_repr}, and then deducing with Fatou's lemma that $\int_\Omega|v^\ast|\,\d\mu\leq\liminf_{\ell\to\infty}\int_\Omega|v_{k_\ell}^\ast|\,\d\mu \leq C\limsup_{k\to\infty}\int_\Omega|\nabla v_k|\dx<\infty$. Finally,
		since every $v\in\BV(\Omega)$ satisfies $v\leq\widetilde v$ for some $\widetilde v\in \W^{1,1}(\Omega)$, we conclude $\int_\Omega v^+\,\d\mu<\infty$ for all non-negative $v\in\BV(\Omega)$ and thus arrive at \eqref{eq:finite-integral}.
	\end{proof}
	
	\subsection{Anisotropic ICs for general pairs of measures}\label{subsec:charact_non_sing}
	
	Next we deal with several formulations of ICs which are basically well-known and express these conditions by either testing with sets (type 1), testing with functions (type 2), or in a distributional sense (type 3). However, in the characterization results of \cite{MeyZie77,PhuTor08,PhuTor17} the constant $C$ involved may very from one condition to another, while here with our actual focus on $C=1$ we wish to fix one precise constant $C$. In other words, when viewing optimal constants $C$ as a kind of dual norms, the previous results yield equivalence, while we will establish actual equality. In fact, this is also the reason for not considering yet another type of characterization via bounds for the density ratio here, since in this regard the preservation of the constant seems difficult. While a closely related result for the isotropic perimeter and a non-negative measure has already been recorded in \cite[Theorem 7.5]{Schmidt25}, here we will treat general anisotropic perimeters and ``signed pairs'' $(\mu_-,\mu_+)$. We emphasize that in Theorem \ref{thm:anis_equivalence_measure_TOGETHER} directly below we do \emph{not} assume that $\mu_+$ and $\mu_-$ are mutually singular, but rather we postpone specifics of the case $\mu_+\perp\mu_-$ to the subsequent Theorem \ref{thm:anis_equivalence_gen_IC_singular}.
	
	\begin{thm}[characterizations of $\p$-anisotropic ICs, first version] \label{thm:anis_equivalence_measure_TOGETHER}
		We impose Assumptions \ref{assum:phi}\eqref{case:cont} and \ref{assum:mu}, and fix $C >0$. Then the following are \textbf{equivalent}\textup{:}%
		\newcounter{enum0}\refstepcounter{enum0}
		\begin{enumerate}[{\phantom{1}\rm(a)}]
			\renewcommand\theenumi{\arabic{enum0}\alph{enumi}}
			\item\label{item:1a}The pair $(\mu_-,\mu_+)$ satisfies the $\p$-IC with constant $C$ and $(\mu_+,\mu_-)$ satisfies the $\widetilde\p$-IC with constant $C$, that is,  
			\[
			-C\P_{\widetilde\p}(A)
			\leq  \mu_-(A^1)-\mu_+(A^+) 
			\leq \mu_-(A^+)-\mu_+(A^1) 
			\leq  C\P_{\p}(A)
			\]
			for all measurable $A\Subset \Omega$ \ka with emphasis on $A\Subset \Omega$\kz.
			\item\label{item:1b}The pair $(\mu_-,\mu_+)$ satisfies 
			\[
			-C\P_{\widetilde\p}(A) 
			\leq  \mu_-(A^1 \cap \Omega)-\mu_+(A^+ \cap \Omega) 
			\leq  \mu_-(A^+ \cap \Omega)-\mu_+(A^1 \cap \Omega) 
			\leq  C\P_{\p}(A)
			\]
			for all measurable $A \subseteq \Omega$ \ka now with emphasis on $A\subseteq \Omega$\kz.
		\end{enumerate}
		\refstepcounter{enum0}
		\begin{enumerate}[{\phantom{2}\rm(a)}]
			\renewcommand\theenumi{\arabic{enum0}\alph{enumi}}
			\item\label{item:2a}We have \ka where the first and the last inequality are the non-trivial ones\kz
			\begin{multline} \label{2A}
				-C \left( |\D v|_{\widetilde{\p}}(\Omega)+\int_{\partial \Omega} \widetilde{\p}(\,\cdot\,,  v \nu_\Omega) \, \d\H^{N-1}\right) 
				\leq \int_{\Omega} v^- \ \d\mu_- - \int_{\Omega} v^+ \ \d\mu_+  \\
				\leq \int_{\Omega} v^+ \ \d\mu_- - \int_{\Omega} v^- \ \d\mu_+  
				\leq C \left(  |\D v|_{\p}(\Omega) + \int_{\partial \Omega} \p(\,\cdot\,,  v \nu_\Omega) \, \d\H^{N-1}\right)
			\end{multline}
			for all \textbf{non-negative} $v \in\BV(\Omega)$. \label{eq:cond_BV}
		\end{enumerate}
		In addition, each of the properties above \textbf{implies} the subsequent \textbf{equivalent} conditions\textup{:}
		\begin{enumerate}[{\phantom{2}\rm(a)}]	\renewcommand\theenumi{\arabic{enum0}\alph{enumi}}
			\refstepcounter{enumi}
			\item\label{item:2b}It holds 
			\begin{equation} \label{2B'} 
				- C \int_{\Omega} \widetilde\p(\,\cdot\,,\nabla v) \dx
				\leq \int_{\Omega} v \ \d\mu_-  - \int_{\Omega} v \ \d\mu_+ 
				\leq  C\int_{\Omega} \p(\,\cdot\,,\nabla v) \dx
			\end{equation}
			for all \textbf{non-negative} $v \in \C^{\infty}_\c(\Omega)$.
			\item\label{item:2c}It holds
			\begin{equation} \label{2B''} 
				- C \int_{\Omega} \widetilde\p(\,\cdot\,,\nabla v) \dx
				\leq \int_{\Omega} v^\ast \ \d\mu_-  - \int_{\Omega} v^\ast \ \d\mu_+ 
				\leq  C\int_{\Omega} \p(\,\cdot\,,\nabla v) \dx
			\end{equation}
			for all \textbf{non-negative} $v \in \W^{1,1}_0(\Omega)$.
		\end{enumerate}
        Furthermore, if we additionally require convexity in the sense that also \eqref{case:conv-lsc} from Assumption \ref{assum:phi} is in force, then \eqref{item:2b} and \eqref{item:2c} are equivalent with the following property\textup{:}
		\begin{enumerate}[{\rm(1)}]
			\refstepcounter{enumi}\refstepcounter{enumi}
			\item\label{item:3}There exists a vector field $\sigma \in \L^{\infty}(\Omega,\R^N)$ such that $\p^\circ(x,\sigma(x)) \leq C$ holds for \ae{} $x \in \Omega$ and 
			\[
			\div(\sigma) = \mu_+-\mu_-
			\qq\text{holds in the sense of distributions on }\Omega\,. 
			\]
		\end{enumerate}
	\end{thm}
	
	Before proving the theorem, we briefly record that \eqref{item:1a}, \eqref{item:1b}, \eqref{item:2a} ``see'' 
    to some extent the single measures $\mu_\pm$, while \eqref{item:2b}, \eqref{item:2c}, \eqref{item:3} depend on the signed measure $\mu_+{-}\mu_-$ only. Therefore, in the general case without an additional assumption such as $\mu_+\perp\mu_-$, there is no chance of getting back from \eqref{item:2b}, \eqref{item:2c}, \eqref{item:3} to \eqref{item:1a}, \eqref{item:1b}, \eqref{item:2a}.
	
	\begin{proof}
		First we verify that \eqref{item:1a} implies \eqref{item:1b}. 
		We suppose \eqref{item:1a} holds and consider a measurable $A \subseteq \Omega$, where in view of \eqref{eq:comp-p0} we assume without loss of generality $\P(A)<\infty$ and thus $\1_A \in \BV(\R^N) \cap \L^{\infty}(\R^N)$. From \cite[Proposition 4.1]{Schmidt15}, we obtain an increasing sequence $(\Omega_k)_k$ of open sets such that $\Omega_k \Subset \Omega$, $\P(\Omega_k) < \infty$ for all $k$, and $\Omega_k \to \Omega$ in measure as $k \to \infty$ with 
		\begin{equation}  \label{d}
			\int_{\partial^\ast \Omega_k} \left| \innk(\1_A) \right|\, \d\H^{N-1}
			\leq \int_{\partial \Omega}\left| \inn(\1_A)\right|\,\d\H^{N-1} + \frac{1}{k}, \qq \text{ for all } k \in \N\,.
		\end{equation}
		In view of $\bigcup_{k \in \N} \left( {A \cap \Omega_k}\right) ^+= A^+ \cap \Omega$ and $\bigcup_{k \in \N} \left( {A \cap \Omega_k}\right) ^1= A^1 \cap \Omega$, a standard property of measures yields
		\begin{equation}\label{eq:mu-conv-exhaustion}
			\mu_\pm(\left( A \cap \Omega_k \right) ^+) \to \mu_\pm(A^+ \cap \Omega)
			\qq\text{and}\qq
			\mu_\pm(\left( A \cap \Omega_k \right) ^1) \to \mu_\pm(A^1 \cap \Omega)
			\qq\text{as }k \to \infty\,.
		\end{equation}Moreover, we can apply Theorem \ref{thm:dec_BV} and exploit $\Omega^1=\Omega$, $\H^{N-1}(\partial\Omega\triangle\partial^\ast\Omega)=0$ for the Lipschitz domain $\Omega$ to get
		\begin{align*}
			\D \1_{A \cap \Omega_k} = \D \1_A \mres {(\Omega_k)^1} +\innk(\1_A) \nu_{\Omega_k} \H^{N-1} \mres \partial^\ast \Omega_k\,,
			\qq
			\D  \1_{A} = \D \1_A \mres \Omega +\inn(\1_A)  \nu_{\Omega} \H^{N-1} \mres \partial \Omega\,.
		\end{align*}
		Hence, for every $k$, applying \eqref{d} we obtain
		\begin{align*} 
			\P(A \cap \Omega_k) 
			= |\D\1_{A \cap \Omega_k}|(\R^N) 
			&= |\D\1_A|\big((\Omega_k)^1\big) + 	\int_{\partial^\ast \Omega_k} \left| \innk \1_A  \right|\,\d \H^{N-1} \\
			& \leq  \P(A,\Omega) + \int_{\partial \Omega}\left| \inn \1_A \right|\,\d\H^{N-1} + \frac{1}{k}
			=|\D\1_A|(\R^N)+\frac1k
			=\P(A)  + \frac{1}{k}\,.
		\end{align*}
		As $A \cap \Omega_k \to A$ in measure, the previous inequality together with the lower semicontinuity of the perimeter ensures 
		$\P(A \cap \Omega_k) \to \P(A)$ and $\1_{A \cap \Omega_k} \to \1_{A}$ strictly in $\BV(\R^N)$. As a consequence, Theorem \ref{anis_thm:Reshetnyak_cont} establishes $\P_{\p}(A \cap \Omega_k) \to \P_{\p}(A)$ for $k \to \infty$ and the analogous convergence for the mirrored integrand $\widetilde{\p}$. Finally, relying on \eqref{eq:mu-conv-exhaustion}, on assumption \eqref{item:1a} for the sets $A \cap \Omega_k \Subset \Omega$, and on the convergences just observed, we conclude
		\[
		\mu_-(A^+ \cap \Omega)-\mu_+(A^1 \cap \Omega)
		= \lim_{k \to \infty} \left( \mu_- \left( (A \cap \Omega_k)^+ \right) -  \mu_+ \left( (A \cap \Omega_k)^1 \right) \right) 
		\leq C \lim_{k \to \infty} \P_{\p}(A \cap  \Omega_k)
		= C \ \P_{\p}(A)\,.
		\]
		The estimate from below involves $\P_{\widetilde\p}$, but otherwise is analogous, and thus we arrive at \eqref{item:1b}.
		
		\smallskip
		
		In order to check that \eqref{item:1a} follows from \eqref{item:2a}, it suffices to plug $v\coleq\1_A$ with measurable $A\Subset\Omega$ into \eqref{2A}, where again we may suppose $\1_A \in\BV(\Omega)$. Taking into account $(\1_A)^+ = \1_{A^+}$ and $(\1_A)^- = \1_{A^1}$ in $\Omega$,
		the conclusion is then straightforward.
		
		\smallskip
		
		Now we turn to the proof that \eqref{item:1b} implies \eqref{item:2a}.
		We extend the measures $\mu_\pm$ by zero outside $\Omega$. Moreover, we consider $0\leq v\in\BV(\Omega)$, extend it to $\ol{v} \in \BV(\R^N)$ such that $\ol{v}=0$ on $\R^N \setminus \ol{\Omega}$, and then deduce from the layer-cake formula and Remark \ref{rem:equivalence_sets_wrt_mu_ext} that we have
		\[
		\int_{\Omega} v^+ \d\mu_-  - \int_{\Omega} v^- \d\mu_+
		=   \int_{0}^{\infty} \Big( \mu_- \left( \{ \ol{v} > t \} ^+ \cap \Omega \right)- \mu_+ \left( \{ \ol{v} > t \} ^1  \cap \Omega \right) \Big) \dt\,.
		\]
		We may now apply the hypothesis \eqref{item:1b} to the superlevel sets $\{ \ol{v} > t \} \subseteq \Omega$ for \ae{} $t \geq 0$ and by the $\p$-coarea formula \eqref{anis_eq:coarea} for non-negative functions on $\Omega$ obtain
		\[
		\int_{\Omega} v^+ \d\mu_-  - \int_{\Omega} v^- \d\mu_+ 
		\leq C \int_{0}^{\infty} \P_{\p}\left( \left\{ \ol{v} \geq t \right\} \right) \dt 
		= C |\D\ol{v}|_{\p}(\R^N)\,,
		\]
		where as usual we can rewrite $|\D\ol{v}|_{\p}(\R^N)=|\D v|_{\p}(\Omega) + \int_{\partial \Omega}\p(\,\cdot\,,v \nu_\Omega) \ \d\H^{N-1}$.
		Together with an analogous lower estimate this establishes \eqref{item:2a}.
		
		\smallskip
		
		Next we record that \eqref{item:2a} implies \eqref{item:2c} and \eqref{item:2c} implies \eqref{item:2b} by straightforward specialization from $\BV(\Omega)$ to $\W^{1,1}_0(\Omega)$ and from $\W^{1,1}_0(\Omega)$ to $\C^\infty_\c(\Omega)$, respectively.
		
		\smallskip
		
		In a further step, we prove that \eqref{item:2b} implies \eqref{item:2c}. Given $0\leq v\in\W^{1,1}_0(\Omega)$, for any $M >0$ we consider the truncation $v^M \in \W^{1,1}_0(\Omega) \cap \L^{\infty}(\Omega)$. Then there is a sequence $({v}^M_k)_k$ in $\C^\infty_\c(\Omega)$ such that ${v}^M_k \to v^M$ strongly in $\W^{1,1}(\Omega)$ for $k \to \infty$.
		Since $0\leq v^M\leq M$ on $\Omega$, by truncation of $v^M_k$ at level $M$ and mollification we can assume that $0\leq v^M_k\leq M$ on $\Omega$ for all $k\in\N$.
		Moreover, strong convergence and Theorem \ref{strict_Hausdorff_repr} give that (up to relabeling a subsequence) ${v}^M_k \to (v^M)^\ast$ pointwise $\H^{N-1}$-\ae{} in $\Omega$ and hence also pointwise $\mu_\pm$-\ae{} in $\Omega$. Applying the dominated convergence theorem to the sequence $({v}^M_k)_k$ in $\L^1(\Omega,\mu_\pm)$, we get 
		\[
		\lim_{k\to\infty}\int_{\Omega} {v}^M_k\ \d \mu_\pm
		=\int_{\Omega}  (v^M)^\ast \d \mu_\pm\,.
		\]
		In addition, by Theorem \ref{anis_thm:Reshetnyak_cont} the strong convergence of $v^M_k$ in $\W^{1,1}(\Omega)$ induces their $\p$- and $\widetilde\p$-strict convergence. Then, we can apply \eqref{2B'} to $v^M_k$ and exploit the preceding convergences in passing to the limit. We achieve
		\begin{equation} \label{eq:vM1}
			- C\int_{\Omega} \widetilde\p(\,\cdot\,,\nabla v^M) \dx
			=  \int_{\Omega} (v^M)^\ast \,\d\mu_-  - \int_{\Omega} (v^M)^\ast  \,\d\mu_+ \\
			\leq C \int_{\Omega} \p(\,\cdot\,,\nabla v^M) \dx\,.
		\end{equation}
		It remains to send $M\to\infty$. From Lemmas \ref{lem:cut_M}\eqref{item:cut_M_iii} and \ref{lem:int-uM-to-int-u} we know $v^M \to v$ strongly in $\W^{1,1}(\Omega)$ and
		\[
		\lim_{M\to\infty}\int_{\Omega} (v^M)^\ast \d \mu_\pm
		= \int_{\Omega}  v^\ast \d \mu_\pm \,.
		\]
		Therefore, passing to the limit in \eqref{eq:vM1}, we infer
		\[
		- C \int_{\Omega} \widetilde\p(\,\cdot\,,\nabla v) \dx
		\leq \int_{\Omega} v^\ast \d\mu_-  - \int_{\Omega} v^\ast  \d\mu_+
		\leq C \int_{\Omega} \p(\,\cdot\,,\nabla v) \dx\,.
		\]
		Hence, \eqref{item:2c} is proved.
		
		\smallskip
		
		If \eqref{item:3} holds, we consider the vector field $\sigma \in \L^{\infty}(\Omega,\R^N)$ such that $\p^\circ(x,\sigma(x)) \leq C$ for \ae{} $x \in \Omega$ and $\div(\sigma)=\mu_+-\mu_-$ on $\Omega$. 
		From definition of polar function $\p^\circ$, we have
		\[
		\sigma(x)\ip\xi\leq\p^\circ(x,\sigma(x))\,\p(x,\xi)\leq C\,\p(x,\xi)
		\qq \text{for all } x \in \Omega\,,\, \xi \in \R^N\,,
		\]
		and by change of sign in $\xi$ we obtain in fact
		\begin{equation} \label{eq:polar_varphi_dim1}
			- C \, \widetilde\p(x,\xi)
			\leq \sigma(x) \cdot \xi 
			\leq  C \, \p(x,\xi)
			\qq \text{for all } x \in \Omega\,,\, \xi \in \R^N\,. 
		\end{equation}
		Integrating any non-negative function $v \in \C^{\infty}_\c(\Omega)$ on $\Omega$ and exploiting the equality $\div(\sigma)=\mu_+-\mu_-$, we have
		\begin{align*}
			\int_{\Omega} v \ \d\mu_-  -  \int_{\Omega} v \ \d\mu_+ 
			= - \int_{\Omega} v \ \d(\mu_+{-}\mu_-) 
			= - \int_{\Omega} v \ \d(\div(\sigma)) 
			= \int_{\Omega} \sigma \cdot \nabla v \,\dx\,,
		\end{align*}
		and then with the help of \eqref{eq:polar_varphi_dim1} we arrive at the estimate in \eqref{item:2b}.
		
		\smallskip
		
		Finally, we assume \eqref{item:2c} and will deduce \eqref{item:3}.
		To this end, we first claim that the inequalities \eqref{2B''} from \eqref{item:2c} stay valid for functions $v\in\W^{1,1}_0(\Omega)$ of arbitrary sign. To verify this, we apply \eqref{item:2c} to $v_\pm\ge0$ and read off
		\begin{align*}
			-C\int_\Omega\widetilde\p(\,\cdot\,,\nabla v_+)\dx
			&\leq\int_\Omega v_+^\ast\,\d(\mu_-{-}\mu_+)
			\leq C\int_\Omega\p(\,\cdot\,,\nabla v_+)\dx\,,\\
			C\int_\Omega\widetilde\p(\,\cdot\,,{-}\nabla v_-)\dx
			&\ge\int_\Omega v_-^\ast\,\d(\mu_-{-}\mu_+)
			\ge -C\int_\Omega\p(\,\cdot\,,{-}\nabla v_-)\dx\,.
		\end{align*}
		We subtract the second line from the first and take into account $\nabla v=\nabla v_+$, $\nabla v_-\equiv0$ \ae{} on $\{v\ge0\}$ and $\nabla v={-}\nabla v_-$, $\nabla v_+\equiv0$  \ae{} on $\{v\le0\}$. Then we indeed arrive at \eqref{2B''} for $v$ of arbitrary sign.

		Next we consider the subspace 
		\[
		   \mathcal{W} \coleq \big\{ \nabla v \, : \, v \in \W^{1,1}_0(\Omega) \big\}
		   \subseteq\L^1\big(\Omega,\R^N\big)
		\]
		and define a linear functional $\F$ on $\mathcal{W}$ such that
		\[
		   \F[W]=\F[\nabla v]
		   \coleq {-}\int_{\Omega} v^\ast \ \d(\mu_+{-}\mu_-)
		   = \int_{\Omega} v^\ast \, \d\mu_- - \int_{\Omega} v^\ast \, \d\mu_+
		   \qq \text{for all } W=\nabla v \in \mathcal{W}\,.
		\]
		Let now $\I_{\p}\colon \L^1(\Omega,\R^N) \to \ol{\R}$ be defined through
		\[
		   \I_{\p}[W]\coleq C \int_{\Omega} \p(\,\cdot\,,W) \dx
		\qq\text{for all } W \in \L^1(\Omega,\R^N)\,.
		\]
		At this point we decisively exploit the convexity from Assumption \ref{assum:phi}\eqref{case:conv-lsc} to ensure that $\I_{\p}$ is sub-linear. Moreover, by the arbitrary-sign version of \eqref{2B''} it is $\F \leq \I_{\p}$ on $\mathcal{W}$. Additionally, $\F$ is bounded on $\mathcal{W}$ in the sense of
		\[
          |\F[W]|
          =\left|  \int_{\Omega} v^\ast \ \d(\mu_+{-}\mu_-) \right|
          \leq C \max \left\{  \int_{\Omega} \p(\,\cdot\,,\nabla v) \dx \,, \int_{\Omega} \widetilde\p(\,\cdot\,,\nabla v) \dx \right\}
          \leq C\beta \|W\|_{\L^1}
		   \text{ for all }W=\nabla v\in\mathcal{W}\,.
		\]
		By the Hahn-Banach theorem we then extend $\F$ to a functional $\widetilde{\F}$ defined on the whole space $\L^1(\Omega,\R^N)$ with same dual norm and such that the inequality $\widetilde{\F} \leq \I_{\p}$ is preserved on all $\L^1(\Omega,\R^N)$. Then, applying Riesz' representation theorem to $\widetilde{\F}$, we find $\sigma \in \L^{\infty}(\Omega,\R^N)$
		with $\|\sigma\|_{\L^{\infty}} = \|\widetilde{\F}\|_{(\L^1)^\ast} =  \|{\F}\|_{\mathcal{W}^\ast} \leq C \beta$ as well as $\widetilde{\F}[W] = \int_{\Omega}\sigma \ip W\dx$ whenever $W\in\L^1(\Omega,\R^N)$. In particular, on $\mathcal{W}$ we obtain
		\[
		- \int_{\Omega} v^\ast \,\d(\mu_+{-}\mu_-)
		= \int_{\Omega} v^\ast \,\d\mu_- - \int_{\Omega} u^\ast \,\d\mu_+ 
		= \F[\nabla v] 
		= \int_{\Omega}\sigma \ip \nabla v \dx
		\qq \text{for all } v \in \W^{1,1}_0(\Omega)
		\]
		and a fortiori for all $v\in\C^\infty_\c(\Omega)$. Thus, we have verified $\div(\sigma)=\mu_+-\mu_-$ as distributions.
		
		We claim furthermore that $\p^\circ(x,\sigma(x)) \leq C$ holds for \ae{} $x\in\Omega$. Indeed, from $\widetilde{\F} \leq \I_{\p}$ we deduce
		\[
		\int_{\Omega}\sigma \ip W\,\dx 
		= \widetilde{\F}[W] 
		\leq \I_{\p}[W]
		= C \int_{\Omega}\p(\,\cdot\,,W) \,\dx
		\qq \text{for all } W \in \L^1(\Omega,\R^N)\,.
		\]
		Specifically, when we plug in $W=\psi\xi$ with arbitrary $0\leq\psi\in\C^\infty_c(\Omega)$ and a constant vector $\xi\in\R^N$, the homogeneity of $\p$ leads to
		\[
		\int_{\Omega}\big( \sigma \ip \xi -  C\, \p(\,\cdot\,,\xi)\big) \psi \dx \leq 0
		\qq \text{whenever } 0 \leq \psi \in \C^{\infty}_\c(\Omega) \,.
		\]
		Hence, the fundamental lemma of the calculus of variations gives for every $\xi\in\R^N$ a negligible set $N_\xi\subseteq\Omega$ such that
		\begin{equation} \label{anis_eq:polar_IC1}
			\sigma(x) \ip \xi 
			\leq C \, \p(x,\xi)
			\qq\text{for }x\in\Omega\setminus N_\xi\,.
		\end{equation}
		In order to estimate $\p^\circ(x,\sigma(x))$, we first deduce from \eqref{anis_eq:polar_IC1} that
		\[
		\sup_{ \xi \in \mathds{Q}^N \setminus \{0\} } \frac{\sigma(x) \ip \xi}{\p(x,\xi)} 
		\leq C
		\qq\text{for }x\in\Omega\setminus\bigcup_{\xi \in \mathds{Q}^N} N_\xi\,.
		\]
		Then, by density of $\mathds{Q}^N$ in $\R^N$ and continuity of $\p$ in the second variable we arrive at
		\[
		\p^\circ(x,\sigma(x))
		= \sup_{ \xi \in \R^N \setminus \{0\} } \frac{\sigma(x) \ip \xi}{\p(x,\xi)} 
		\leq C
		\qq\text{for }x\in\Omega\setminus\bigcup_{\xi \in \mathds{Q}^N} N_\xi\,.
		\]
		As $\bigcup_{\xi \in \mathds{Q}^N} N_\xi$ is still negligible, this finalizes the deduction of $\p^\circ(x,\sigma(x))\leq C$ for \ae{} $x\in\Omega$ and of \eqref{item:3}. 
	\end{proof}
	
	Among the characterizations of the $\p$-IC for $(\mu_-,\mu_+)$ plus the $\widetilde\p $-IC for $(\mu_+,\mu_-)$ in Theorem \ref{thm:anis_equivalence_measure_TOGETHER}, at least the conditions of types 1 and 2 have all been expressed through two separate estimates from above and below. At least for the conditions of type 2 we now record the alternative of giving up the non-negativity of the functions $v$ and expressing the conditions as just one inequality for $v$ of arbitrary sign (compare also the preceding proof, where for \eqref{item:2c} this was partially touched upon):
	
	\begin{rem}[$\p$-anisotropic ICs for functions of arbitrary sign]\label{rem:IC_global}
		Yet another equivalent reformulation of condition \eqref{item:2a} of Theorem \ref{thm:anis_equivalence_measure_TOGETHER} is
		\begin{equation} \label{2A_global}
			\int_{\Omega} v^+ \ \d\mu_- - \int_{\Omega} v^- \ \d\mu_+ 
			\leq C \left(  |\D v|_{\p}(\Omega) + \int_{\partial \Omega} \p(\,\cdot\,, v \nu_{\Omega} )\,\d\H^{N-1}\right) 
			\qq \text{for all } v \in \BV(\Omega)
		\end{equation} 
		with no restriction of the sign of $v$. In an analogous manner, also conditions \eqref{item:2b} and \eqref{item:2c} of Theorem \ref{thm:anis_equivalence_measure_TOGETHER} can be expressed equivalently with functions of arbitrary sign.
	\end{rem}
	
	\begin{proof}
		To check that \eqref{item:2a} of Theorem \ref{thm:anis_equivalence_measure_TOGETHER} implies \eqref{2A_global}, we consider $v \in \BV(\Omega)$ and apply the upper bound of \eqref{item:2a} for $v_+$, the lower bound for $v_-$. This yields
		\begin{align*}
			\int_{\Omega} (v_+)^+ \,\d\mu_- - \int_{\Omega} (v_+)^- \,\d\mu_+  
			&\leq C \left(  |\D v_+|_{\p}(\Omega) + \int_{\partial \Omega} \p(\,\cdot\,, v_+ \nu_\Omega)\,\d\H^{N-1}\right)\,,\\
			- \int_{\Omega} (v_-)^- \,\d\mu_- + \int_{\Omega} (v_-)^+ \,\d\mu_+  
			&\leq C \left(  |\D v_-|_{\widetilde\p}(\Omega) + \int_{\partial \Omega} \widetilde\p(\,\cdot\,, v_- \nu_\Omega)\,\d\H^{N-1}\right)\,.
		\end{align*}
		Summing up term by term and employing the $\H^{N-1}$-\ae{} decomposition of $v^+$ and from $v^-$ from Lemma \ref{lem:dec_upper_limit} together with the equality $|\D v|_\p=|\D v_+|_\p+|\D v_-|_{\widetilde\p}$ of Lemma \ref{anis_lem:add_parts}, the last two estimates combine to precisely \eqref{2A_global}.
		
		\smallskip
		
		Now suppose that \eqref{2A_global} is fulfilled. Rewriting this inequality for ${-}v$ with $(-v)^\pm=-v^\mp$ we obtain 
		\begin{equation} \label{2A_global_neg}
			- \int_{\Omega} v^- \ \d\mu_- + \int_{\Omega} v^+ \ \d\mu_+
			\leq C \left(  |\D v|_{\widetilde \p}(\Omega) + \int_{\partial \Omega} \widetilde\p(\,\cdot\,, v \nu_{\Omega} )\,\d\H^{N-1}\right)
			\qq \text{for all } v \in \BV(\Omega)\,.
		\end{equation} 
		Finally, specializing both \eqref{2A_global} and \eqref{2A_global_neg} to non-negative $v$, we deduce the two inequalities of \eqref{item:2a}.
		
		\smallskip
		
		For what concerns \eqref{item:2b} and \eqref{item:2c}, the reasoning is analogous (and with $v$ and $v^\ast$ in place of $v^\pm$ even slightly simpler).
	\end{proof}
	
	\subsection{Anisotropic ICs for a signed measure}
		
	Next we turn specifically to mutually singular measures $\mu_+\perp\mu_-$, or in other words to the most relevant case of a signed measure $\mu=\mu_+-\mu_-$. Before eventually coming back to the assertions of Theorem \ref{thm:anis_equivalence_measure_TOGETHER}, we show that in this situation we can actually construct strict approximations of an arbitrary $u\in\BV(\Omega)$ such that in the limit we reproduce the representative $u^+$ in the $\mu_-$-term and at the same time the representative $u^-$ in the $\mu_+$-term. This is made precise in the next technical proposition, which will also be useful in the different context of the later Section \ref{sec:recovery}. We emphasize that a similar result, which allows even for arbitrary representatives between $u^+$ and $u^-$, has been obtained in a slightly different framework in \cite[Lemma 4.17]{ComLeo25} on the basis of fine approximations from \cite[Theorem 3.2]{ComLeo25}.

    \begin{prop}\label{prop:conv-wk}
      We impose Assumption \ref{assum:mu}, even with an arbitrary open $U\subseteq\R^N$ instead of $\Omega$, and additionally assume $\mu_+\perp\mu_-$. Then, for every $u\in\BV(U)$, there exists a sequence $(w_k)_k$ in $\W^{1,1}(U)$ such that $(w_k)_k$ converges to $u$ strictly in $\BV(U)$ with
      \begin{equation}\label{eq:conv-wk}
        \lim_{k \to \infty}\int_{U} w_k^\ast\,\d\mu_- 
        = \int_{U} u^+ \,\d\mu_-
        \qq\qq\text{and}\qq\qq
        \lim_{k \to \infty}\int_{U} w_k^\ast\,\d\mu_+
        = \int_{U} u^- \,\d\mu_+\,.
      \end{equation}
      If $u$ has compact support in $U$, we may even choose $(w_k)_k$ in $\W^{1,1}_0(U)$.
    \end{prop}

    \begin{proof}
        Given $u \in \BV(U)$, we employ Proposition \ref{prop:mon_area_strict_appr} and Lemma \ref{lem:mon_area_strict_appr} with the measures $\mu_\pm$ and pass to suitable subsequences. In this way, we obtain on one hand a sequence $(u_\ell)_\ell$ in $\W^{1,1}(U)$ which converges to $u$ strictly in $\BV(U)$ with $u_\ell \geq u$ \ae{} in $U$ for all $\ell\in\N$ and with
        \begin{equation} \label{eq:conv_appr_above_mu_pm}
          \lim_{\ell \to \infty} \int_{U} u_\ell^\ast\,\d\mu_\pm 
          = \int_{U} u^+ \,\d\mu_\pm\,.
        \end{equation}
        On the other hand, we also get a sequence $(v_\ell)_\ell$ in $\W^{1,1}(U)$ which converges to $u$ strictly in $\BV(U)$ with $v_\ell \leq u$ \ae{} in $U$ for all $\ell\in\N$ and with
        \begin{equation}\label{eq:conv_appr_below_mu_pm}
          \lim_{\ell \to \infty} \int_{U} v_\ell^\ast\,\d\mu_\pm 
          = \int_{U} u^-\,\d\mu_\pm\,.
        \end{equation}
    
        In the sequel we construct $w_k$ by interpolation between $u_{\ell_k}$ and $v_{\ell_k}$ with the help of suitable cut-off functions $\eta_k$. To this end we first decisively use $\mu_+ \perp \mu_-$ to decompose $U=P\cupdot M$ into Borel sets $P$ and $M$ such that $\mu_+(M)=\mu_-(P)=0$. In view of Assumption \ref{assum:mu}, we have $\int_U(u^+-u^-)\,\d\mu_\pm\leq2\int_U|u|^+\,\d\mu_\pm<\infty$ for the non-negative function $u^+-u^-$. Thus, by absolute continuity of the integral, for every $k\in\N$, there exists some $\delta_k>0$ such that $\mu_+(E)<\delta_k$ for Borel $E\subseteq U$ implies $\int_E(u^+-u^-)\,\d\mu_+\leq\frac1k$ and likewise $\mu_-(E)<\delta_k$ implies $\int_E(u^+-u^-)\,\d\mu_-\leq\frac1k$. Since $\mu_+$ is finite and concentrated on $P$, for every $k\in\N$, there is a compact set $K_{k;P} \subseteq P$ such that $\mu_+(U \setminus K_{k;P})<\delta_k$, and analogously there is a compact set $K_{k;M}\subseteq M$ such that $\mu_-(U \setminus K_{k;M}) < \delta_k$. Altogether we may thus record
        \begin{equation}\label{eq:abs-con-est}
          \int_{U \setminus K_{k;P}}(u^+-u^-)\,\d\mu_+\le{\ts\frac1k}
          \qq\qq\text{and}\qq\qq
          \int_{U \setminus K_{k;M}}(u^+-u^-)\,\d\mu_-\le{\ts\frac1k}\,.
        \end{equation}
        Next we consider cut-off functions $\eta_k \in \C^{1}_\c(U)$ with $\eta_k \equiv 1$ on $K_{k;M}$ and $\eta_k \equiv 0$ on $K_{k;P}$ and $0 \leq \eta_k \leq 1$ in $U$. At this stage, we come back to the strict convergences of $(u_\ell)_\ell$ and $(v_\ell)_\ell$ to $u$ and record that these imply $\|u_\ell - v_\ell \|_{\L^1(U)}\to0$ for $\ell\to\infty$ plus $(\psi|\D u_\ell|)(U)\to(\psi|\D u|)(U)$ and $(\psi|\D v_\ell|)(U)\to(\psi|\D u|)(U)$ for $\ell\to\infty$ whenever $\psi\in\mathrm{C}^0(U)\cap\L^\infty(U)$; compare e.\@g.\@ \cite[Proposition 1.80]{AFP00}. Also keeping \eqref{eq:conv_appr_above_mu_pm} and \eqref{eq:conv_appr_below_mu_pm} in mind, we may then choose an increasing sequence $(\ell_k)_k$ of positive integers such that, for later convenience, we have the inequalities
        \begin{align}
          &\hspace{-12ex}\|u_{\ell_k}-v_{\ell_k}\|_{\L^1(U)} \, \|\nabla\eta_k\|_{\L^\infty(U,\R^N)} \leq {\ts\frac1k} \,,\label{eq:fast1}\\
          (\eta_k|\D u_{\ell_k}|)(U)
          \leq (\eta_k|\D u|)(U)+{\ts\frac1k}
          &\qq\text{and}\qq
          \big((1-\eta_k)|\D v_{\ell_k}|\big)(U)
          \leq \big((1-\eta_k)|\D u|\big)(U)+{\ts\frac1k}\label{eq:fast2}\,,\\
          \int_U(u_{\ell_k}^\ast-u^+)\,\d\mu_\pm
          \leq{\ts\frac1k}
          &\qq\text{and}\qq
          \int_U(u^--v_{\ell_k}^\ast)\,\d\mu_\pm
          \leq{\ts\frac1k}\label{eq:fast3}
    	\end{align}
    	for all $k\in\N$. We proceed by defining
    	\[
          w_k\coleq\eta_k u_{\ell_k} + (1-\eta_k) v_{\ell_k}\in\W^{1,1}(U)
    	\]
        and aim at establishing \eqref{eq:conv-wk}.
        To this end, for fixed $k \in \N$, we rewrite       \begin{equation}\label{estimate_u_k_v_k_prop1}\begin{aligned}
          \int_{U} w_k^\ast\ \d\mu_-
          & = \int_{U} u_{\ell_k}^\ast\ \d\mu_-
          - \int_{U} (1-\eta_k) (u_{\ell_k}^\ast-v_{\ell_k}^\ast) \,\d\mu_-\,,\\    
          \int_{U} w_k^\ast\ \d\mu_+
          & = \int_{U} v_{\ell_k}^\ast\ \d\mu_+
          + \int_{U} \eta_k (u_{\ell_k}^\ast-v_{\ell_k}^\ast) \,\d\mu_+\,.
        \end{aligned}\end{equation}
        Next we exploit that $v_\ell^\ast\leq u^\pm\leq u_\ell^\ast$ holds $\mu_\pm$-\ae{} in $U$ and that we have $\eta_k\equiv1$ on $K_{k;M}$ and
        $\eta_k\equiv0$ on $K_{k;P}$. Also bringing in \eqref{eq:abs-con-est} and \eqref{eq:fast3}, we then estimate
        \begin{align*}
          0\leq\int_U (1-\eta_k) (u_{\ell_k}^\ast-v_{\ell_k}^\ast) \,\d\mu_-
          &\leq\int_{U\setminus K_{k;M}}(u^+-u^-)\,\d\mu_-
          +\int_U(u_{\ell_k}^\ast-u^+)\,\d\mu_-
          +\int_U(u^--v_{\ell_k}^\ast)\,\d\mu_-
          \leq{\ts\frac3k}\,,\\
          0\leq\int_U \eta_k (u_{\ell_k}^\ast-v_{\ell_k}^\ast) \,\d\mu_+
          &\leq\int_{U\setminus K_{k;P}}(u^+-u^-)\,\d\mu_+
          +\int_U(u_{\ell_k}^\ast-u^+)\,\d\mu_+
          +\int_U(u^--v_{\ell_k}^\ast)\,\d\mu_+
          \leq{\ts\frac3k}\,.
        \end{align*}
        Employing \eqref{eq:conv_appr_above_mu_pm}, \eqref{eq:conv_appr_below_mu_pm} together with the preceding estimates we may pass to the limit in \eqref{estimate_u_k_v_k_prop1} and then indeed reach \eqref{eq:conv-wk}.
				
        In addition, we now check that $(w_k)_k$ converges to $u$ strictly in $\BV(U)$. To this end, we first read off from the strict convergence of $(u_\ell)_\ell$ and $(v_\ell)_\ell$ to $u$ that $w_k\to u$ in $\L^1(U)$ for $k\to\infty$, and by lower semicontinuity of the total variation we deduce 
        \[
          \liminf_{k \to \infty}|\D w_k|(U)
          \geq |\D u|(U)\,.
        \]
        Furthermore, from $\nabla w_k=(u_{\ell_k}-v_{\ell_k})\nabla\eta_k + \eta_k\nabla u_{\ell_k} + (1-\eta_k)\nabla v_{\ell_k}$ and the estimates \eqref{eq:fast1}, \eqref{eq:fast2} we get
        \[\begin{aligned}
          |\D w_k|(U)
          &\leq \|u_{\ell_k} - v_{\ell_k} \|_{\L^1(U)} \, \|\nabla \eta_k\|_{\L^\infty(U,\R^N)} + (\eta_k|\D u_{\ell_k}|)(U) + \big((1-\eta_k)|\D v_{\ell_k}|\big)(U)\\
          &\leq (\eta_k|\D u|)(U) + \big((1-\eta_k)|\D u|\big)(U) + {\ts\frac3k}
          = |\D u|(U) + {\ts\frac3k} \,.
        \end{aligned}\]
        This yields
        \[
          \limsup_{k \to \infty}|\D w_k|(U)
          \le |\D u|(U)\,,
        \]
        and all in all we have verified strict convergence of $(w_k)_k$ to $u$ in $\BV(U)$.

        Finally, we turn to the case that $u$ has compact support in $U$. In this case, we still use the construction described above and additionally exploit that the approximations $u_\ell$ and $v_\ell$ from Proposition \ref{prop:mon_area_strict_appr} can even be taken in $\W^{1,1}_0(U)$. As a consequence, we find $w_k=\eta_ku_{\ell_k}+(1-\eta_k)v_{\ell_k}\in\W^{1,1}_0(U)$, since the term $\eta_ku_{\ell_k}$ inherits zero boundary values from $u_\ell$ (and also from $\eta_k$), the term $(1-\eta_k)v_{\ell_k}$ from $v_\ell$.
    \end{proof}

    \begin{rem}
      In case $|U|<\infty$ the sequence $(w_k)_k$ of Proposition \ref{prop:conv-wk} can be taken such that it converges not only strictly, but even area-strictly. In fact, a variant of Proposition \ref{prop:mon_area_strict_appr} provides even area-strict approximations $u_\ell$ and $v_\ell$, and it is possible to preserve even this somewhat better type of convergence in the construction of the preceding proof.
    \end{rem}

    As announced earlier, we now improve on Theorem \ref{thm:anis_equivalence_measure_TOGETHER} in case $\mu_+\perp\mu_-$, and complete the picture by establishing the equivalence of all six assertions considered there. In addition, we will show that, in contrast to Theorem \ref{thm:anis_equivalence_measure_TOGETHER}, we need not to impose condition \eqref{eq:negligible} from Assumption \ref{assum:mu} (that is, the vanishing of $\mu_\pm$ on $\H^{N-1}$-negligible sets) as an explicit hypothesis, as now it follows automatically from each of the six assertions.

		\begin{thm}[characterizations of $\p$-anisotropic ICs, second refined version]\label{thm:anis_equivalence_gen_IC_singular}
			We impose Assumption \ref{assum:phi}\eqref{case:cont}, consider non-negative Radon measures $\mu_+$ and $\mu_-$ on $\Omega$ such that $\mu_+ \perp \mu_-$ and \eqref{eq:finite-integral} hold, and fix $C>0$. Then all the assertions \eqref{item:1a}, \eqref{item:1b}, \eqref{item:2a}, \eqref{item:2b}, \eqref{item:2c} of Theorem \ref{thm:anis_equivalence_measure_TOGETHER} are mutually equivalent with each other. If additionally Assumption \ref{assum:phi}\eqref{case:conv-lsc} is in force, also assertion \eqref{item:3} is equivalent. Moreover, any of the six assertions implies the vanishing condition \eqref{eq:negligible}.
			
			\begin{proof}
				We first reason that \eqref{item:1a} implies \eqref{eq:negligible}. To this end, we consider an $\H^{N-1}$-negligible Borel set $Z\subseteq\Omega$, and in view of $\mu_+\perp\mu_-$ we decompose $\Omega=P\cupdot M$ into Borel sets such that $\mu_+(M)=0=\mu_-(P)$. We will now show
				\begin{equation}\label{eq:mu_abs_con}
					\mu_+(K)=0
					\qq\text{for every compact }K\subseteq Z\cap P\,.
				\end{equation}
				Given $\eps>0$ we fix an open set $O$ with $K\subseteq O\Subset\Omega$ and $\mu_-(O)<\eps$. Then $\H^{N-1}(K)=0$ implies by a well-known covering argument (see e.\@g.\@ \cite[Lemma 2.7]{Schmidt25}) the existence of an open set $A$ with $K\subseteq A\Subset O$ and $\P(A)<\eps$. Exploiting the left-hand inequality in \eqref{item:1a} for this $A$ and involving \eqref{eq:comp-p0} we infer
				\[
				\mu_+(A^+)-\mu_-(A^1)
				\le C\P_{\widetilde\p}(A)
				\le C\beta\P(A)
				<C\beta\eps\,.
				\]
				Taking into account $\mu_-(A^1)\le\mu_-(O)<\eps$ this leads to $\mu_+(K)\le\mu_+(A^+)<C\beta\eps+\eps$, and then, since $\eps>0$ was arbitrary, we arrive at \eqref{eq:mu_abs_con}. An analogous argument based on the right-hand inequality in \eqref{item:1a} gives also $\mu_-(K)=0$ for every compact $K\subseteq Z\cap M$. Then, via inner regularity of $\mu_\pm$ we arrive at $\mu_\pm(Z)=0$ and have completed the proof of \eqref{eq:negligible}.
				
				\smallskip
				
				At this point, since \eqref{item:1b} is stronger than \eqref{item:1a} and \eqref{item:2a} induces \eqref{item:1a} by specializing to characteristic functions (compare the proof of Theorem \ref{thm:anis_equivalence_measure_TOGETHER}), we also know that \eqref{item:1b} and \eqref{item:2a} imply \eqref{eq:negligible}.
				
				\smallskip
				
				Next we argue that also \eqref{item:2b} implies \eqref{eq:negligible}. We begin the reasoning right the same way as in the first step of this proof, but then suitably mollify the characteristic function $\1_A$ of the open set $A$ with $K\subseteq A\Subset O$ and $\P(A)<\eps$. In this way, we obtain $v\in\C^\infty_\c(\Omega)$ with $\1_K\leq v\leq\1_O$ in $\Omega$ and $\int_\Omega|\nabla v|<\eps$. Using this $v$ in \eqref{item:2b}, we find
				\[
				\int_\Omega v\,\d\mu_+-\int_\Omega v\,\d\mu_-
				\le C\int_\Omega\widetilde\p(\,\cdot\,,\nabla v)
				\le C\beta\int_\Omega|\nabla v|\dx
				<C\beta\eps\,.
				\]
				As before, in view of $\int_\Omega v\,\d\mu_-\le\mu_-(O)<\eps$ we deduce $\mu_+(K)\le\int_\Omega v\,\d\mu_+<C\beta\eps+\eps$ and then arrive at \eqref{eq:mu_abs_con} once more. Together with the analog for $\mu_-$ this yields \eqref{eq:negligible} also under \eqref{item:2b} as an hypothesis.
				
				\smallskip
				
				Again, as \eqref{item:2c} is just stronger than \eqref{item:2b}, we infer that \eqref{item:2c} implies \eqref{eq:negligible} as well, and finally \eqref{item:3} implies \eqref{eq:negligible} by \cite[Proposition 3.1]{CheFri99}. At this point, having verified that each of the six assertions enforces the validity of \eqref{eq:negligible}, we can deduce from Theorem \ref{thm:anis_equivalence_measure_TOGETHER} that all equivalences and implications of that theorem are valid also in the present framework.
				
				\smallskip
				
				Therefore, it suffices to additionally prove that \eqref{item:2c} implies \eqref{item:1a}. So, we assume \eqref{item:2c} and consider a measurable $A\Subset\Omega$, for which we suppose without loss of generality $\P(A)<\infty$. By applying Proposition \ref{prop:conv-wk} to the compactly supported function $\1_A\in\BV(\Omega)$, we then find a sequence $(w_k)_k$ in $\W^{1,1}_0(\Omega)$ which converges to $\1_A$ strictly in $\BV(\Omega)$ and satisfies
                \[
                  \lim_{k \to \infty}\int_{\Omega} w_k^\ast\,\d\mu_- 
                  = \int_{\Omega} (\1_A)^+ \,\d\mu_-
                  = \mu_-(A^+)
                  \qq\qq\text{and}\qq\qq
                  \int_{\Omega} w_k^\ast\,\d\mu_+
                  = \int_{\Omega} u^- \,\d\mu_+
                  = \mu_+(A^1)\,.
                \]
                Moreover, by Theorem \ref{anis_thm:Reshetnyak_cont}, the convergence is also $\p$-strict and thus in particular
				\[
                  \lim_{k \to \infty}|\D w_k|_\p(\Omega)
                  = |\D\1_A|_\p(\Omega)\,.
				\]
				Employing now the right inequality of \eqref{2B''} in \eqref{item:2c} for $w_k\in\W^{1,1}_0(\Omega)$ and letting $k \to \infty$, by the convergences just recorded we obtain
				\[
                  \mu_-(A^+) - \mu_+(A^1)
                  \le C\,|\D\1_A|_\p(\Omega)
                  = C\,\P_\p(A)\,.
				\]
				Moreover, from Proposition \ref{prop:conv-wk} we can also obtain yet another sequence $(\widetilde w_k)_k$ in $\W^{1,1}_0(\Omega)$ which has the same properties as $(w_k)_k$ with roles of $\mu_+$ and $\mu_-$ switched, and clearly we may also use $\widetilde\p$-variations instead of $\p$-variations.
				Employing then the left inequality of \eqref{2B''} in \eqref{item:2c} for $\widetilde w_k\in\W^{1,1}_0(\Omega)$ and letting $k\to\infty$, we get analogously
				\[
                  \mu_-(A^1) - \mu_+(A^+)
                  \ge {-}C\,|\D\1_A|_{\widetilde\p}(\Omega)
                  = {-}C\,\P_{\widetilde\p}(A)\,.
				\]
				With the last two inequalities, we arrive at \eqref{item:1a}, and this completes the final step of the equivalence proof. 
			\end{proof}
		\end{thm}
		
		Finally, in the next remark we point out that in case of $\mu_+\perp\mu_-$ we may also switch the representatives in our ICs. Nevertheless, the particular choice of Definition \ref{defi:IC} and the preceding results is convenient for a few later statements, valid even without $\mu_+\perp\mu_-$, and thus we mostly stick to this convention.
		
		\begin{rem}[$\p$-anisotropic ICs with other choices of representatives]
			In case $\mu_+\perp\mu_-$, condition \eqref{item:1a} of Theorem \ref{thm:anis_equivalence_measure_TOGETHER} can be equivalently rewritten as
			\begin{equation}\label{eq:IC-with-boxes}
				-C\P_{\widetilde\p}(A)
				\leq  \mu_-(A^{\Box_1})-\mu_+(A^{\Box_2}) 
				\leq  C\P_{\p}(A)
			\end{equation}
			for all measurable $A\Subset\Omega$ with any fixed choice of\/ $\Box_1,\Box_2\in\{+,1\}$, that is, when using one of the four possible combinations of\/ $A^+$ and\/ $A^1$ in the formula. Still for $\mu_+\perp\mu_-$, conditions \eqref{item:1b} and \eqref{item:2a} can be equivalently rephrased in the same manner, in case of \eqref{item:2a} with any of the four combinations of\/ $u^+$ and\/ $u^-$. \ka For the further conditions \eqref{item:2b}, \eqref{item:2c}, \eqref{item:3} instead, there is no similar issue of choosing representatives.\kz
		\end{rem}
		
		\begin{proof}[Sketch of proof]
			In view of $A^1\subseteq A^+$, the original \eqref{item:1a} implies the variant \eqref{eq:IC-with-boxes} for any choice of $\Box_1,\Box_2$.
			
			\smallskip
			
			Moreover, if we start from \eqref{eq:IC-with-boxes}, we can mimic parts of the proof of Theorem \ref{thm:anis_equivalence_measure_TOGETHER} in order to deduce \eqref{item:2c}. The essential argument for this purpose is the one based on the coarea and layer-cake formulas, recorded previously in going from \eqref{item:1b} to \eqref{item:2a}. In particular, when considering only $v\in\W^{1,1}_0(\Omega)$ as in \eqref{item:2c}, Remark \ref{rem:equivalence_sets_wrt_mu_ext} ensures $\{\overline v>t\}^+=\{\overline v>t\}^1$ up to $\mu_\pm$-negligible sets for \ae{} $t\in(0,\infty)$, which underpins the irrelevance of the chosen representatives $A^+$ and $A^1$ in this reasoning. Anyway, since we assume $\mu_+\perp\mu_-$, once \eqref{item:2c} is at hand, Theorem \ref{thm:anis_equivalence_gen_IC_singular} gives the validity of the original condition \eqref{item:1a} as well.
			
			\smallskip
			
			The equivalence with the variants of \eqref{item:1b} and \eqref{item:2a} comes out along the same lines, since these imply \eqref{item:2c} as well (by the coarea and layer-cake argument just roughly touched upon).
		\end{proof}

	\section{Various examples}\label{sec:exp}
	
	In this section we are concerned with the details of the various examples presented earlier, where a recurring issue will be the verification of ($\p$"~)""ICs. While the validity of such ICs is often geometrically plausible from their very definition, for their precise verification it is typically much more convenient to uncover a divergence structure and then apply the very useful Theorem \ref{thm:anis_equivalence_gen_IC_singular}.
	
	\subsection[A non-trivial signed IC]
	{A non-trivial signed IC (Example \ref{exp:signed-IC})}
	\label{subsec:signed-IC}
	
	Here we briefly justify the claims made for the Radon measures
	\[
	\mu\coleq(1{+}\theta/2)\H^1\ecke\partial\B_2
	\qq\text{and}\qq
	\nu\coleq\theta\H^1\ecke\partial\B_1
	\qq\text{on }\R^2\,,
	\qq\text{with fixed }\theta\in{(0,1]}\,.
	\]
	
	\begin{proof}[Verification that $(\mu,\nu)$ \emph{does} satisfy the isotropic IC in $\R^2$ with constant $1$]
		By setting
		\[
		\sigma(x)\coleq\begin{cases}
			0&\text{for }|x|<1\\
			{-}\theta\frac x{|x|^2}&\text{for }1<|x|<2\\
			2\frac x{|x|^2}&\text{for }2<|x|
		\end{cases}
		\]
		we obtain a vector field $\sigma\in\L^\infty(\R^2,\R^2)$, and in view of $\theta\le1$ we have $\|\sigma\|_{\L^\infty;\R^2}\le1$ (i.\@e.\@ $\sigma$ is a sub-unit field). Moreover, a standard computation with the divergence theorem reveals
		$\mathrm{div}(\sigma)=\mu-\nu$ in the sense of distributions on $\R^2$. At this point, Theorem \ref{thm:anis_equivalence_gen_IC_singular}, applied with $(\mu_+,\mu_-)=(\mu,\nu)$ and $\sigma$ above (or alternatively with $(\mu_+,\mu_-)=(\nu,\mu)$ and ${-}\sigma$ in place of $\sigma$) and in any case with $\p(x,\xi)=|\xi|$, yields the isotropic IC for $(\mu,\nu)$ with constant $1$ first on $\Omega=\B_r$, $r\gg1$, and as a consequence also on $\R^2$.
	\end{proof}
	
	\begin{proof}[Verification that $\mu$ alone \emph{does not} satisfy the isotropic IC in $\R^2$ with constant $1$]
		This claim is immediately confirmed by observing $\mu(\overline{\B_2})=4\pi(1{+}\theta/2)>4\pi=\P(\B_2)$.
	\end{proof}
	
	We additionally record that in this exemplary case also $(\nu,\mu)$ and $\nu$ alone \emph{do} satisfy the isotropic IC in $\R^2$ with constant $1$. This can be checked analogous to the reasoning for $(\mu,\nu)$, where now one merely needs to consider $\nu$ alone and then can use an even slightly simpler $\sigma$.
	
	\subsection[\texorpdfstring{The signed IC does not enforce finiteness on $\W^{1,1}$}{}]
    {\boldmath The signed IC does not enforce finiteness on $\W^{1,1}$ (Example \ref{exp:non-finite})}
    \label{subsec:non-finite}
	
	With regard to Example \ref{exp:non-finite} it merely remains to check the ICs for the Radon measures
	\[
	\mu=\H^1\ecke\bigcup_{k=1}^\infty S_{2k-1}
	\qq\text{and}\qq
	\nu=\H^1\ecke\bigcup_{k=1}^\infty S_{2k}
	\qq\text{on }\R^2\,,
	\]
	where we recall the abbreviation $S_i=\partial\B_{1/i^2}$.
	
	\begin{proof}[Verification that $(\mu,\nu)$ and $(\nu,\mu)$ satisfy the isotropic IC in $\R^2$ with constant $1$]
		For the auxiliary numbers $\alpha_i\coleq\sum_{j=i}^\infty\frac{(-1)^{j-1}}{j^2}$ (the remainders of the Dirichlet $\eta(2)$ series), we observe $\mathrm{sgn}(\alpha_i)=(-1)^{i-1}$ and $|\alpha_i|\leq\frac1{i^2}$. Then, by setting
		\[
		\sigma(x)\coleq\begin{cases}
			\alpha_1\frac x{|x|^2}
			&\text{if }|x|>1\\
			\alpha_i\frac x{|x|^2}
			&\text{if }\frac1{(i-1)^2}>|x|>\frac1{i^2}\text{ for some }i\ge2
		\end{cases}
		\]
		we obtain a vector field $\sigma\in\L^\infty(\R^2,\R^2)$ with $\|\sigma\|_{\L^\infty;\R^2}\le1$. Once more, with the divergence theorem one checks
		$\div(\sigma)=\mu-\nu$ in the sense of distributions on $\R^2$. Then, as in the previous Section \ref{subsec:signed-IC}, Theorem \ref{thm:anis_equivalence_gen_IC_singular} yields the claimed IC for both $(\mu,\nu)$ and $(\nu,\mu)$.
	\end{proof}
	
	\subsection[Non-existence in the extreme case with unbounded boundary datum]
	{Non-existence in the extreme case with unbounded boundary datum (Example \ref{exp:non-exist})}
	\label{subsec:non-exist}
	
	Here we provide a detailed treatment of Example \ref{exp:non-exist}, which --- as we recall --- works on the Lipschitz domain $\Omega=\{x\in\B_2\,:\,x_2>{-}1\} \subseteq\R^2$ with any datum $u_0\in\W^{1,1}(\R^2)$ which extends $u_0(x)\coleq(|x|{-}1)^{-\alpha}$ for $x\in\B_3\setminus\overline\Omega$ with some fixed $\alpha\in{(0,1/2)}$.
	
	At first we turn to the auxiliary claim that such an extension $u_0$ exists, which follows from standard Sobolev extension results as soon as $u_0(x)=(|x|{-}1)^{-\alpha}$ and $|\nabla u_0(x)|=\alpha(|x|{-}1)^{-(\alpha+1)}$ give $\L^1$ functions on $\B_3\setminus\overline\Omega$. These integrabilities, in turn, depend only on the behavior near the singular point $(0,{-}1)\in\partial\Omega$ and are verified through the subsequent lemma (for $u_0$ with $\beta=\alpha<1/2$ and for $|\nabla u_0|$ with $\beta=\alpha{+}1<3/2$).
	
	\begin{lem}
		For $\beta\in{(0,\infty)}$, we have
		\[
		\int_{{(-1,1)}{\times}{(-2,-1)}}\big(|x|-1\big)^{-\beta}\dx<\infty
		\qq\iff\qq
		\beta<\frac32\,.
		\]
	\end{lem}
	
	\begin{proof}
		A shift of domain transforms the integral of the lemma into $\int_{{(-1,1)}{\times}{(-1,0)}}\big(\sqrt{|x|^2{-}2x_2{+}1}-1\big)^{-\beta}\dx$. Then, in view of the routine estimates $\frac14t\leq\sqrt{1{+}t}-1\leq\frac12 t$ for $t\in{[0,8]}$ and
		$x_1^2-x_2\le|x|^2-2x_2\le3(x_1^2-x_2)$ for $x\in{(-1,1)}{\times}{(-1,0)}$, it suffices to investigate finiteness of
		$\int_{{(-1,1)}{\times}{(-1,0)}}(x_1^2{-}x_2)^{-\beta}\dx$. For the last-mentioned integral, however, we can use symmetry in $x_1$ and carry out $x_2$-integration to derive in case $\beta\neq1$ the formula
		\[
		\int_{{(-1,1)}{\times}{(-1,0)}}(x_1^2-x_2)^{-\beta}\dx
		=\frac2{\beta-1}\bigg(\int_0^1x_1^{2-2\beta}\dx_1-\int_0^1(x_1^2+1)^{1-\beta}\dx_1\bigg)\,,
		\]
		where $\int_0^1(x_1^2{+}1)^{1-\beta}\dx_1$ is always finite, while $\int_0^1x_1^{2-2\beta}\dx_1$ is finite if and only if $\beta<\frac32$. In case $\beta=1$, finiteness is checked analogously (then with $x_2$-integration giving logarithmic, but still integrable terms).
	\end{proof}
	
	Now we come to the main non-existence claims of Example \ref{exp:non-exist}. In this regard, we crucially exploit another lemma:
	
	\begin{lem}\label{lem:IC-1/|x|}
		For every measurable $A\subseteq\R^2$ with $|A|<\infty$, we have
		\begin{equation}\label{eq:IC-1/|x|}
			\int_A\frac1{|x|}\dx\leq\P(A)\,.
		\end{equation}
		Moreover, equality occurs in \eqref{eq:IC-1/|x|} if and only if\/ $|A\triangle\B_r|=0$ holds for some $r\in{[0,\infty)}$ \ka where in connection with this lemma we understand\/ $\B_0=\emptyset$\kz.
	\end{lem}
	
	\begin{proof}
		We fix $r\in{[0,\infty)}$ with $|\B_r|=|A|$. Then the isoperimetric inequality \eqref{eq:isop_ineq} and a standard computation give
		\begin{equation}\label{eq:IC-1/|x|-1}
			\P(A)\ge\P(\B_r)=\int_{\B_r}\frac1{|x|}\dx\,.
		\end{equation}
		Additionally, in view of $|A\setminus\B_r|=|\B_r\setminus A|$, we find
		\begin{equation}\label{eq:IC-1/|x|-2}  
			\int_A\frac1{|x|}\dx\leq\int_{A\cap\B_r}\frac1{|x|}\dx+\frac1r\,|A\setminus\B_r|
			=\int_{A\cap\B_r}\frac1{|x|}\dx+\frac1r\,|\B_r\setminus A|
			\leq\int_{\B_r}\frac1{|x|}\dx\,,
		\end{equation}  
		and both estimates combined yield \eqref{eq:IC-1/|x|}. Moreover, it is evident from the reasoning that $|A\setminus\B_r|=0=|\B_r\setminus A|$ is sufficient for equality in both \eqref{eq:IC-1/|x|-1} and \eqref{eq:IC-1/|x|-2}, and it is necessary (already) for equality in \eqref{eq:IC-1/|x|-2}.
	\end{proof}
	
	We record that Lemma \ref{lem:IC-1/|x|} (or alternatively the representation $1/|x|=\div_x(x/|x|)$ for $0\neq x\in\R^2$) implies the claimed limit-case IC for the case of Example \ref{exp:non-exist} with density $H(x)=1/|x|$ and then e.\@g.\@ by Proposition \ref{prop:admissible} ensures finiteness of the last term in the relevant functional
	\begin{equation}\label{eq:P-non-exist}
		\widehat\Phi[w]=|\D w|(\Omega)+\int_{\partial\Omega}|w{-}u_0|\,\d\H^1-\int_\Omega\frac{w(x)}{|x|}\dx
		\qq\text{among }w\in\BV(\Omega)
	\end{equation}
	(where in the $\H^1$-integral we use the traces of $w$ and $u_0$). We can now proceed with:
	
	\begin{proof}[Verification that there exists no minimizer for \eqref{eq:P-non-exist}]
		Since only the values of $u_0$ near $\partial\Omega$ matter and the formula for $u_0$ on $\B_3\setminus\overline\Omega$ depends only on $|x|$, we can assume that $u_0\in\W^{1,1}(\R^2)$ is rotationally symmetric and radially decreasing on all of $\R^2$ with bounded support. Moreover, by the usual extension procedure we can pass from \eqref{eq:P-non-exist} to the equivalent minimization problem for
		\[
		\widehat\Phi_\ast[w]\coleq|\D w|(\R^2)-\int_{\R^2}\frac{w(x)}{|x|}\dx
		\qq\text{among }w\in\BV_{u_0}\,,
		\]
		where we abbreviate
		\[
		\BV_{u_0}\coleq\big\{w\in\BV(\R^2)\,:\,
		w=u_0\text{ \ae{} on }\R^2\setminus\overline\Omega\big\}\,.
		\]
		We will now prove non-existence by showing on one hand
		\begin{equation}\label{eq:inf<=0}
			\inf_{\BV_{u_0}}\widehat\Phi_\ast\le0
		\end{equation}
		and on the other hand
		\begin{equation}\label{eq:P-bar>0}
			\widehat\Phi_\ast[w]>0\qq\text{for all }w\in\BV_{u_0}\,.
		\end{equation}
		
		In order to establish \eqref{eq:inf<=0}, for arbitrary $\eps>0$, we define $u_\eps\in\W^{1,\infty}(\Omega)$ by
		\[
		u_\eps(x)\coleq\frac1{\max\{\eps,|x|{-}1\}^\alpha}
		\qq\text{for }x\in\Omega\,,
		\]
		and extend to a non-negative function $u_\eps\in\BV_{u_0}$.
		Then, since $\{u_\eps>t\}$ is a ball centered at $0$ (up to negligible sets) for $t\in\big(0,\eps^{-\alpha}\big)$, the coarea and layer-cake formulas together with Lemma \ref{lem:IC-1/|x|} yield
		\[\begin{aligned}
			\widehat\Phi_\ast[u_\eps]
			&=\int_0^\infty\bigg(\P(\{u_\eps>t\})-\int_{\{u_\eps>t\}}\frac1{|x|}\dx\bigg)\,\d t
			=\int_{\eps^{-\alpha}}^\infty\bigg(\P(\{u_\eps>t\})-\int_{\{u_\eps>t\}}\frac1{|x|}\dx\bigg)\,\d t\\
			&\leq\int_{\eps^{-\alpha}}^\infty\P(\{u_\eps>t\})\,\d t
			=\int_{\B_{1+\eps}\setminus\overline\Omega}|\nabla u_0|\dx
			+\int_{\B_{1+\eps}\cap\partial\Omega}(u_0{-}u_\eps)\,\d\H^1\,.
		\end{aligned}\]  
		In view of the previously checked integrabilities and the resulting integrability of the trace of $u_0$ on $\partial\Omega$, the right-hand side of the last estimate vanishes in the limit $\eps\to0$. Thus, we have verified \eqref{eq:inf<=0}.
		
		For the proof of \eqref{eq:P-bar>0}, we use Lemma \ref{lem:IC-1/|x|} once more and observe
		\[
		\widehat\Phi_\ast[w]\ge|\D w|(\{w>0\})-\int_{\{w>0\}}\frac{w(x)}{|x|}\dx
		=\int_0^\infty\bigg(\P(\{w>t\})-\int_{\{w>t\}}\frac1{|x|}\dx\bigg)\,\d t
		\ge0
		\]
		for all $w\in\BV_{u_0}$, where equality occurs only if $\{w>t\}$ is a circular disc
		centered at $0$ (up to negligible sets) for \ae{} $t\in{(0,\infty)}$. However, since
		$\{w>t\}\setminus\overline\Omega=\{u_0>t\}\setminus\overline\Omega=\B_{r(t)}\setminus\overline\Omega$ with some $r(t)>1$ (in case $t\ge2^{-\alpha}$ explicitly
        $r(t)=1{+}t^{-1/\alpha}\le3$) is
        prescribed, in the equality case we even deduce $|\{w>t\}\,\triangle\,\B_{r(t)}|=0$ for \ae{} $t>0$. This, however, results in $w=\infty$ \ae{} on $\B_1$, which is impossible for $w\in\BV(\R^2)$. Therefore, the equality case is ruled out, and we have verified \eqref{eq:P-bar>0}.
	\end{proof}
	
	Finally, the case of Example \ref{exp:non-exist} with the modified density $H\in\L^\infty(\R^2)$ (which, by the way, satisfies $H(x)=\div_x(x/\max\{|x|,1\})$ for $x\in\R^2\setminus\partial\B_1$) is similar. Indeed, a reasoning mostly analogous to the proof of Lemma \ref{lem:IC-1/|x|} establishes
	\[
      \int_{A\setminus\B_1}\frac1{|x|}\dx+2|A\cap\B_1|\leq\P(A)
	\]
	for all measurable $A\subseteq\R^2$ with $|A|<\infty$, where equality occurs if and only if $|A\triangle\B_r|=0$ for some $r\in{[1,\infty)}\cup\{0\}$. On the basis of this observations, one can then establish the second non-existence claim in Example \ref{exp:non-exist} by treating the full-space-extended functional
	\[
	\widehat\Phi_\ast[w]\coleq|\D w|(\R^2)-\int_{\R^2\setminus\B_1}\frac{w(x)}{|x|}\dx-2\int_{\B_1}w(x)\dx
	\qq\text{among }w\in\BV_{u_0}
	\]
	right as before.

		\subsection[\texorpdfstring{Failure of consistency for $\mu_+$ and $\mu_-$ not singular to each other}{}]
		{\boldmath Failure of consistency for $\mu_+$ and $\mu_-$ not singular to each other (Example \ref{exp:non-consist})}\label{subsec:non-consist}
		
		We recall that Example \ref{exp:non-consist} works on the unit disc $\Omega=\B_1\subset\R^2$ with the isotropic integrand $\p(x,\xi)=|\xi|$, a boundary datum $u_0$ with trace $u_0(x)=\mathrm{sgn}(x_1)$ for $x\in\partial\B_1$, and the measures $\mu_+=\mu_-=\H^1\ecke(\{0\}{\times}{({-}1,1)})$.		
		
		\begin{proof}[Verification of the claims made in Example \ref{exp:non-consist}]
			For all $w\in\W^{1,1}_{u_0}(\B_1)$ with trace $u_0(x)=\mathrm{sgn}(x_1)$, by absolute continuity along lines (or alternatively by specializing formula \eqref{eq:int-by-parts} below), we have
			\[
              \TV[w]\ge\int_{\B_1}\partial_1w\dx
              =\int_{-1}^1\int_{-\ell(x_2)}^{\ell(x_2)}\partial_1w(x_1,x_2)\dx_1\dx_2
              =\int_{-1}^1\big[u_0(\ell(x_2),x_2)-u_0({-}\ell(x_2),x_2)\big]\dx_2
              =4\,,
			\]
			where we abbreviated $\ell(x_2)\coleq\sqrt{1{-}x_2^2}$. Moreover, for $u\in\BV(\B_1)$ defined by $u(x)\coleq\mathrm{sgn}(x_1)$ for $x\in\B_1$, we have $\TV^{u_0}[u]=|\D\ol{u}|\big(\ol{\B_1}\big)=2\H^1(\{0\}{\times}{[-1,1]})=4$. Thus, by standard consistency and relaxation results for the $\TV$ term alone (that is, the versions of Corollaries \ref{cor:consist}, \ref{cor:relax} with $\mu_+=\mu_-\equiv0$), we can identify
			\[
              \min_{\BV(\B_1)}\TV^{u_0}
              =\inf_{\W^{1,1}_{u_0}(\B_1)}\TV=4
              \qq\qq\text{and}\qq\qq
              \TV_\mathrm{rel}[w]=\TV^{u_0}[w]
              \quad\text{for }w\in\BV(\B_1)\,.
			\]
			Since with $\mu=\mu_+-\mu_-\equiv0$ the measure term of $\Phi$ entirely vanishes, this gives in particular
			\[
			   \inf_{\W^{1,1}_{u_0}(\B_1)}\Phi=4
              \qq\qq\text{and}\qq\qq
              \Phi_\mathrm{rel}[w]=\TV^{u_0}[w]
              \quad\text{for }w\in\BV(\B_1)\,.
			\]
			Moreover, now taking into account $\mu_+=\mu_-=2\H^1\ecke(\{0\}{\times}{(-1,1)})$, for all $w\in\BV(\B_1)$ we observe
			\[
			\widehat\Phi[w]
			=\TV^{u_0}[w]+\int_{\{0\}{\times}{(-1,1)}}\big(w^-{-}w^+\big)\,\d\H^1
			=\TV^{u_0}[w]-|\D w|\big(\{0\}{\times}{(-1,1)}\big)\ge0
			\]
			with equality in case $w=u$ for the specific $u$ considered before. Therefore, with
			\[
              \min_{\BV(\B_1)}\widehat\Phi=0
              \qq\text{and}\qq
              \TV^{u_0}[w]
              =\widehat\Phi[w]+|\D w|\big(\{0\}{\times}{(-1,1)}\big)
              \quad\text{for }w\in\BV(\B_1)
			\]
			the main claims of Example \ref{exp:non-consist} are verified.

            Finally, the claimed IC for $\H^1\ecke(\{0\}{\times}{({-}1,1)})$ follows from \cite[Proposition A.2]{Schmidt25}, where the isotropic IC with constant $1$ has been checked even for $2\H^1\ecke(\{0\}{\times}\R)$ in all of $\R^2$. Alternatively, one can get the same IC by writing $2\H^1\ecke(\{0\}{\times}\R)=\div(\sigma)$ with $\sigma(x)\coleq(\mathrm{sgn}(x_1),0)$ and then arguing as in Sections \ref{subsec:signed-IC} and \ref{subsec:non-finite}.
		\end{proof}

	\subsection[\texorpdfstring{The $\p$-IC does not imply the $\widetilde\p$-IC}{}]
	{\boldmath The $\p$-IC does not imply the $\widetilde\p$-IC (Example \ref{exp:til1})}
	\label{subsec:til1}
	
	With regard to Example \ref{exp:til1}, we comment further only on verifying the claimed $\p$-IC. As indicated earlier, we believe that this can be approached by a general line of argument based on $\P_\p$-outward-minimality. Alternatively, a quick verification can be based once more on Theorem \ref{thm:anis_equivalence_gen_IC_singular} with the vector field $\sigma$, which vanishes on $\Delta$ and coincides with the one of the subsequent Section \ref{subsec:til2} outside $\Delta$. However, here we prefer providing yet another comparably self-contained reasoning tailored out for the concrete situation at hand.
	
	Indeed, the sole information needed on the integrand $\p$ will be
	\begin{equation}\label{eq:til-p-estimate}
		\max\{|\xi_1|,|\xi_2|,-\xi_1-\xi_2\}\leq\p(x,\xi)
		\qq\text{for all }x,\xi\in\R^2\,,
	\end{equation}
	which is easily verified for the choice of $\p$ in the example. Moreover, for $u\in\BV(\R^N)$, a bounded open Lipschitz domain $\Omega$ in $\R^N$, a Borel set $X\subseteq\partial\Omega$, and $i\in\{1,2,\ldots,N\}$ we will rely on the integration-by-parts formula
	\begin{equation}\label{eq:int-by-parts}
		\partial_iu\,(\Omega\cup X)
		={-}\int_{(\partial\Omega)\setminus X}u_{\partial\Omega}^\mathrm{int}\,(\nu_\Omega)_i\,\d\H^{N-1}-\int_Xu_{\partial\Omega}^\mathrm{ext}\,(\nu_\Omega)_i\,\d\H^{N-1}\,,
	\end{equation}
	with the inward normal $\nu_\Omega$ to $\Omega$, as it can be derived initially with $X=\emptyset$ from \cite[Theorem 3.87]{AFP00} and then for general $X$ with the help of \cite[Theorem 3.77]{AFP00}. On this basis, we now turn to:
	
	\begin{proof}[Verification of the $\p$-IC for
		$\P_\p(\Delta,\,\cdot\,)=\H^1\ecke(C_1\cup C_2)+\sqrt2\,\H^1\ecke H$]
		We consider a measurable set $A\Subset\R^2$, for which we may assume $\P(A)<\infty$ and $A\Subset\B_r$, $r\in{(0,\infty)}$,
		and we fix $u=\1_A$ in \eqref{eq:int-by-parts}. Then, since up to $\H^{N-1}$-negligible sets $\max\{(\1_A)_{\partial\Omega}^\mathrm{int},(\1_A)_{\partial\Omega}^\mathrm{ext}\}$ equals $1$ on $A^+\cap\partial\Omega$ and equals $0$ elsewhere on $\partial\Omega$, for any choice of $X$ such that $\{(\1_A)_{\partial\Omega}^\mathrm{int}=0\,,\,(\1_A)_{\partial\Omega}^\mathrm{ext}=1\}\subseteq X$ and $\{(\1_A)_{\partial\Omega}^\mathrm{int}=1\,,\,(\1_A)_{\partial\Omega}^\mathrm{ext}=0\}\subseteq(\partial\Omega)\setminus X$, formula \eqref{eq:int-by-parts} reduces to
		\begin{equation}\label{eq:int-by-parts-1A}
			\partial_i\1_A(\Omega\cup X)
			={-}\int_{A^+\cap\partial\Omega}(\nu_\Omega)_i\,\d\H^{N-1}\,,
		\end{equation}
		and this simplified formula will now be used in implementing the basic idea illustrated in Figure \ref{fig:til1}.
		
		\begin{figure}[H]\centering
                        \includegraphics{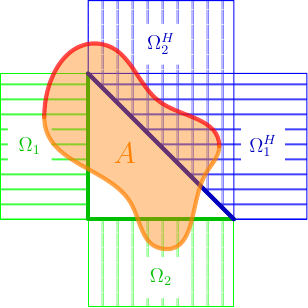}
			\vspace{-1ex}
			\parbox{13.7cm}{\caption{The portions of $C_i$ and $H$ (thick green and blue lines) covered by $A^+$ (orange area) are estimated by integrating along the green and blue stripes, respectively, till they leave $A^+$. This works out even in the blue intersection area, since $\H^{N-1}$-\ae{} point where blue stripes leave $A^+$ in both right and upward direction is in particular contained in the upward red-colored part of $\partial^\ast\!A$ which contributes to $\P_\p(A)$ with its larger $1$-norm length.\label{fig:til1}}}
		\end{figure}
		
		\noindent Specifically, for $C_1=\{0\}\times{(0,1)}$ and $\Omega_1\coleq{({-}R,0)}\times{(0,1)}$ with corresponding $X_1\subseteq\partial\Omega_1$, we obtain from \eqref{eq:int-by-parts-1A} and \eqref{eq:til-p-estimate}, applied in the first and last step, respectively,
		\[
		\H^1(A^+\cap C_1)
		=\partial_1\1_A\,(\Omega_1\cup X_1)
		=\int_{\Omega_1\cup X_1}(\nu_A)_1\,\d|\D\1_A|
		\leq\P_\p(A,\overline{\Omega_1})\,.
		\]
		Analogously, for $C_2={(0,1)}\times\{0\}$ and $\Omega_2\coleq{(0,1)}\times{({-}R,0)}$ with corresponding $X_2\subseteq\partial\Omega_2$, we get
		\[
		\H^1(A^+\cap C_2)
		=\partial_2\1_A\,(\Omega_2\cup X_2)
		=\int_{\Omega_2\cup X_2}(\nu_A)_2\,\d|\D\1_A|
		\leq\P_\p(A,\overline{\Omega_2})\,.
		\]
		For the hypotenuse $H$ of the unit triangle $\Delta$, $\Omega^H_1\coleq\big({(0,R)}\times{(0,1)}\big)\setminus\Delta$, and
		$\Omega^H_2\coleq\big({(0,1)}\times{(0,R)}\big)\setminus\Delta$ with corresponding $X^H_1\subseteq\partial\Omega^H_1$ and $X^H_2\subseteq\partial\Omega^H_2$, we similarly achieve
		\[
		\frac1{\sqrt2}\H^1(A^+\cap H)
		=-\partial_i\1_A\,(\Omega^H_i\cup X^H_i)
		=\int_{\Omega^H_i\cup X^H_i}\big({-}(\nu_A)_i\big)\,\d|\D\1_A|  
		\]
		for $i\in\{1,2\}$. We add up the last equation over $i\in\{1,2\}$ and use \eqref{eq:til-p-estimate} to estimate ${-}(\nu_A)_1-(\nu_A)_2\leq\p(\,\cdot\,,\nu_A)$ on the intersection of domains and $-(\nu_A)_i\leq\p(\,\cdot\,,\nu_A)$ elsewhere. Then we arrive at
		\[
		\sqrt2\,\H^1(A^+\cap H)\leq\P_\p\left(A,\overline{\Omega^H_1}\cup\overline{\Omega^H_2}\right)\,.
		\]
		Finally, we combine the estimates and take into account that $\overline{\Omega_1}$, $\overline{\Omega_2}$, and $\overline{\Omega^H_1}\cup\overline{\Omega^H_2}$ are essentially disjoint. Thus, we conclude with
		\[
		\big(\H^1\ecke(C_1\cup C_2)+\sqrt2\,\H^1\ecke H\big)(A)
		\leq\P_\p(A)\,,
		\]
		as required.
	\end{proof}
	
	\subsection[\texorpdfstring{The $\p$-IC does not even imply the small-volume $\widetilde\p$-IC}{}]
	{\boldmath The $\p$-IC does not even imply the small-volume $\widetilde\p$-IC (Example \ref{exp:til2})}
	\label{subsec:til2}
	
		Here we go into the details of Example \ref{exp:til2}, which involves the self-similar fractal $\Delta_\infty$. We start with:
		
		\begin{proof}[Brief justification of the claim $\H^1(\Delta_\infty)=\sqrt2$]
			The upper bound $\H^1(\Delta_\infty)\leq\sqrt2$ follows from the definition of the Hausdorff measure by covering $\Delta_\infty$ for $k\gg1$ with the $3^k$ triangles of diameter $3^{-k}\sqrt2$ which constitute $\Delta_k$. The lower bound $\H^1(\Delta_\infty)\ge\sqrt2$ results from the observation that the orthogonal projection of $\Delta_\infty$ on the second diagonal fully covers the line segment $\big\{(t,-t):t\in{\big[{-}\frac12,\frac12\big]}\big\}$ of length $\sqrt2$.
		\end{proof}
		
		In addition, we record that the polar $\p^\circ$ of the anisotropic integrand $\p$ defined already in Example \ref{exp:til1} takes the form
		\begin{equation}\label{eq:p^o}
			\p^\circ(x,\xi^\ast)=
			\begin{cases}
				|\xi^\ast|&\text{if }\xi^\ast_2\ge0\\
				\max\{|\xi^\ast_1|,|\xi^\ast_2|\}&\text{if }\xi^\ast_2\le0
			\end{cases}
		\end{equation}
		and then we are left merely with the verification of the claimed $\p$-IC for the equi-distributed measure $\mu$ on $\Delta_\infty$. This verification will now be achieved via a construction partially inspired by \cite{Hutchinson81} and once more with the help of Theorem \ref{thm:anis_equivalence_gen_IC_singular}.
	
	\begin{proof}[Verification that $\mu=2\sqrt2\,\H^1\ecke\Delta_\infty$ satisfies the $\p$-IC in $\R^2$ with constant $1$]
			We define $\sigma_k\in\L^\infty(\R^2,\R^2)$ as illustrated in Figure \ref{fig:sigma_k} for $k\in\{1,2\}$ and subsequently formally described for all $k\in\N$.

		\begin{figure}[ht]
			\centering
                        \includegraphics{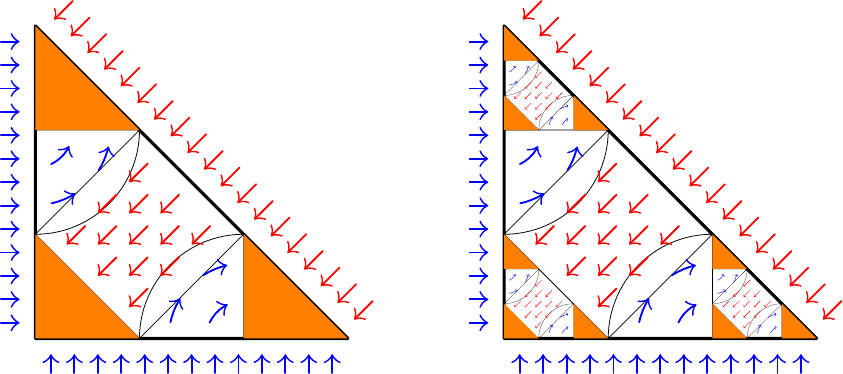}
			\caption{An illustration of $\sigma_k$ for $k\in\{1,2\}$, where
				$\sigma_k\equiv0$ on $\Delta_k$ (orange regions),
				$\sigma_k\in{[0,1]}^2\cap\partial\B_1$ on blue-arrow regions,
				$\sigma_k\equiv({-}1,{-}1)$ on red-arrow regions, and
				$\sigma_k\in{[{-}1,0]}^2$ in red-blue-superposition
				regions.\label{fig:sigma_k}}
		\end{figure}
		
			We fix the values of $\sigma_k$ outside $\Delta_0$ as $\sigma_k\colequiv(1,0)$ on ${(-\infty,0)}{\times}{(0,1)}$ and $\sigma_k\colequiv(0,1)$ on ${(0,1)}{\times}{(-\infty,0)}$ and $\sigma_k(x)\coleq({-}1,{-}1)$ if $|x_2{-}x_1|<1<x_1{+}x_2$, and we complement these choices with $\sigma_k\colequiv0$ elsewhere outside $\Delta_0$. In addition, the values on $\Delta_0$ are defined iteratively. First, we define $\sigma_1$ on $\Delta_0$ by setting (with slight abuse of notation for the characteristic functions)
			\[
			\sigma_1(x)\coleq\frac{(\frac23{-}x_2,x_1)}{|(\frac23{-}x_2,x_1)|}\1_{|(x_1,x_2{-}\frac23)|<\frac13}
			+({-}1,{-}1)\1_{|x_2{-}x_1|<\frac13}
			+\frac{(x_2,\frac23-x_1)}{|(x_2,\frac23-x_1)|}\1_{|(x_1{-}\frac23,x_2)|<\frac13}
			\qq\text{for }x\in\Delta_0\setminus\Delta_1
			\]
			and $\sigma_1\colequiv0$ on $\Delta_1$. Then, for $2\leq k\in\N$, we define $\sigma_k$ on $\Delta_0$ by setting $\sigma_k(T_i(x))\coleq\sigma_{k-1}(x)$ for $x\in\Delta_{k-1}$, $i\in\{1,2,3\}$ and $\sigma_k(x)\coleq\sigma_{k-1}(x)$ for $x\in\Delta_0\setminus\Delta_{k-1}$.
			
			Despite the somewhat lengthy formal description of $\sigma_k$, we can read off that all values of $\sigma_k$ remain in $({[0,1]}^2\cap\partial\B_1)\cup({[{-}1,0]}^2)$, and therefore \eqref{eq:p^o} implies $\p^\circ(\,\cdot\,,\sigma_k)\le1$ \ae{} on $\R^2$. Moreover, we decisively record that only the normal component of $\sigma_k$ jumps at the boundary of the superposition regions illustrated in Figure \ref{fig:sigma_k}. Since these normal jumps do not contribute to $\div(\sigma_k)$, it then turns out that $\div(\sigma_k)={-}\theta_k\H^1\ecke\partial\Delta_k$ in the sense of distributions on $\R^2$ with $\theta_k\equiv1$ on the horizontal and vertical portions of $\partial\Delta_k$, but $\theta_k\equiv\sqrt2$ on the diagonal ones.
			
			Finally, we pass to the limit $k\to\infty$. Since for \ae{} fixed $x\in\R^2$ the value $\sigma_k(x)$ is eventually constant for $k\gg1$ and moreover $\sigma_k$ are bounded in $\L^\infty(\R^2,\R^2)$, it turns out that $\sigma_k$ converge \ae{} on $\R^2$ and weak-$\ast$ in $\L^\infty(\R^2,\R^2)$ to some $\sigma\in\L^\infty(\R^2,\R^2)$ with $\p^\circ(\,\cdot\,,\sigma)\le1$ \ae{} on $\R^2$. In addition, since $(\theta_k\H^1\ecke\partial\Delta_k)(T)=(2\sqrt2\,\H^1\ecke\Delta_\infty)(T)$ holds for each of the $3^k$ triangles $T$ which $\Delta_k$ consists of, a standard argument shows that $\theta_k\H^1\ecke\partial\Delta_k$ weak-$\ast$ converge to $2\sqrt2\,\H^1\ecke\Delta_\infty$ as measures on $\R^2$. In view of the convergences we can then pass to the limit with the condition $\div(\sigma_k)={-}\theta_k\H^1\ecke\partial\Delta_k$ and obtain $\div(\sigma)={-}2\sqrt2\,\H^1\ecke\Delta_\infty$ on $\R^2$. At this stage, applying Theorem \ref{thm:anis_equivalence_gen_IC_singular} as in the earlier Sections \ref{subsec:signed-IC} and \ref{subsec:non-finite}, we conclude that $2\sqrt2\,\H^1\ecke\Delta_\infty$ indeed satisfies the $\p$-IC in $\R^2$ with constant $1$.
	\end{proof}
	
	\section{Refined semicontinuity of parametric functionals}\label{sec:par-lsc}

    In this section we work in a parametric setting and mainly aim at refining the finite-measure case of the semicontinuity results in \cite[Theorems 1.2, 4.1]{Schmidt25}. Indeed, while these previous results deal with the isotropic perimeter and with separate small-volume ICs for $\mu_+$ and $\mu_-$, we here cover general anisotropic perimeters and reach the ``truly signed situation'' with small-volume ICs for $(\mu_-,\mu_+)$ and $(\mu_+,\mu_-)$, where we can even allow for $\mu_+\not\perp\mu_-$. We remark that in this generality we cannot deduce suitable ICs for $\mu_+$ and $\mu_-$ alone (as one can see at hand of Example \ref{examp:unrectifiable}), and thus we cannot follow the approach of \cite[Section 4]{Schmidt25}, which first treats the functionals with one of the two measures only. Rather we will implement a related, but somewhat more delicate reasoning, which can cope with both measures and the two corresponding types of cancellation effects at the same time.

    However, before reaching the main statement in Theorem \ref{thm:anis_semicont_funct_per_A+A1_gen}
    we deal with a key lemma, which closely resembles and slightly refines \cite[Lemma 4.4]{Schmidt25}. In fact, in the specific case of a finite measure $\mu$ we can here dispense with a certain technical condition imposed in \cite[Lemma 4.4]{Schmidt25} (see \eqref{eq:M(eps)-IC} below) and can work with the vanishing on $\H^{N-1}$-negligible sets as our sole hypothesis on $\mu$. This refinement is helpful in coping in the sequel with semicontinuity under assumption of merely our ``signed'' ICs.
    
	\begin{lem}[good exterior approximation]
		\label{lem:good-ext-approx}
		We consider a \emph{finite} non-negative Radon measure $\mu$ on $\R^N$ such that $\mu(Z)=0$ for every $\H^{N-1}$-negligible $Z\subseteq\R^N$, a Borel function $\p\colon\R^N\times\R^N\to{[0,\infty)}$ such that $\xi\mapsto \p(x,\xi)$ is positively homogeneous of degree $1$ and \ka the upper bound in\kz{} \eqref{eq:comp-p0} is valid, and an arbitrary $\eps>0$. If\/ $\1_{A_k}\in\BV(\R^N)$ converge in $\L^1(\R^N)$ to $\1_A\in\BV(\R^N)$, then there exists a Borel set\/ $S\subseteq\R^N$ with $\1_S\in\BV(\R^N)$ such that we have
		\[
		A^+\subseteq\Int(S)\,,\qq\qq
		\mu\big(\overline S\big)<\mu(A^+)+3\eps\,,
		\qq\qq\text{and}\qq\qq
		\liminf_{k\to\infty}\P_\p(S,A_k^+)<\beta\eps\,.
		\]
	\end{lem}
	
	\begin{proof}
		An analogous claim for the case $\p(x,\xi)=|\xi|$ has been established in \cite[Lemma 4.4]{Schmidt25} under the additional assumption that there exist $M\in{[0,\infty)}$ and $\delta>0$ such that
		\begin{equation}\label{eq:M(eps)-IC}
			\mu(A^+)\le M\P(A)+\eps
			\qq\text{for all measurable }A\subseteq\R^N\text{ with }|A|<\delta\,.
		\end{equation}
		
		In the present case with finite $\mu$ we will now employ a contradiction argument to show 
		that \eqref{eq:M(eps)-IC} is automatically valid for some $M\in{[0,\infty)}$ even with $\delta=\infty$ (where the reader should carefully observe that this does \emph{not} mean having a small-volume IC for $\mu$, since our $M$ will actually depend on $\eps$). We base the contradiction argument on the assumption that \eqref{eq:M(eps)-IC} with $\delta=\infty$ fails for arbitrarily large $M$ and thus infer, for each $\ell\in\N$, the existence of a measurable set $A_\ell\subseteq\R^N$ with $|A_\ell|<\infty$ and
		$\mu(A_\ell^+)>\ell^2\P(A_\ell)+\eps$. In particular we record
		\[
		\P(A_\ell)<\mu(\R^N)\ell^{-2}
		\qq\qq\text{and}\qq\qq
		\mu(A_\ell^+)>\eps
		\]
		for all $\ell\in\N$, and we introduce
		\[
		E \coleq\bigcap_{k=1}^\infty\bigcup_{\ell=k}^\infty A_\ell^+\,.
		\]
		Now, from the subadditivity in \eqref{eq:subadd_cap} and the characterization in Proposition \ref{prop:per_cap}, on one hand we find
		\[
		\mathrm{Cap}_1(E)
		\le\mathrm{Cap}_1\bigg(\bigcup_{\ell=k}^\infty A_\ell^+\bigg)
		\le\sum_{\ell=k}^\infty\mathrm{Cap}_1(A_\ell^+)
		\le\sum_{\ell=k}^\infty\P(A_\ell)
		\le\mu(\R^N)\sum_{\ell=k}^\infty\ell^{-2}
		\]
		for all $k\in\N$. From this, taking into account the finiteness of $\mu$ we infer first $\mathrm{Cap}_1(E)=0$, then (recalling Proposition \ref{prop:negl_cap})
		$\H^{N-1}(E)=0$, and by assumption on $\mu$ finally $\mu(E)=0$. On the other hand, we have
		\[
		\mu(E)
		=\lim_{k\to\infty}\mu\bigg(\bigcup_{\ell=k}^\infty A_\ell^+\bigg)
		\ge\liminf_{k\to\infty}\mu(A_k^+)\ge\eps\,.
		\]
		This contradicts the previous finding $\mu(E)=0$ and verifies the existence of $M$ such that \eqref{eq:M(eps)-IC} is valid.
		
		On the basis of \eqref{eq:M(eps)-IC} we may then deduce the claim of Lemma \ref{lem:good-ext-approx} for the case $\p(x,\xi)=|\xi|$ with $\beta =1$ by applying \cite[Lemma 4.4]{Schmidt25}. The claim for general $\p$ directly follows by exploiting the upper bound in \eqref{eq:comp-p0}.
	\end{proof}

    With Lemma \ref{lem:good-ext-approx} at hand we now turn to the announced main statement.
	
	\begin{thm}[lower semicontinuity of anisotropic parametric functionals]\label{thm:anis_semicont_funct_per_A+A1_gen}
		We impose Assumption \ref{assum:phi}\eqref{case:conv-lsc} and, for finite non-negative Radon measures $\mu_\pm$ on $\R^N$, assume $\mu_\pm(Z)=0$ for every $\H^{N-1}$-negligible $Z\subseteq\R^N$. If $(\mu_-,\mu_+)$ satisfies the small-volume $\p$-IC in $\R^N$ with constant $1$ and $(\mu_+,\mu_-)$ satisfies the small-volume $\widetilde\p$-IC in $\R^N$ with constant $1$, then it holds 
		\begin{equation} \label{eq:lsc}
			\liminf_{k \to \infty} \left[ \P_{\p}(A_k) + \mu_+(A_k^1) - \mu_-(A_k^+)\right]
			\geq \P_{\p}(A) + \mu_+(A^1) - \mu_-(A^+)
		\end{equation}
		whenever $A_k$ and $A$ are measurable in $\R^N$ and $A_k$ converge locally in measure to $A$. 
	\end{thm}
	
	\begin{proof}
		We first treat the case that $\bigcup_{k=1}^\infty
                A_k\Subset\B_r$ (and then also $A\Subset\B_r$) for some large
                ball $\B_r$. Clearly, in this case we need not distinguish
                between local and global convergence in measure of $A_k$ to
                $A$. Possibly passing to a subsequence, we assume that
                $\lim_{k\to\infty}\big[\P_\p(A_k){+}\mu_+(A_k^1){-}\mu_-(A_k^+)\big]$
                exists and is finite. By finiteness of $\mu_-$ and
                \eqref{eq:comp-p0}, this implies
                $\limsup_{k\to\infty}\P(A_k)<\infty$ and then also
                $\P(A)<\infty$. In particular we have
                $\1_{A_k},\1_A\in\BV(\R^N)$ for $k\gg1$. We now fix an arbitrary $\eps>0$ and apply Lemma \ref{lem:good-ext-approx} to $\mu_-$ with $\widetilde\p$ instead of $\p$ to find a Borel set $S$ with $\1_S\in\BV(\R^N)$ and a subsequence of $(A_k)_k$ (which we do not relabel) such that
		\[
		A^+\subseteq\Int(S)\,,\qq\qq
		\mu_-\big(\overline S\big)<\mu_-(A^+)+3\eps\,,\qq\qq
		\lim_{k\to\infty}\P_{\widetilde\p}(S,A_k^+)<\beta\eps\,.
		\]
		Moreover, we apply Lemma \ref{lem:good-ext-approx} also to $\mu_+$ and for the complements $\B_r\setminus A_k$, $\B_r\setminus A$ to obtain another Borel set $S'\subseteq\R^N$ with $\1_{S'}\in\BV(\R^N)$ and yet another subsequence such that
		\[
		(\B_r\setminus A)^+\subseteq\Int(S')\,,\qq\qq
		\mu_+\big(\overline{S'}\big)<\mu_+((\B_r\setminus A)^+)+3\eps\,,\qq\qq
		\lim_{k\to\infty}\P_\p(S',(\B_r\setminus A_k)^+)<\beta\eps\,.
		\]
		In the sequel we work with $S$ and the complement $R\coleq\B_r\setminus S'$ of $S'$, which in view of $A_k\cup A\Subset\B_r$, $\partial\B_r\subseteq\Int(S')$, and $R \Subset \B_r$ has the dual properties $\1_R\in\BV(\R^N)$,
		\[
		\ol{R}\subseteq A^1\,,\qq\qq
		\mu_+\big(\Int(R)\big)>\mu_+(A^1)-3\eps\,,\qq\qq
		\lim_{k\to\infty}\P_{\widetilde\p}(R,(A_k^\c)^+)<\beta\eps\,.
		\]  
		We record in particular $\overline R\subseteq\Int(S)$ and start our line of estimates by splitting terms in the sense of the (in)equalities
		\begin{align*}
			\P_\p(A_k)&=\P_\p\big(A_k,\Int(S)\setminus\ol{R}\big)
			+\P_\p\big(A_k,\Int(S)^\c\big)+\P_{\widetilde\p}\big(A_k^\c,\ol{R}\big)\,,\\
			\mu_+(A_k^1)&\ge\mu_+\big(A_k^1\setminus\Int(S)\big)-\mu_+\big((A_k^\c)^+\cap\Int(R)\big)+\mu_+\big(\Int(R)\big)\,,\\
			\mu_-(A_k^+)&\le\mu_-\big(A_k^+\setminus\ol{S}\big)-\mu_-\big((A_k^\c)^1\cap\ol{R}\big)+\mu_-\big(\ol{S}\big)\,.
		\end{align*}
		The splittings combine to
		\begin{equation}\label{eq:lsc-main}\begin{aligned}
			\lim_{k\to\infty}\big[\P_\p(A_k)+\mu_+(A_k^1)-\mu_-(A_k^+)\big]
			&\ge\liminf_{k\to\infty}\P_\p\big(A_k,\Int(S)\setminus\ol{R}\big)\\
				&\quad+\liminf_{k\to\infty}\Big[\P_\p\big(A_k,\Int(S)^\c\big)
				+\mu_+\big(A_k^1\setminus\Int(S)\big)
				-\mu_-\big(A_k^+\setminus\ol{S}\big)\Big]\\
				&\quad+\liminf_{k\to\infty}\Big[\P_{\widetilde\p}\big(A_k^\c,\ol{R}\big)
				-\mu_+\big((A_k^\c)^+\cap\Int(R)\big)
				+\mu_-\big((A_k^\c)^1\cap\ol{R}\big)\Big]\\
				&\quad+\mu_+\big(\Int(R)\big)-\mu_-\big(\ol{S}\big)\,,
		\end{aligned}\end{equation}
		and the terms on the right-hand side of \eqref{eq:lsc-main} are now estimated separately. For the first term, we apply Theorem \ref{anis_thm:LSC_TV_1} on the open set $\Int(S)\setminus\ol{R}$ and exploit the inclusions $\ol{R}\subseteq A^1\subseteq A^+\subseteq\Int(S)$ in the estimate
		\begin{equation}\label{eq:lsc-first}
			\liminf_{k\to\infty}\P_\p(A_k,\Int(S)\setminus\ol{R})
			\geq\P_\p(A,\Int(S)\setminus\ol{R})
			\geq\P_\p(A,A^+\setminus A^1)
			=\P_\p(A)\,.
		\end{equation}
		In order to	control the second term, we first record that in view of $A^+\subseteq S$ we get  $|A_k{\setminus}S|\le|A_k{\setminus}A|\le|A_k\triangle A|$ and that consequently the assumed convergence in measure implies $\lim_{k\to\infty}|A_k{\setminus}S|=0$. This allows the crucial application of the small-volume $\p$-IC for $(\mu_-,\mu_+)$ to $A_k\setminus S$ for $k\gg1$. Taking into account the inclusions $A_k^+\setminus\overline{S}\subseteq(A_k\setminus S)^+$, $(A_k\setminus S)^1\subseteq A_k^1\setminus\Int(S)$, $S^0\subseteq\Int(S)^\c$ and bringing in Lemma \ref{lem:P(AcapB),P(A-S)}, we deduce
		\[\begin{aligned}
			\mu_-\big(A_k^+\setminus\ol{S}\big)
			-\mu_+\big(A_k^1\setminus\Int(S)\big)
			&\leq\mu_-((A_k\setminus S)^+)
			-\mu_+((A_k\setminus S)^1\big)
			\leq\P_\p(A_k\setminus S)+\eps\\
			&\leq\P_\p(A_k,S^0)+\P_{\widetilde\p}(S,A_k^+)+\eps \\
			&\leq\P_\p(A_k,\Int(S)^\c)+\P_{\widetilde\p}(S,A_k^+)+\eps
		\end{aligned}\]
		for $k\gg1$. Now we rearrange terms in the resulting estimate and take limits. Then, also employing the last property from the choice of $S$, we conclude
		\begin{equation}\label{eq:lsc-second}
			\liminf_{k\to\infty}\big[\P_\p(A_k,\Int(S)^\c)+\mu_+\big(A_k^1\setminus\Int(S)\big)-\mu_-\big(A_k^+\setminus\ol{S}\big)\big]
			\ge-\lim_{k\to\infty}\P_{\widetilde\p}(S,A_k^+)-\eps
			>{-}(\beta{+}1)\eps\,.
		\end{equation}
		For the third term on the right-hand side of \eqref{eq:lsc-main}, by applying the small-volume $\widetilde\p$-IC for $(\mu_+,\mu_-)$ to $A_k^\c\cap R$ again for $k \gg 1$ and by analogous reasoning, we have the dual estimate
		\[\begin{aligned}
			\mu_+\big((A_k^\c)^+\cap\Int(R)\big)
			-\mu_-\big((A_k^\c)^1\cap\ol{R}\big)
			&\leq\mu_+((A_k^\c\cap R)^+)
			-\mu_-((A_k^\c\cap R)^1\big)
			\leq\P_{\widetilde\p}(A_k^\c\cap R)+\eps\\
			&\leq\P_{\widetilde\p}(A_k^\c,R^1)+\P_{\widetilde\p}(R,(A_k^\c)^+)+\eps \\
			&\leq\P_{\widetilde\p}(A_k^\c,\ol{R})+\P_{\widetilde\p}(R,(A_k^\c)^+)+\eps
		\end{aligned}\]
		for $k\gg1$, and this leads to
		\begin{equation}\label{eq:lsc-third}
			\liminf_{k\to\infty}\big[\P_{\widetilde\p}(A_k^\c,\ol{R})-\mu_+\big((A_k^\c)^+\cap\Int(R)\big)+\mu_-\big((A_k^\c)^1\cap\ol{R}\big)\big]
			\ge-\lim_{k\to\infty}\P_{\widetilde\p}(R,(A_k^\c)^+)-\eps
			>{-}(\beta{+}1)\eps\,.
		\end{equation}
		Finally, for the last terms in \eqref{eq:lsc-main}, it suffices to recall that by choice of $R$ and $S$ we have
		\begin{equation}\label{eq:lsc-last}
			\mu_+\big(\Int(R)\big)>\mu_+(A^1)-3\eps
			\qq\text{and}\qq
			\mu_-\big(\ol{S}\big)<\mu_-(A^+)+3\eps\,.
		\end{equation}
		Collecting the estimates \eqref{eq:lsc-main}, \eqref{eq:lsc-first}, \eqref{eq:lsc-second}, \eqref{eq:lsc-third}, and \eqref{eq:lsc-last}, we finally arrive at
		\[
		\lim_{k\to\infty}\big[\P_\p(A_k)+\mu_+(A_k^1)-\mu_-(A_k^+)\big]
		\ge\P_\p(A)+\mu_+(A^1)-\mu_-(A^+)-(2\beta{+}8)\eps\,.
		\]
		Since $\eps>0$ is arbitrary, with this we have proven the claim \eqref{eq:lsc} under the initial uniform boundedness assumption for $A_k$ and $A$.
		
		In order to extend \eqref{eq:lsc} to the general case without boundedness assumption, we assume once more that $\lim_{k\to\infty}\big[\P_\p(A_k){+}\mu_+(A_k^1){-}\mu_-(A_k^+)\big]$ exists and is finite, and we infer $\P(A_k)+\P(A)<\infty$ for $k\gg1$. By the isoperimetric inequality, we have either $|A|<\infty$ or $|A^\c|<\infty$, and we now consider the first of these two cases. We fix an arbitrary $\eps>0$ and then choose $r$ large enough for having 
		\begin{equation*}
			|A\setminus\B_r|<\eps \,,
			\qq \P(A,(\B_r)^\c)<\eps \,, \qq \text{ and }
			\qq \mu_\pm((\B_r)^\c)<\eps \,,  
		\end{equation*}
		where the latter exploits finiteness of $\mu_\pm$. From the local convergence of $A_k$ to $A$ we infer $|A_k\cap(\B_{r+1}\setminus\B_r)|<\eps$ for $k\gg1$, and in view of
		\[
		\int_r^{r+1}\H^{N-1}(A_k^1\cap\partial\B_\varrho)\,\d\varrho
		\le|A_k\cap(\B_{r+1}\setminus\B_r)|<\eps
		\]
		we can choose $\varrho_k\in{[r,r+1]}$ such that $\H^{N-1}(A_k^1\cap\partial\B_{\varrho_k})<\eps$. Via Lemma \ref{lem:P(AcapB),P(A-S)} and \eqref{eq:comp-p0} this implies
		\[
		\P_\p(A_k\cap\B_{\varrho_k})
		\le\P_\p(A_k,(\B_{\varrho_k})^+)+\P_\p(\B_{\varrho_k},A_k^1)
		\le\P_\p(A_k)+\beta\H^{N-1}(A_k^1\cap\partial\B_{\varrho_k})
		<\P_\p(A_k)+\beta\eps\,
		\]
		for $k\gg1$. Passing to a subsequence, we can additionally assume that $\varrho_k\in{[r,r+1]}$ converge to $\varrho\in{[r,r+1]}$. Then, since $A_k\cap\B_{\varrho_k}$ are all contained in $\B_{r+1}$ and converge to $A\cap\B_\varrho$ in measure, we may use the claim established under the uniform boundedness assumption for this sequence. We additionally manipulate the left-hand side by using the inclusions $(A_k\cap\B_{\varrho_k})^1\subseteq A_k^1$ and $(A_k\cap\B_{\varrho_k})^+\supseteq A_k^+\cap\B_r$, and we manipulate the right-hand side by exploiting $\P_\p(A\cap\B_\varrho)\ge\P_\p(A,\B_r)$ together with the inclusions $(A\cap\B_\varrho)^1\supseteq A^1\cap\B_r$ and $(A\cap\B_\varrho)^+\subseteq A^+$. In this way we deduce
		\[
		\liminf_{k\to\infty}\big[\P_\p(A_k\cap\B_{\varrho_k})+\mu_+(A_k^1)-\mu_-(A_k^+\cap\B_r)\big]
		\ge\P_\p(A,\B_r)+\mu_+(A^1\cap\B_r)-\mu_-(A^+)\,.
		\]
		By choice of $r$ and $\varrho_k$ we can pass on to
		\[
		\liminf_{k\to\infty}\big[\P_\p(A_k)+\mu_+(A_k^1)-\mu_-(A_k^+)\big]+(\beta+1)\eps
		\ge\P_\p(A)+\mu_+(A^1)-\mu_-(A^+)-(\beta+1)\eps
		\]
		and by arbitrariness of $\eps$ arrive at \eqref{eq:lsc}. In the remaining case $|A^\c|<\infty$, we exploit the rewriting of the functional (here spelled out for $A$, but valid in the same way for $A_k$)
		\[
		\P_\p(A)+\mu_+(A^1)-\mu_-(A^+)
		=\P_{\widetilde\p}(A^\c)+\mu_-((A^\c)^1)-\mu_+((A^\c)^+)+C\,,
		\]
		where the finite constant $C\coleq\mu_+(\R^N)-\mu_-(\R^N)$ does not affect semicontinuity at all. Thus, we can run the same reasoning as in the previous case, now with $A_k^\c$ and $A^\c$ in place of $A_k$ and $A$, with $\widetilde\p$ in place of $\p$, and with the roles of $\mu_+$ and $\mu_-$ just exchanged (where the combination of such modifications also matches with the assumed ICs). This completes the deduction of \eqref{eq:lsc} in the general case.
	\end{proof}	

    When eventually applying Theorem \ref{thm:anis_semicont_funct_per_A+A1_gen}, we will choose the measures $\mu_\pm$ as extensions from an open set $U\subseteq\R^N$ to all of $\R^N$. The following lemma adapts an argument from the proof of \cite[Lemma 7.3]{Schmidt25} and ensures that also the relevant ICs carry over suitably from $U$ to $\R^N$; compare with the subsequent Remark \ref{rem:equiv_small_strong_IC}.
    
	\begin{lem}[small-volume IC for extended measures]\label{anis_lem:equiv_small_IC}
	   We impose Assumption \ref{assum:phi}. For finite non-negative Radon measures $\mu$ and $\nu$ on open $U \subseteq \R^N$, we introduce the extended measures $\ol{\mu}$ and $\ol{\nu}$, defined by $\ol{\mu}(S)\coleq \mu(S \cap U)$ and $\ol{\nu}(S)\coleq \nu(S \cap U)$ for Borel sets $S \subseteq \R^N$. If $(\mu,\nu)$ satisfies the small-volume $\p$-IC in $U$ with constant $C\in{[0,\infty)}$, then also $(\ol{\mu},\ol{\nu})$ satisfies the small-volume $\p$-IC in $\R^N$ with constant $C$.
	\end{lem}
			
		\begin{proof}
            We assume the small-volume $\p$-IC for $(\mu,\nu)$ in $U$ and will deduce even a seemingly stronger version of the small-volume $\p$-IC for $(\ol{\mu},\ol{\nu})$ in $\R^N$, namely that, for every $\eps>0$, there exists some $\delta>0$ such that
            \begin{equation}\label{eq:IC-II-for-ol-mu}
              \ol{\mu}(A^+) - \ol{\nu}(A^1) \leq C \P_{\p}(A,U)+\eps
            \end{equation}
            for all measurable $A\Subset\R^N$ with $|A|<\delta$. To this end, we denote by $d\colon \R^N \to {[0,\infty]}$ the distance function such that $d(x)\coleq\dist(x,U^\c)$ for $x \in \R^N$. We record that $d$ is Lipschitz continuous in $\R^N$ with Lipschitz constant $1$ and by Rademacher's theorem satisfies $|\nabla d| \leq 1$ \ae{} on $\R^N$. Since $\mu$ and consequently also $\ol{\mu}$ are finite measures, we have $\lim_{t\searrow0}\ol{\mu}\big(\big\{d<t\big\}\big)=\ol{\mu}\big(U^\c\big)=0$. Therefore, given an arbitrary $\eps>0$, we may choose a value $t_0>0$ such that
            \[
              \ol{\mu}\big(   \big\{ d < t_0 \big\} \big) < \frac\eps3\,.
            \]
            At this stage, we fix $\delta' > 0$ such the small-volume $\p$-IC for $(\mu,\nu)$ in $U$ with constant $C$ (see \eqref{eq:signed-SVIC} for the underlying definition) applies with $\eps/3$ and $\delta'$ in place of $\eps$ and $\delta$, and we set $\delta\coleq\min\big\{ \delta', \frac{t_0 \eps}{3\beta C}\big\}$. For proving \eqref{eq:IC-II-for-ol-mu}, we consider a measurable $A\Subset\R^N$ with $|A|<\delta$ as above and in view of the lower bound in \eqref{anis_eq:bd_PER} may additionally assume $\P(A,U)<\infty$. With the help of Theorem \ref{thm:coarea_Lip} we then deduce
			\begin{align*}
				\int_{0}^{t_0} \H^{N-1} \left( A^+ \cap \left\{ d = t \right\} \right)  \dt 
				\leq
				\int_{A} |\nabla d|\,\dx 
				\leq |A| < \frac{t_0 \eps}{3 \beta C}\,.
			\end{align*}
			Thus, there exists $\ol{t} \in (0,t_0)$ such that
			\begin{equation}\label{eq:anis_III_primo}
				\H^{N-1} \left( A^+ \cap \left\{ d = \ol{t} \right\} \right) 
				\leq \frac{1}{t_0} \int_{0}^{t_0} \H^{N-1}\left( A^+ \cap \left\{ d = t \right\} \right) \dt 
				< \frac{\eps}{3\beta C}\,.
			\end{equation}
            At this point we introduce $E\coleq A \cap  \left\{ d > \ol{t} \right\} \Subset U$. Then Remark \ref{rem:inclusions_positive_density} gives $A^+ \cap  \left\{ d > \ol{t} \right\} \subseteq E^+\Subset U$, and thus taking into account the small-volume $\p$-IC for $(\mu,\nu)$ in $U$ together with $|E| \leq |A| < \delta'$ we find
            \begin{equation}\label{eq:mu_sv_est}\begin{aligned}
			   \ol{\mu}\big(A^+\big) - \ol{\nu}\big(A^1\big)
			   &\leq \ol{\mu} \big( A^+ \cap \big\{ d > \ol{t} \big\} \big) - \ol{\nu}\big(E^1\big) + \ol{\mu} \big(\big\{ d < t_0 \big\} \big) \\
			   &\leq  \mu\big(E^+\big) - \nu\big(E^1\big) + \frac{\eps}{3}
			   \leq C \P_\p(E) + \frac{2\eps}{3}\,. 
            \end{aligned}\end{equation}
		      We further estimate $\P_\p(E)$ by exploiting Lemma \ref{lem:P(AcapB),P(A-S)}, the upper bound in \eqref{anis_eq:bd_PER}, Theorem \ref{thm:DG}, and the previously derived estimate \eqref{eq:anis_III_primo}. In this way, we deduce
			\begin{equation}\label{eq:per_est}\begin{aligned}
				\P_\p(E)
                =\P_\p(E,U)
				&\leq\P_\p\big( A, \big\{ d > \ol{t} \big\}^1 \cap U \big) + \P_\p \big( \big\{ d > \ol{t} \big\}, A^+ \cap U \big)\\
				& \leq \P_\p(A,U) + \beta \, \P\big( \big\{ d > \ol{t} \big\}, A^+ \big) \\
				& = \P_\p(A,U) + \beta \, \H^{N-1}\big(A^+ \cap \big\{ d = \ol{t} \big\} \big) \\
                & \leq \P_\p(A,U) + \frac\eps{3C}\,.
		      \end{aligned}\end{equation}
		      Combining \eqref{eq:mu_sv_est} and \eqref{eq:per_est} we arrive at
		      \[
		       \ol{\mu}\big(A^+\big) - \ol{\nu}\big(A^1\big)
		       \leq C \left( \P_\p(A,U) + \frac\eps{3C} \right) + \frac{2\eps}3
		       = C \P_\p(A,U) + \eps\,.
		      \]
            This establishes \eqref{eq:IC-II-for-ol-mu} and completes the proof.
		\end{proof}

        We stress that Lemma \ref{anis_lem:equiv_small_IC} and its proof crucially rely on the small-volume feature of the ICs in the following sense: Even when starting from a $\p$-IC for $(\mu,\nu)$, in general one may only expect the \emph{small-volume} $\p$-IC, but not the $\p$-IC for $(\ol{\mu},\ol{\nu})$. Anyway, this will be enough for our purposes and will eventually be exploited in the manner explained next.
		
		\begin{rem}[on the interplay of Theorem \ref{thm:anis_semicont_funct_per_A+A1_gen} and Lemma \ref{anis_lem:equiv_small_IC}]\label{rem:equiv_small_strong_IC}
            In the non-parametric framework of our main results, $\mu_\pm$ are Radon measures on an open set\/ $\Omega\subseteq\R^N$ such that\/ $(\mu_-,\mu_+)$ and $(\mu_+,\mu_-)$ satisfy a $\p$-IC and a $\widetilde\p$-IC in $\Omega$, respectively. When in this situation we extend to measures $\ol{\mu_\pm}$ on $\R^N$ by $\ol{\mu_\pm}(S)\coleq\mu_\pm(S\cap\Omega)$, Lemma \ref{anis_lem:equiv_small_IC} guarantees that $(\ol{\mu_-},\ol{\mu_+})$ and $(\ol{\mu_+},\ol{\mu_-})$ satisfy at least the small-volume $\p$-IC and the small-volume $\widetilde\p$-IC in $\R^N$. It is only the validity of the latter small-volume ICs on all of\/ $\R^N$ which indeed paves the way for eventual applications of Theorem \ref{thm:anis_semicont_funct_per_A+A1_gen} with $\ol{\mu_\pm}$ replacing $\mu_\pm$ in its statement.
		\end{rem}

    Before closing this section we wish to point out that a slightly less general version of Theorem \ref{thm:anis_semicont_funct_per_A+A1_gen} can also be derived in closer analogy with the approach of \cite{Schmidt25}. This approach requires separate small-volume ICs for $\mu_+$ alone and $\mu_-$ alone and then derives semicontinuity at first separately for functionals of type $\P(A){+}\mu_+\big(A^1\big)$ and $\P(A){-}\mu_-\big(A^+\big)$ and only eventually for the general type $\P(A){+}\mu_+\big(A^1\big){-}\mu_-\big(A^+\big)$; see \cite[Section 4]{Schmidt25} for the details. In principle this strategy extends to our setting with $\p$-anisotropic perimeters as soon as suitable separate ICs for $\mu_+$ alone and $\mu_-$ alone are available. Hence, we could in fact strengthen our hypotheses to the small-volume $\p$-IC on $\mu_+$ alone and the small-volume $\widetilde\p$-IC on $\mu_-$ alone (both in $\R^N$ with constant $1$) and could then derive the conclusion \eqref{eq:lsc} by the alternative strategy just described. What seems more interesting to us, however, is that we could also proceed that way when keeping our ``signed ICs'' for $(\mu_-,\mu_+)$ and $(\mu_+,\mu_-)$ and additionally requiring $\mu_+\perp\mu_-$, because then the separate ICs will automatically follow by Proposition \ref{prop:reduction_to_SVIC} below. We will not work out the details of this general approach (since after all it requires the extra hypothesis $\mu_+\perp\mu_-$ and gives a less general result), and thus we will not make use of Proposition \ref{prop:reduction_to_SVIC} in the sequel. Nevertheless, we feel that at least the proposition and its proof may be of independent interest and are worth being briefly recorded here. In fact, the proof partially resembles the one of Lemma \ref{anis_lem:equiv_small_IC} and thus will be kept comparably concise.

	\begin{prop}[the small-volume $\p$-IC for $(\mu,\nu)$ with $\mu\perp\nu$ implies the small-volume $\p$-IC for $\mu$ alone]\label{prop:reduction_to_SVIC} 
      We impose Assumption \ref{assum:phi} and consider finite non-negative Radon measures $\mu$ and $\nu$ on open $U \subseteq \R^N$ such that $\mu \perp \nu$. If $(\mu,\nu)$ satisfies the small-volume $\p$-IC in $U$ with constant $C\in{(0,\infty)}$, then also $\mu$ alone satisfies the small-volume $\p$-IC in $U$ with constant $C$, and indeed even the extended measure $\ol{\mu}$ of Lemma \ref{anis_lem:equiv_small_IC} satisfies the small-volume $\p$-IC in $\R^N$ with constant $C$.
	\end{prop}
 
	\begin{proof}
        We assume the small-volume $\p$-IC for $(\mu,\nu)$ in $U$ and establish the following claim which contains the first part of the conclusion: For every $\eps>0$, there exist an open set $O\Subset U$ and some $\delta>0$ such that
		\begin{equation}\label{eq:tilde_anis_O_mu}
			\mu(A^+)
			\leq C \P_\p(A,O)+\eps
			\qq\text{for all measurable } A \Subset U\text{ such that } |A| < \delta\,.
		\end{equation}
        For verifying this claim, we first deduce from the assumption $\mu\perp\nu$ the existence of Borel sets $S_\mu,S_\nu\subseteq U$ such that
		$U=S_\mu \cupdot S_\nu$ and $\mu(S_\nu)=\nu(S_\mu)=0$ hold. Now we fix an arbitrary $\eps>0$. Since $\mu$ and $\nu$ are concentrated on $S_\mu$ and $S_\nu$, respectively, there exist compact sets $K_\mu \subseteq S_\mu$ and $K_\nu \subseteq S_\nu$ such that $\mu(U \setminus K_\mu) \leq \eps/6$ and $\nu(U \setminus K_\nu) \leq \eps/6$. In view of $\dist(K_\mu,K_\nu)>0$, we may further choose
        the open set $O$ such that we have $K_\mu \subseteq O \Subset U\setminus K_\nu$. We now let $d(x)\coleq\dist(x,O^\c)$ for $x \in\R^N$. Observing $\lim_{t \searrow 0} \mu\left(U\cap\left\{ d < t \right\} \right) = \mu(U\setminus O) \leq \mu(U\setminus K_\mu) \leq \eps/6$, we may choose $t_0>0$ small enough for having
		\[
		   \mu\big(U\cap\big\{ d < t_0 \big\} \big) < \frac\eps3\,.
		\]
        We further proceed in close analogy with the proof of Lemma \ref{anis_lem:equiv_small_IC}. We choose $\delta'>0$ such that the assumed small-volume IC for $(\mu,\nu)$ holds with $\eps$ and $\delta$ replaced by $\eps/3$ and $\delta'$. Then we fix a measurable $A\Subset U$ with $|A|<\delta\coleq\min{\big\{ \delta',\frac{t_0 \eps}{6\beta C}\big\}}$, and w.\@l.\@o.\@g.\@ we assume $\P(A,O)<\infty$. Arguing in the same way as for \eqref{eq:anis_III_primo}, we find some $\ol{t} \in (0,t_0)$ such that
		\[
			\H^{N-1} \big( A^+ \cap \big\{ d = \ol{t} \big\} \big) 
			< \frac{\eps}{6\beta C}.
		\]
		Setting $E\coleq A \cap  \left\{ d > \ol{t} \right\} \Subset O$, we then follow the derivation of \eqref{eq:mu_sv_est} to obtain
		\[
			\mu(A^+)
			\leq C \P_\p(E) + \nu(E^1) + \frac{2\eps}{3}
		\]
        (where we have brought the $\nu$-term to the right-hand side and have now explicitly recorded that we need to evaluate $\nu$ only on $E^1$ rather than $A^1$). In analogy with \eqref{eq:per_est}, we also get
		\[
			\P_\p(E)
			\leq \P_\p(A,O) + \frac\eps{6C}\,.
		\]
        Finally, we combine the preceding estimates and exploit $E\Subset O\Subset U\setminus K_\nu$ in order to additionally control $\nu(E^1) \leq \nu(U \setminus K_\nu) \leq \eps/6$. This results in
		\[
		   \mu(A^+)
		   \leq C \left( \P_\p(A,O) + \frac\eps{6C} \right) + \frac\eps6 +\frac{2\eps}3
		   = C \P_\p(A,O) +\eps\,.
		\]
		Hence, the claim \eqref{eq:tilde_anis_O_mu} and the asserted small-volume $\p$-IC for $\mu$ in $U$ are proved.

        Finally, the asserted small-volume $\p$-IC for $\ol{\mu}$ in $\R^N$ can now be read off from Lemma \ref{anis_lem:equiv_small_IC} (or can alternatively be derived directly by adapting the previous reasoning for estimating $\ol{\mu}(A^+)\le C\P_\p(A,O)$ even for all measurable $A\Subset\R^N$ such that $|A|<\delta$).
	\end{proof}

	\section{Existence theory for non-parametric functionals}
      \label{sec:exist}

    We recall that we work, for arbitrary $N\in\N$, over a bounded open set $\Omega\subseteq\R^N$ with Lipschitz boundary.
    
	\subsection{Lower semicontinuity of non-parametric functionals}
      \label{subsec:npar-lsc}

    Here we aim at proving the lower semicontinuity result of Theorem \ref{thm:lsc} for the functional $\widehat\Phi$ from \eqref{eq:P-functional}. In this regard it will be convenient to introduce two closely related auxiliary functionals:
	
	\begin{defi}\label{defi:anis_G_tilde}
		We impose Assumptions \ref{assum:phi} and \ref{assum:mu}. Then we introduce a full-space extension $\WPhione$ of our non-parametric functional\/ $\widehat\Phi$ and a variant\/ $\WPhitwo$ of the full-space extension with mirrored integrand $\widetilde\p$ and roles of\/ $\mu_+$ and\/ $\mu_-$ switched by setting, for $w \in \BV(\R^N)$,
		\[
			\WPhione[w]
			\coleq|\D w|_{\p}(\R^N)
              + \int_{\Omega} w^- \, \d\mu_+
              -\int_{\Omega}w^+ \, \d\mu_-
            \quad\text{and}\quad
			\WPhitwo[w]
			\coleq|\D w|_{\widetilde\p}(\R^N)
              + \int_{\Omega} w^- \, \d\mu_-
              - \int_{\Omega}w^+ \ \d\mu_+ \,.
		\]
	\end{defi}

    In fact, we shall consider the functionals $\WPhione$ and $\WPhitwo$ only on those $w \in \BV(\R^N)$ which are non-negative and coincide outside $\ol{\Omega}$ with a fixed boundary datum $u_0 \in \W^{1,1}(\R^N)$. Before turning to semicontinuity itself, we put on record a basic observation on the boundary term and an auxiliary estimate for truncations.

	\begin{rem} \label{estimate_phi_variation}
		Whenever $\p$ satisfies Assumption \ref{assum:phi}, the reverse triangle inequality \eqref{anis_eq:subadd_phi} for $\xi$ and $\tau$ linearly dependent gives, for any $w \in \BV(\R^N)$, the estimate
		\begin{align*}		    
          |\D w|_{\p}(\partial \Omega)
		   &=\int_{\partial \Omega} \p\left( \,\cdot\,, \left( \inn w{-}\ext w \right) \nu_\Omega \right)\d\H^{N-1} \\
		   &\geq \int_{\partial \Omega}\p\left( \,\cdot\,,\inn w \, \nu_\Omega \right)\d\H^{N-1} 
		   - \int_{\partial \Omega} \p\left( \,\cdot\,,  \ext w \, \nu_\Omega   \right)\d\H^{N-1}\,.
		\end{align*}
	\end{rem}

    \begin{lem}\label{lem:WPhione-cut-off}
        We impose Assumptions \ref{assum:phi}\eqref{case:cont} and \ref{assum:mu}. Let $(\mu_-,\mu_+)$ satisfy the $\p$-IC in $\Omega$ with constant $1$ and $(\mu_+,\mu_-)$ the $\widetilde\p$-IC in $\Omega$ with constant $1$, and fix a non-negative $u_0\in\W^{1,1}(\R^N)$. Then, for all non-negative $w\in\BV(\R^N)$ such that $w = u_0$ \ae{} in $\R^N \setminus \ol{\Omega}$ and all\/ $M>0$, we have
		\begin{equation}\label{eq:cut-off-estimate-with-cM}
          \WPhione[w] \geq \WPhione[w^M] + c_M
          \qq\text{with }c_M\coleq \alpha\int_{\R^N\setminus\ol{\Omega}}|\nabla(u_0-(u_0)^M)|\dx - \beta\int_{\partial \Omega}|\mathrm{T}_{\partial\Omega}(u_0{-}(u_0)^M)|\,\d\H^{N-1}\,,
        \end{equation}
        where $u^M$ denote the truncations of Lemma \ref{lem:cut_M}.
        Moreover, we record\/ $\lim_{M\to\infty}c_M=0$.
    \end{lem}

    \begin{proof}
		For $w$ as in the statement, the equivalent $\p$-IC \eqref{2A}, applied to $w - w^M \geq 0$ with $C = 1$, yields
		\[
			-  \int_{\Omega} (w - w^M)^+ \ \d\mu_-  + \int_{\Omega} (w - w^M)^- \ \d\mu_+ 
			\geq - |\D \left( w-w ^M\right)|_{\p}(\Omega) - \int_{\partial \Omega} \p\left( \,\cdot\,, \inn(w-w^M) \nu_\Omega \right)\,\d\H^{N-1}\,.
		\]
		Employing the results of Lemma \ref{lem:cut_M}\eqref{item:cut_M_ii}, Remark \ref{rem:additivity_upper_lower_limit}, and
		Remark \ref{estimate_phi_variation} for the difference $w-w^M$, we conclude
		\begin{align*}
			\WPhione[w] - \WPhione[w^M] 
			&=|\D(w-w ^M)|_\p(\R^N)
			+ \int_{\Omega} (w-w^M)^- \,\d\mu_+
			- \int_{\Omega} (w-w^M)^+ \,\d\mu_- \\
			& \geq   |\D (w-w^M) |_\p(\R^N \setminus \Omega) - \int_{\partial \Omega} \p\left( \,\cdot\,, \inn(w-w^M) \, \nu_\Omega \right)  \,\d\H^{N-1} \\
			& \geq |\D( w-w^M) |_\p(\R^N \setminus \ol{\Omega}) - \int_{\partial \Omega} \p\left( \,\cdot\,, \ext(w-w^M) \, \nu_\Omega \right)  \,\d\H^{N-1} \\ 
			&  = |\D (u_0-(u_0)^M) |_{\p}\left( \R^N \setminus \ol{\Omega}\right)   - \int_{\partial \Omega}  \p(\,\cdot\,, \mathrm{T}_{\partial\Omega}(u_0-(u_0)^M)\,\nu_\Omega )\,\d\H^{N-1}\\
			&\geq  \alpha\int_{\R^N\setminus\ol{\Omega}}|\nabla(u_0-(u_0)^M)|\dx - \beta\int_{\partial \Omega}|\mathrm{T}_{\partial\Omega}(u_0-(u_0)^M)|\,\d\H^{N-1}\,.
		\end{align*}
		This confirms \eqref{eq:cut-off-estimate-with-cM} with the constants $c_M$ defined there. Moreover, the convergence of $(u_0)^M$ to $u_0$ in $\W^{1,1}(\R^N)$ and Theorem \ref{thm:trace_cont} imply the claim $\lim_{M\to\infty}c_M=0$.
    \end{proof}

    The next lemma is the main technical achievement of this section and carries over semicontinuity from the parametric functionals of Section \ref{sec:par-lsc} at first to the non-parametric functional $\WPhione$.
	
	\begin{lem}[semicontinuity of $\WPhione$ on non-negative functions with prescribed values outside $\ol{\Omega}$]		\label{case3_lem:anis_G_tilde_LSC_nonneg}
		We impose Assumptions \ref{assum:phi}\eqref{case:conv-lsc},\eqref{case:cont} and \ref{assum:mu}. Let $(\mu_-,\mu_+)$ satisfy the $\p$-IC in $\Omega$ with constant $1$
		and $(\mu_+,\mu_-)$ the $\widetilde\p$-IC in $\Omega$ with constant $1$, and fix a non-negative $u_0 \in \W^{1,1}(\R^N)$. Then, for $(u_k)_k$ and $u$ in $\BV(\R^N)$ such that $u_k \to u$ in $\L^1(\R^N)$ with $u_k, u \geq 0$ \ae{} in $\Omega$ and $u_k = u = u_0$ \ae{} in $\R^N \setminus \ol{\Omega}$ for all $k$, we have
		\[
          \liminf_{k \to \infty} \WPhione[u_k] \geq \WPhione[u]\,.
        \]
    \end{lem}

    \begin{proof}			
		Consider a sequence $(u_k)_k$ and $u$ as in the statement, where by a standard reasoning with subsequences we may additionally assume $u_k\to u$ \ae{} in $\R^N$. Once more, we denote by $\ol{\mu_\pm}$ the extended measures given by $\overline{\mu_\pm}(S)\coleq\mu_\pm(S\cap\Omega)$ for Borel sets $S\subseteq\R^N$, and we record that by Remark \ref{rem:equivalence_sets_wrt_mu_ext} we have
		\[
		   \left\{ (u_k^M) ^+ > t \right\} 
		   = \left\{ u_k^M > t \right\}^+ 
		   \ \text{ and } \  \left\{ (u_k^M) ^- > t \right\} 
		   = \left\{ u_k^M > t \right\}^1
		   \qq \ol{\mu_\pm} \text{-\ae{} in } \R^N \mbox{ for }\mathcal{L}^1 \mbox{-\ae{} } t > 0\,. 
		\]
        We now employ Lemma \ref{lem:WPhione-cut-off}, the anisotropic coarea formula of Theorem \ref{anis_thm:coarea}, and
		a layer-cake argument for $u_k^M \geq 0$ (which exploits the previous equalities) to deduce
		\begin{align*}
			\WPhione[u_k] 
			\geq \WPhione[u_k^M] + c_M
			&= c_M + |\D u_k^M|_{\p}(\R^N) + \int_{\R^N} \left( u_k^M\right) ^- \, \d\ol{\mu_+} - \int_{\R^N} \left( u_k^M\right) ^+ \, \d\ol{\mu_-} \\
			&=  c_M + \int_{0}^{\infty} \left[ \P_{\p} \left( \left\{ u_k^M > t \right\} \right) + \ol{\mu_+} \left( \left\{ u_k^M > t \right\}^1 \right) - \ol{\mu_-}\left( \left\{ u_k^M > t \right\}^+ \right) \right] \dt\,. 
		\end{align*}
		Since the bound $u_k^M\leq M$ implies $\ol{\mu_-}\big( \left\{ u_k^M > t \right\}^+ \big)\le\1_{(0,M)}(t)\mu_-(\Omega)$ with $\1_{(0,M)}\mu_-(\Omega)\in\L^1(\R^+)$, the integrands in the last integral are bounded from below by an $\L^1$ function independent of $k$. Therefore, we may apply Fatou's lemma to deduce
		\[ 
			\liminf_{k \to \infty} \WPhione[u_k] 
			\geq   c_M + \int_{0}^{\infty} \liminf_{k \to \infty}  
			\left[   \P_{\p} \left( \left\{ u_k^M > t \right\} \right) + \ol{\mu_+} \left( \left\{ u_k^M > t \right\}^1 \right) - \ol{\mu_-}\left( \left\{ u_k^M > t \right\}^+ \right) \right] \dt\,.
		\]
		Since $u_k^M \to u^M$ \ae{} in $\R^N$ for $k \to \infty$ with $u_k^M=u^M$ outside $\ol{\Omega}$, the sets $A_{k,M}^t\coleq\left\{ u_k^M > t \right\}$
		converge in measure to $A_M^t\coleq\left\{ u^M > t \right\}$ for $k \to \infty$ whenever $t\ge0$ satisfies $\left|\{u^M=t\}\right|=0$. In particular, this convergence of sets is valid for $\mathcal{L}^1$-\ae{} $t\geq0$, and thus by the parametric semicontinuity result of Theorem \ref{thm:anis_semicont_funct_per_A+A1_gen},
        applied in combination with Lemma \ref{anis_lem:equiv_small_IC} as described in Remark \ref{rem:equiv_small_strong_IC}, we infer
		\[
		   \liminf_{k \to \infty} \WPhione[u_k]
		   \geq  c_M + \int_{0}^{\infty} \left[  \P_{\p}\left(\left\{ u^M > t \right\}\right) + \ol{\mu_+} \left( \left\{ u^M > t \right\}^1 \right) - \ol{\mu_-} \left( \left\{ u^M > t \right\}^+\right) \right] \dt\,.
        \]
        We can identify the right-hand side of the last inequality as simply $\WPhione[u^M]+c_M$ by the same rewriting as before via Remark \ref{rem:equivalence_sets_wrt_mu_ext}, the coarea formula of Theorem \ref{anis_thm:coarea}, and the layer-cake formula. Therefore, we arrive at
		\[
		   \liminf_{k \to \infty} \WPhione[u_k]
		   \geq \WPhione\big[u^M\big] + c_M \,.
		\]  
		In order to finally send $M\to\infty$ we exploit on one hand that 
            $|\D u^M|_\p(\R^N)\to|\D u|_\p(\R^N)$ by Lemma \ref{lem:cut_M}\eqref{item:cut_M_iv}
        and on the other hand that $\int_\Omega \big(u^M\big)^\pm\,\d\mu_\mp\to\int_\Omega u^\pm\,\d\mu_\mp$ by Lemma \ref{lem:int-uM-to-int-u}. Thus, also recalling $\lim_{M \to \infty} c_M = 0$ from Lemma \ref{lem:WPhione-cut-off}, we may pass to the limit with all terms of $\WPhione$ to conclude
		\[
		   \liminf_{k \to \infty}\WPhione[u_k]
		   \geq \WPhione[u]\,.
		\]
        This completes the proof.  
	\end{proof}

    Likewise we have semicontinuity also for the functional $\WPhitwo$:

 	\begin{lem}[semicontinuity of $\WPhitwo$ on non-negative functions with prescribed values outside $\overline\Omega$]
    \label{case3_lem:anis_G_tilde_LSC_nonneg2}
        We impose the same assumptions as in Lemma \ref{case3_lem:anis_G_tilde_LSC_nonneg}. Then, for $(u_k)_k$ and $u$ in $\BV(\R^N)$ such that $u_k \to u$ in $\L^1(\R^N)$ with $u_k, u \geq 0$ \ae{} in $\Omega$ and $u_k = u = u_0$ \ae{} in $\R^N \setminus \ol{\Omega}$ for all $k$, we also have
		\[
          \liminf_{k \to \infty} \WPhitwo[u_k] \geq \WPhitwo[u]\,.
        \]
    \end{lem}

    \begin{proof}
        We record that the assumptions of Lemma \ref{case3_lem:anis_G_tilde_LSC_nonneg} imply the analogous assumptions in which $\widetilde\p$ takes the role of $\p$ and at the same time the roles of $\mu_+$ and $\mu_-$ are switched. In particular, all relevant properties directly carry over from $\p$ to $\widetilde\p$, while the roles of $\mu_+$ and $\mu_-$ and the two assumed ICs are suitably symmetric and are just exchanged. With these findings in mind, we may apply Lemma \ref{case3_lem:anis_G_tilde_LSC_nonneg} with $(\widetilde\p,\mu_-,\mu_+)$ in place of $(\p,\mu_+,\mu_-)$ and directly arrive at the claim.
    \end{proof}

    In order to finally reach the main semicontinuity result of Theorem \ref{thm:exist} for the functional $\widehat\Phi$, we split $\widehat\Phi$ into $\WPhione$- and $\WPhitwo$-terms as follows.

    \begin{lem}\label{lem:decomp-Phi}
        We consider $u_0\in\W^{1,1}(\R^N)$ and impose Assumptions \ref{assum:phi} and \ref{assum:mu}. Then, for every $w\in\BV(\Omega)$, we have
        \[
            \widehat\Phi[w]
            =\WPhione\,[\ol{w}_+\big]
            +\WPhitwo\big[\ol{w}_-\big]
            -|\D u_0|_\p\big(\R^N\setminus\ol{\Omega}\big)\,,
        \]
        where the extension $\ol{w}=w\1_\Omega+u_0\1_{\R^N\setminus\ol{\Omega}}\in\BV(\R^N)$ incorporates the values of $u_0$ outside $\ol{\Omega}$.
    \end{lem}

    \begin{proof}
        From Lemma \ref{anis_lem:add_parts}, we have
		\[
          |\D\ol{w}|_{\p}\big(\ol{\Omega}\big)
          =|\D\ol{w}|_{\p}\big(\R^N\big)-|\D u_0|_{\p}\big(\R^N\setminus\ol{\Omega}\big)
		   =|\D\ol{w}_+|_{\p}\big(\R^N\big)+|\D\ol{w}_-|_{\widetilde{\p}}\big(\R^N\big)-|\D u_0|_{\p}\big(\R^N\setminus\ol{\Omega}\big)\,.
        \]
        Moreover, the $\H^{N-1}$-\ae{} decompositions $w^+=(w_+)^+-(w_-)^-$ and $w^-=(w_+)^--(w_-)^+$ of Lemma \ref{lem:dec_upper_limit} together with \eqref{eq:negligible} and $w_\pm=\ol{w}_\pm$ on the open set $\Omega$ directly yield
		\[
            \int_{\Omega}w^-\,\d\mu_+
		      -\int_{\Omega}w^+\,\d\mu_-
            =\int_{\Omega}(\ol{w}_+)^-\,\d\mu_+
			-\int_{\Omega}(\ol{w}_+)^+\,\d\mu_-
            +\int_{\Omega}(\ol{w}_-)^-\,\d\mu_-
            -\int_{\Omega}(\ol{w}_-)^+\,\d\mu_+\,.
		\]
        By definition of $\widehat\Phi$, $\WPhione$, $\WPhitwo$ the last two displayed equations combine to the claim of the lemma. 
    \end{proof}

    With the preceding lemmas at hand, the final semicontinuity conclusion is then quite quick:
    
	\begin{proof}[Proof of Theorem \ref{thm:lsc}]
        If we have $\L^1(\Omega)$-convergence of a sequence $(u_k)_k$ in $\BV(\Omega)$ to $u\in\BV(\Omega)$, then $(\ol{u_k})_\pm$ converge to $\ol{u}_\pm$ in the same sense. In view of $(\ol{u_k})_\pm,\ol{u}_\pm\ge0$ and $(\ol{u_k})_\pm=\ol{u}_\pm=(u_0)_\pm$ on $\R^N\setminus\ol{\Omega}$, Lemmas \ref{case3_lem:anis_G_tilde_LSC_nonneg} and \ref{case3_lem:anis_G_tilde_LSC_nonneg2} yield
        \[        
          \liminf_{k\to\infty}\WPhione\big[(\ol{u_k})_+\big]\ge\WPhione\big[\ol{u}_+\big]
          \qq\qq\text{and}\qq\qq
          \liminf_{k\to\infty}\WPhitwo\big[(\ol{u_k})_-\big]\ge\WPhitwo\big[\ol{u}_-\big]\,.
        \]
        Taking into account Lemma \ref{lem:decomp-Phi}, we can add up to arrive at
        \[
          \liminf_{k\to\infty}\widehat\Phi[u_k]\ge\widehat\Phi[u]\,.
        \]
        This is the claimed lower semicontinuity of $\widehat\Phi$.
	\end{proof}

	\subsection{ICs are necessary and sufficient for coercivity}
      \label{subsec:coercivity}

    Next we observe that essentially the same ICs exploited for semicontinuity are also necessary and sufficient for both coercivity and existence of minimizers. Our statements in this direction are rather straightforward extensions of those made for the function case $\mu_\pm=H_\pm\LN$ already in \cite{Giaquinta74a,Giaquinta74b}.

    We start by addressing the necessity of the ICs with constant $1$ even for boundedness from below of $\widehat\Phi$ on $\BV(\Omega)$. Clearly, this entails the necessity for both coercivity of $\widehat\Phi$ on $\BV(\Omega)$ and the existence of any minimizer of $\widehat\Phi$ in $\BV(\Omega)$.
    
	\begin{prop}[necessity of ICs with $C=1$] \label{prop:anis_CN_coercivity}
		We consider $u_0\in\W^{1,1}(\R^N)$ and impose Assumptions \ref{assum:phi} and \ref{assum:mu}. If\/ $\widehat\Phi$ is bounded from below on $\BV(\Omega)$, then $(\mu_-,\mu_+)$ satisfies the $\p$-IC in $\Omega$ with constant\/ $1$, and $(\mu_+,\mu_-)$ satisfies the $\widetilde\p$-IC in $\Omega$ with constant\/ $1$.
    \end{prop}
  
	\begin{proof}
        It suffices to establish the ICs in form of
		\begin{equation}\label{eq:nec-claim}
          -\P_{\widetilde\p}(A)
		   \leq  \mu_-(A^1)-\mu_+(A^+) 
		   \leq \mu_-(A^+)-\mu_+(A^1) 
		   \leq \P_{\p}(A) \qq\text{for all measurable }A\Subset \Omega\,.
        \end{equation}
        Indeed, suppose that the right-hand inequality in \eqref{eq:nec-claim} fails for some measurable $A \Subset \Omega$, which then satisfies
		\[
			\mu_- (A^+) - \mu_+(A^1) > \P_\p(A)\,.
		\]
		The functions $u_k \coleq k \1_{A} \in \BV(\Omega)$, $k \in \N$, are compactly supported in $\Omega$, and we record
        $|\D u_k|_{\p}(\Omega)=k\P_{\p}(A)$ and 
        $\int_\Omega(u_k)^+\,\d\mu_--\int_\Omega(u_k)^-\,\d\mu_+=k\mu_-(A^+)-k\mu_+(A^1)$.
        Consequently, we find
		\[
			\widehat\Phi[u_k]
			= k \left(  \P_{\p}(A) - \mu_-(A^+)  + \mu_+(A^1)  \right) +  \int_{\partial \Omega}\p(\,\cdot\,, -u_0 \nu_\Omega)\, \d\H^{N-1} \xrightarrow[k \to \infty]{} {-}\infty\,.
		\]
		This means that $\widehat\Phi$ is unbounded from below and contradicts the opposing hypothesis of the proposition. Thus, the right-hand inequality in \eqref{eq:nec-claim} is verified.

        If the left-hand inequality in \eqref{eq:nec-claim} fails for some measurable $A\Subset\Omega$, an entirely analogous reasoning with the functions $u_k\coleq-k\1_A$, which satisfy $|\D u_k|_{\p}(\Omega)=k\P_{\widetilde\p}(A)$, leads to a contradiction. Thus, the left-hand inequality in \eqref{eq:nec-claim} holds as well (and the middle inequality in \eqref{eq:nec-claim} is trivially valid anyway).
	\end{proof}

    The following statement formalizes that our ICs with constant $C<1$ are also sufficient for coercivity of $\widehat\Phi$. Clearly, the coercivity will eventually be combined with the semicontinuity of Theorem \ref{thm:lsc} in order to establish existence of minimizers.

    \begin{prop}[sufficiency of ICs with $C<1$ for coercivity] \label{case3_prop:anis_CS_coercivity} 
        We impose Assumptions \ref{assum:phi}\eqref{case:conv-lsc},\eqref{case:cont} and \ref{assum:mu} and consider $u_0\in\W^{1,1}(\R^N)$. If $(\mu_-,\mu_+)$ satisfies the $\p$-IC in $\Omega$ with constant $C<1$ and $(\mu_+,\mu_-)$ satisfies the $\widetilde\p$-IC in $\Omega$ with the same constant $C<1$, then $\widehat\Phi$ is coercive on $\BV(\Omega)$ \textup{(}in the sense of\/ $\widehat\Phi[w]\ge\nu\|w\|_{\BV(\Omega)}-L$ for all $w\in\BV(\Omega)$ with constants $\nu>0$ and $L\in\R$\textup{)}.
    \end{prop}
 
	\begin{proof} 
		We here exploit the ICs in the convenient form of \eqref{2A_global} and make use of \eqref{eq:comp-p0} and \eqref{anis_eq:subadd_phi}. In this way, for arbitrary $w \in \BV(\Omega)$,  we estimate
		\begin{align*} 
			\widehat\Phi[w]
			&=|\D w|_{\p}(\Omega)+\int_{\partial \Omega}\p(\,\cdot\,,(w{-}u_0)\nu_\Omega)\, \d\H^{N-1} - \int_{\Omega}w^+\,\d\mu_- + \int_{\Omega}w^-\,\d\mu_+ \\
			& \geq  |\D w|_{\p}(\Omega) + \int_{\partial \Omega} \big[ \p(\,\cdot\,, w \nu_\Omega) - \p(\,\cdot\,, u_0 \nu_\Omega) \big] \,\d\H^{N-1} 
			-C \left( |\D w|_{\p}(\Omega) + \int_{\partial \Omega}\p(\,\cdot\,, w \nu_\Omega)\,\d\H^{N-1}  \right) \\
			& = (1-C) \left( |\D w|_{\p}(\Omega) + \int_{\partial \Omega}\p(\,\cdot\,, w \nu_\Omega)\,\d\H^{N-1}  \right) 
			- \int_{\partial \Omega} \p(\,\cdot\,, u_0 \nu_\Omega)\,\d\H^{N-1}\\
			& \geq \alpha (1-C) \left( |\D w|(\Omega) + \int_{\partial \Omega}|w|\,\d\H^{N-1}  \right) 
			- \beta \int_{\partial \Omega} |u_0| \,\d\H^{N-1}\,.
		\end{align*}
        The claim follows by estimating the first term in the last line
        from below via the Poincar\'e inequality \eqref{eq:Poinc_BV3}.
	\end{proof}

	\subsection{Existence of minimizers} \label{subsec:exist}

    At this stage, we provide a proof of Theorem \ref{thm:exist}. In fact, with Theorem \ref{thm:lsc} and Proposition \ref{case3_prop:anis_CS_coercivity} at hand, existence in the non-extreme cases is a routine consequence of the direct method:
 
	\begin{proof}[Proof of Theorem \ref{thm:exist} in case $C<1$]
		By Proposition \ref{case3_prop:anis_CS_coercivity}, the functional $\widehat\Phi$ is coercive on $\BV(\Omega)$, and thus every minimizing sequence $(u_k)_k$ for $\widehat\Phi$ in $\BV(\Omega)$ has a subsequence which converges in $\L^1(\Omega)$ to some $u \in \BV(\Omega)$. By the lower semicontinuity of Theorem \ref{thm:lsc}, we then obtain
		\[
          \widehat{\Phi}[u]
          \le\lim_{k\to\infty}\widehat\Phi[u_k]
          =\inf_{w\in\BV(\Omega)}\widehat\Phi[w]\,.
        \]
		Thus, $u$ is a minimizer of $\widehat\Phi$, and the proof is complete.
    \end{proof}

    Finally, we turn to the more interesting extreme case of Theorem \ref{thm:exist}, which will be covered with the help of the following auxiliary lemma.

	\begin{lem} \label{lem:anis_cut_off_min_sign}
		We impose Assumptions \ref{assum:phi}\eqref{case:cont} and \ref{assum:mu} and consider $u_0\in\W^{1,1}(\R^N)\cap\L^\infty(\R^N)$. If $(\mu_-,\mu_+)$ satisfies the $\p$-IC in $\Omega$ and $(\mu_+,\mu_-)$ satisfies the $\widetilde\p$-IC in $\Omega$ with constant $1$, 
		then we have
		\[
			\widehat\Phi[w^M] \leq \widehat\Phi[w] \quad \text{ for all } w \in \BV(\Omega) \text{ and all } M \geq \|u_0\|_{\L^{\infty}(\R^N)}\,.
		\]
    \end{lem}

    \begin{proof}
      We first record $(\ol{w}_+)^M=\big(\ol{w}^M\big)_+=\big(\,\ol{w^M}\,\big)_+$, where the second equality is based on the bound $M \geq \|u_0\|_{\L^{\infty}(\R^N)}$. Moreover, this bound also means $(u_0)^M=u_0$. Therefore, Lemma \ref{lem:WPhione-cut-off} applies with $c_M=0$ and yields
      \[
        \WPhione[\ol{w}_+]\ge\WPhione\big[(\ol{w}_+)^M\big]
        =\WPhione\big[\big(\,\ol{w^M}\,\big)_+\big]\,.
      \]
      Applying this outcome with $({-}w,\widetilde\p,\mu_+,\mu_-)$ in place of $(w,\p,\mu_-,\mu_+)$, we also get
      \[
        \WPhitwo[\ol{w}_-]\ge\WPhitwo\big[\big(\,\ol{w^M}\,\big)_-\big]\,.
      \]
      Relying on Lemma \ref{lem:decomp-Phi} we may then add up the two inequalities obtained to reach the claim of the lemma.
    \end{proof}

    \begin{proof}[Proof of Theorem \ref{thm:exist} in case $C=1$]
        We work with $u_0\in\W^{1,1}(\R^N)\cap\L^\infty(\R^N)$ and a minimizing sequence $(u_k)_k$ for $\widehat\Phi$ in $\BV(\Omega)$. If we fix $M \coleq \|u_0\|_{\L^{\infty}(\R^N)}$, the estimate $\widehat\Phi\big[u_k^M\big] \leq \widehat\Phi[u_k]$ of Lemma \ref{lem:anis_cut_off_min_sign} applies and guarantees that $\big(u_k^M\big)_k$ is still a minimizing sequence for $\widehat\Phi$. Possibly passing to a subsequence, we may suppose
        \[
          \sup_{k\in\N}\widehat\Phi\big[u_k^M\big]<\infty\,.
        \]
		Moreover, since we have $\big\|u_k^M\big\|_{\L^\infty(\Omega)}\leq M$ and $\mu_\pm$ are finite measures, the terms
        $\int_\Omega\big(u_k^M\big)^\pm\,\d\mu_\mp$ are $k$-uniformly bounded, and thus we can deduce even
        \[
          \sup_{k\in\N}\TV_\p^{u_0}\big[u_k^M\big]<\infty\,.
        \]
        Employing \eqref{eq:comp-p0}, we estimate
        \[
          \TV_\p^{u_0}\big[u_k^M\big]
          \ge\alpha\TV^{u_0}\big[u_k^M\big]
          \ge\alpha\big|\D u_k^M\big|(\Omega)\,,
        \]
        and then we use the bound $\big\|u_k^M\big\|_{\L^\infty(\Omega)}\leq M$
        once more to infer the boundedness of $\big(u_k^M\big)_k$ in
        $\BV(\Omega)$. At this stage, using Theorem \ref{thm:lsc} in the same
        way as in the case $C<1$, we obtain a minimizer $u$ as limit of an
        $\L^1(\Omega)$-convergent subsequence of $\big(u_k^M\big)_k$. From
        $\big\|u_k^M\big\|_{\L^\infty(\Omega)}\leq M$ we deduce the claimed
        bound $\|u\|_{\L^\infty(\Omega)}\leq M$.
    \end{proof}

\section{Existence of recovery sequences}\label{sec:recovery}

We recall once again that $N\in\N$ is arbitrary and $\Omega\subseteq\R^N$ denotes a bounded open set with Lipschitz boundary and that, with Assumptions \ref{assum:phi} and \ref{assum:mu} in force and a given $u_0\in\W^{1,1}(\R^N)$, the functionals 
\[
  \Phi[w]\coleq\int_{\Omega}\p(\,\cdot\,,\nabla w)\dx
  +\int_\Omega w^\ast\,\d(\mu_+{-}\mu_-)
  \qq\text{for }w\in\W^{1,1}_{u_0}(\Omega)
\]
and
\[
  \widehat\Phi[w]
  \coleq|\D w|_\p(\Omega)
  +\int_{\partial \Omega}\p(\,\cdot\,,(w{-}u_0)\nu_\Omega)\,\d\H^{N-1}
  +\int_\Omega w^-\,\d\mu_+
  -\int_\Omega w^+\,\d\mu_-
  \qq\text{for }w\in\BV(\Omega)
\]
are well-defined. Moreover, we record that $\widehat\Phi$ is an extension of $\Phi$ in the sense of $\widehat\Phi[w]=\Phi[w]$ for $w\in\W^{1,1}_{u_0}(\Omega)$.

The aim of this section is proving Theorem \ref{thm:recovery}, which crucially relies on the additional assumption $\mu_+\perp\mu_-$ and then recovers the value $\widehat\Phi[u]$ even at an arbitrary $u\in\BV(\Omega)$ from the values of $\Phi$ on $\W^{1,1}_{u_0}(\Omega)$. In fact, a slightly weaker statement with recovery from the values of $\widehat\Phi$ on all of $\W^{1,1}(\Omega)$ is already at hand by Proposition \ref{prop:conv-wk} and --- as said earlier --- to some extent resembles \cite[Lemma 4.1]{LeoCom24}. Thus, we start by recasting Proposition \ref{prop:conv-wk} in only slightly modified form and recall for this purpose that $\ol{w}=w\1_\Omega+u_0\1_{\R^N\setminus\ol{\Omega}}\in\BV(\R^N)$ denotes the extension of $w\in\BV(\Omega)$ which uses the values of $u_0$ outside $\ol{\Omega}$.

\begin{prop}[recovery sequences with unconstrained boundary values]\label{prop:recovery-BV-W}
  We impose Assumptions \ref{assum:phi}\eqref{case:cont} and \ref{assum:mu} with additionally $\mu_+\perp\mu_-$. Then, for every $u\in\BV(\Omega)$, there exists a sequence $(w_k)_k$ in $\W^{1,1}(\Omega)$ such that $(w_k)_k$ converges to $u$ strictly in $\BV(\Omega)$, in particular also $(\ol{w_k})_k$ converges to $\ol{u}$ strictly in $\BV(\R^N)$, and such that we have
  \begin{gather}
    \lim_{k\to\infty}|\D w_k|_\p(\Omega)
    =|\D u|_\p(\Omega)\,,
    \label{eq:recovery-BV-W-p-strict}\\
    \lim_{k\to\infty}\int_{\partial\Omega}\p(\,\cdot\,,(w_k{-}u_0)\nu_\Omega)\,\d\H^{N-1}
    =\int_{\partial\Omega}\p(\,\cdot\,,(u{-}u_0)\nu_\Omega)\,\d\H^{N-1}
    \qq\text{for each }u_0\in\W^{1,1}(\R^N)\,,
    \label{eq:recovery-BV-W-traces}\\
    \lim_{k\to\infty}\int_\Omega w_k^\ast\,\d\mu_- 
    = \int_\Omega u^+ \,\d\mu_-\,,
    \qq\qq
    \lim_{k\to\infty}\int_\Omega w_k^\ast\,\d\mu_+
    = \int_\Omega u^- \,\d\mu_+\,.
    \label{eq:recovery-BV-W-mu}
  \end{gather}
  In particular, it holds
  \[
    \lim_{k\to\infty}\widehat\Phi[w_k]=\widehat\Phi[u]\,.
  \]
\end{prop}

\begin{proof}
  Proposition \ref{prop:conv-wk} provides a sequence $(w_k)_k$ in $\W^{1,1}(\Omega)$ which converges to $u$ strictly in $\BV(\Omega)$ and satisfies \eqref{eq:recovery-BV-W-mu}. Since the strict convergence implies $\p$-strict convergence by Theorem \ref{anis_thm:Reshetnyak_cont}, we also have \eqref{eq:recovery-BV-W-p-strict}. Moreover, since the strict convergence in $\BV(\Omega)$ induces convergence of the traces in $\L^1(\partial\Omega;\H^{N-1})$ by \cite[Theorem 3.88]{AFP00} and since the upper bound in \eqref{eq:comp-p0} is available, we deduce \eqref{eq:recovery-BV-W-traces}. In addition, with the convergence of traces at hand, also the strict convergence of $(\ol{w_k})_k$ to $\ol{u}$ comes out. The final convergence
  $\lim_{k\to\infty}\widehat\Phi[w_k]=\widehat\Phi[u]$ follows by writing out $\widehat\Phi$ and using
  \eqref{eq:recovery-BV-W-p-strict}, \eqref{eq:recovery-BV-W-traces}, \eqref{eq:recovery-BV-W-mu} for the single terms.
\end{proof}

In order to improve on the situation of Proposition \ref{prop:recovery-BV-W} and regain prescribed boundary values in the trace sense, we next aim at approximating an arbitrary $w\in\W^{1,1}(\Omega)$ by a sequence $(v_k)_k$ in $\W^{1,1}_{u_0}(\Omega)$ such that all relevant terms converge. In this regard, the technical handling of the $\mu_\pm$-terms seems a bit subtle, since even with $\mu_\pm$-\ae{} convergence $v_k^\ast\to w^\ast$ at hand in general cases with $|\D\ol{w}|(\partial\Omega)>0$ it does not seem straightforward to deduce the convergence of the $\mu_\pm$-integrals on $\Omega$ by e.\@g.\@ the application of a convergence theorem or an IC. However, if we arrange for a sequence $(v_k)_k$ bounded in $\L^\infty(\Omega)$, we easily get through with the dominated convergence theorem. The next lemma actually manages to reduce to this favorable situation at the cost of requiring an $\L^\infty$ bound for the boundary datum $u_0$ only.

\begin{lem}[from unconstrained boundary values to $\L^\infty$ boundary values]\label{lem:recovery-W-Wu0}
  We impose Assumptions \ref{assum:phi}\eqref{case:cont} and \ref{assum:mu}. Then, for every $u_0\in\W^{1,1}(\R^N)\cap\L^\infty(\R^N)$ and every $w\in\W^{1,1}(\Omega)$, there exists a sequence $(v_k)_k$ in $\W^{1,1}_{u_0}(\Omega)$ such that $(\ol{v_k})_k$ converges to $\ol{w}$ strictly in $\BV(\R^N)$ and such that we have
  \begin{gather}
    \lim_{k\to\infty}\int_\Omega\p(\,\cdot\,,\nabla v_k)\dx
    =\int_\Omega\p(\,\cdot\,,\nabla w)\dx
    +\int_{\partial\Omega}\p(\,\cdot\,,(w{-}u_0)\nu_\Omega)\,\d\H^{N-1}\,,
    \label{eq:recovery-W-Wu0-p-strict}\\
    \lim_{k\to\infty}\int_\Omega v_k^\ast\,\d\mu_\pm
    =\int_\Omega w^\ast \,\d\mu_\pm \,.
    \label{eq:recovery-W-Wu0-mu}
  \end{gather}
  In particular, there hold\/ $\widehat\Phi[v_k]=\Phi[v_k]$ and
  \[
    \lim_{k\to\infty}\Phi[v_k]=\widehat\Phi[w]\,.
  \]
\end{lem}

\begin{proof}
  We work with the truncations $u^M \in \W^{1,1}(\Omega) \cap \L^{\infty}(\Omega)$ of Lemma \ref{lem:cut_M}.

  \textit{Step 1.} In this step we keep $M\geq\|u_0\|_{\L^\infty(\R^N)}$ fixed. Since this ensures $\ol{u^M}=\ol{u}^M$ for all $u\in\BV(\Omega)$, we stick to writing $\ol{u}^M$ for such terms in the sequel. We apply Theorem \ref{strict_int_appr_BV} to $w^M$ to obtain a sequence $(u_\ell)_\ell$ of approximations in $\W^{1,1}_{u_0}(\Omega)$ such that $(\overline{u_\ell})_\ell$ converges to $\ol{w}^M$ strictly in $\BV(\R^N)$. If we pass to $(u_\ell)^M$ with additional truncation at level $M$, in view of $\|u_0\|_{\L^\infty(\R^N)}\leq M$ we still have $(u_\ell)^M\in\W^{1,1}_{u_0}(\Omega)$. Moreover, since the limit satisfies $\|\ol{w}^M\|_{\L^\infty(\R^N)}\leq M$, it is a standard matter to check that also the convergence is preserved and thus $\big(\overline{u_\ell}^M\big)_\ell$ still converges to $\ol{w}^M$ strictly in $\BV(\R^N)$. By Theorem \ref{anis_thm:Reshetnyak_cont} the latter convergence is also $\p$-strict, and in view of $(u_\ell)^M\in\W^{1,1}_{u_0}(\Omega)$ and $\overline{u_\ell}^M=u_0=\ol{w}^M$ on $\R^N\setminus\ol{\Omega}$ it may be spelled out as
  \begin{equation}\label{eq:recovery-W-Wu0-p-strict-1}
    \lim_{\ell\to\infty}\int_\Omega\p\big(\,\cdot\,,\nabla(u_\ell)^M\big)\dx
    =\int_\Omega\p\big(\,\cdot\,,\nabla w^M\big)\dx
    +\int_{\partial\Omega}\p\big(\,\cdot\,,\big(w^M{-}u_0\big)\nu_\Omega\big)\,\d\H^{N-1}\,.
  \end{equation}
  Moreover, Theorem \ref{strict_Hausdorff_repr} enforces $\H^{N-1}$-\ae{} convergence $\big((u_\ell)^M\big)^\ast\to\big(w^M\big)^\ast$ in $\Omega$ for $\ell\to\infty$, and by \eqref{eq:negligible} this convergence remains true $\mu_\pm$-\ae{} as well. Since $(u_\ell)^M$ are equi-bounded by $M$, the dominated convergence theorem yields
  \begin{equation}\label{eq:recovery-W-Wu0-mu-1}
    \lim_{\ell\to\infty}\int_\Omega\big((u_\ell)^M\big)^\ast\,\d\mu_\pm
    =\int_\Omega\big(w^M\big)^\ast\,\d\mu_\pm\,.
  \end{equation}

  \textit{Step 2.} In this step we pass to the limit $M\to\infty$. By Lemma \ref{lem:cut_M}\eqref{item:cut_M_iii} and \eqref{item:cut_M_iv}, we have that $\ol{w}^M$ converge to $\ol{w}$ strictly in $\BV(\R^N)$ for $M\to\infty$ and also that
  \begin{equation}\label{eq:recovery-W-Wu0-p-strict-2}
    \lim_{M\to\infty}\int_\Omega\p\big(\,\cdot\,,\nabla w^M\big)\dx
    =\int_\Omega\p(\,\cdot\,,\nabla w)\dx
  \end{equation}
  holds. Via convergence of traces $w^M\to w$ in $\L^1(\partial\Omega;\H^{N-1})$ (which can be deduced either from strict continuity of the trace operator or elementarily) and via \eqref{eq:comp-p0}, we additionally arrive at
  \begin{equation}\label{eq:recovery-W-Wu0-traces-2}
    \lim_{M\to\infty}\int_{\partial\Omega}\p\big(\,\cdot\,,\big(w^M{-}u_0\big)\nu_\Omega\big)\,\d\H^{N-1}
    =\int_{\partial\Omega}\p(\,\cdot\,,(w{-}u_0)\nu_\Omega)\,\d\H^{N-1}\,.
  \end{equation}
  Moreover, from Lemma \ref{lem:int-uM-to-int-u} we have
  \begin{equation}\label{eq:recovery-W-Wu0-mu-2}
    \lim_{M\to\infty}\int_\Omega\big(w^M\big)^\ast\,\d\mu_\pm
    =\int_\Omega w^\ast\,\d\mu_\pm\,.
  \end{equation}

  Finally, we combine the convergences of Step 1 and Step 2: We fix any sequence $(M_k)_k$ such that $M_k\ge\|u_0\|_{\L^\infty(\R^N)}$ for all $k$ and $\lim_{k\to\infty}M_k=\infty$. Then, we set
  \[
    v_k\coleq(u_{\ell_k})^{M_k}\in\W^{1,1}_{u_0}(\Omega)\,,
  \]
  where for each $k$ we have chosen $\ell_k$ large enough for ensuring that $\ol{v_k}$ converge to $\ol{w}$ strictly in $\BV(\R^N)$, that on one hand \eqref{eq:recovery-W-Wu0-p-strict-1} and \eqref{eq:recovery-W-Wu0-p-strict-2}, \eqref{eq:recovery-W-Wu0-traces-2} induce \eqref{eq:recovery-W-Wu0-p-strict}, and that on the other hand \eqref{eq:recovery-W-Wu0-mu-1} and \eqref{eq:recovery-W-Wu0-mu-2} induce \eqref{eq:recovery-W-Wu0-mu}. The convergence $\lim_{k\to\infty}\Phi[v_k]=\widehat\Phi[w]$ follows by spelling out $\Phi$ and $\widehat\Phi$ and using \eqref{eq:recovery-W-Wu0-p-strict}, \eqref{eq:recovery-W-Wu0-mu} for the single terms.
\end{proof}

With Proposition \ref{prop:recovery-BV-W} and Lemma \ref{lem:recovery-W-Wu0} at hand, the proof of Theorem \ref{thm:recovery} is in reach. However, since Lemma \ref{lem:recovery-W-Wu0} requires boundedness of $u_0$, we still need to approximate a general boundary datum $u_0$ with bounded ones, and the next lemma will help in preserving all relevant properties of a recovery sequence in this process. Here, since we vary the boundary values, we switch to the more precise notation $\ol{w}^{u_0}$ instead of simply $\ol{w}$ for the extension $w\1_\Omega+u_0\1_{\R^N\setminus\ol{\Omega}}$ of $w\in\BV(\Omega)$ which uses the datum $u_0\in\W^{1,1}(\R^N)$.

\begin{lem}[varying the boundary values]\label{lem:approx-u0}
    We impose Assumption \ref{assum:phi}\eqref{case:conv-lsc} and suppose that $(u_{0,k})_k$ converges  to $u_0$ in $\W^{1,1}(\R^N)$. Then, for every sequence $(z_k)_k$ such that $z_k\in\W^{1,1}_{u_{0,k}}(\Omega)$ for all $k\in\N$, there exists a sequence $(u_k)_k$ in $\W^{1,1}_{u_0}(\Omega)$ such that $(u_k{-}z_k)_k$ converges to $0$ in $\W^{1,1}(\Omega)$, in particular also $(\ol{u_k}^{u_0}{-}\ol{z_k}^{u_{0,k}})_k$ converges to $0$ in $\W^{1,1}(\R^N)$, and such that we have
    \begin{equation}\label{eq:approx-u0}
      \lim_{k\to\infty}\bigg(\int_\Omega\p(\,\cdot\,,\nabla u_k)\dx-\int_\Omega\p(\,\cdot\,,\nabla z_k)\dx\bigg)=0
      \qq\text{and}\qq
      \lim_{k\to\infty}\bigg(\int_\Omega u_k^\ast\,\d\mu-\int_\Omega z_k^\ast\,\d\mu\bigg)=0
    \end{equation}
    whenever a non-negative Radon measure $\mu$ on $\Omega$ satisfies the isotropic IC in $\Omega$ with some constant $C\in{[0,\infty)}$. Specifically, if Assumption \ref{assum:mu} applies for non-negative Radon measures $\mu_+$ and $\mu_-$ on $\Omega$, then we have
    \[
      \lim_{k\to\infty}\big(\Phi[u_k]-\Phi[z_k]\big)=0\,.
    \]
\end{lem}

\begin{proof}
    We choose
    \[
      u_k\coleq z_k-u_{0,k}+u_0\in\W^{1,1}_{u_0}(\Omega)
    \]
    and observe
    $\|u_k{-}z_k\|_{\W^{1,1}(\Omega)}=\|u_{0,k}{-}u_0\|_{\W^{1,1}(\Omega)}$ and $\|\ol{u_k}^{u_0}{-}\ol{z_k}^{u_{0,k}}\|_{\W^{1,1}(\R^N)}=\|u_{0,k}{-}u_0\|_{\W^{1,1}(\R^N)}$. Thus, it is immediate that $(u_k{-}z_k)_k$ and $(\ol{u_k}^{u_0}{-}\ol{z_k}^{u_{0,k}})_k$ converge to $0$ in $\W^{1,1}(\Omega)$ and $\W^{1,1}(\R^N)$, respectively. The first convergence in \eqref{eq:approx-u0} can then be deduced with the help of \eqref{eq:comp-p0} and \eqref{anis_eq:subadd_phi}, since in fact these inequalities yield $\p(\,\cdot\,,\nabla u_k)-\p(\,\cdot\,,\nabla z_k)\le\p(\,\cdot\,,\nabla u_k{-}\nabla z_k)
    \le\beta|\nabla u_k{-}\nabla z_k|$ and an analogous lower estimate. For proving the second convergence in \eqref{eq:approx-u0}, we exploit the assumed isotropic IC for $\mu$ in the form of Remark \ref{rem:IC_global}. Indeed, by using \eqref{2A_global} from that remark with the isotropic integrand, with $(\mu,0)$ in place of $(\mu_-,\mu_+)$, and with $u_k{-}z_k\in\W^{1,1}(\Omega)$ in place of $v\in\BV(\Omega)$, we find
    \[
      \int_\Omega u_k^\ast\,\d\mu-\int_\Omega z_k^\ast\,\d\mu
      =\int_\Omega\big(u_k-z_k\big)^\ast\,\d\mu
      \le C\bigg(\int_\Omega|\nabla u_k-\nabla z_k|\dx
      +\int_{\partial\Omega}|u_{0,k}-u_0|\,\d\H^{N-1}\bigg)\,,
    \]
    where by the previous observations and the continuity of the trace operator the right-hand side vanishes in the limit $k\to\infty$. Since an application of \eqref{2A_global} with $z_k{-}u_k\in\W^{1,1}(\Omega)$ instead of $v\in\BV(\Omega)$ gives an analogous lower estimate, also the second convergence in \eqref{eq:approx-u0} follows.

    Finally, taking into account Proposition \ref{prop:admissible} we deduce from Assumption \ref{assum:mu} that an isotropic IC with a large constant $C$ holds for $\mu_+$ and $\mu_-$. Thus, under Assumption \ref{assum:mu} we have the convergences in \eqref{eq:approx-u0} and may combine them to infer $\lim_{k\to\infty}\big(\Phi[u_k]-\Phi[z_k]\big)=0$.
\end{proof}

With Proposition \ref{prop:recovery-BV-W} and Lemmas \ref{lem:recovery-W-Wu0}, \ref{lem:approx-u0} at hand, we are now ready to restate Theorem \ref{thm:recovery} in a slightly extended version and finally approach its proof. We emphasize that the decisive assumption $\mu_+\perp\mu_-$ enters only through the application of Proposition \ref{prop:recovery-BV-W} and can then be traced back to the proof of the earlier Proposition \ref{prop:conv-wk}.

\begin{thm}[existence of recovery sequences]\label{thm:rec_seq}
  We impose Assumptions \ref{assum:phi}\eqref{case:conv-lsc},\eqref{case:cont} and \ref{assum:mu} and additionally require $\mu_+\perp \mu_-$. Then, for every $u_0\in\W^{1,1}(\R^N)$ and every $u\in\BV(\Omega)$, there exists a recovery sequence $(u_k)_k$ in $\W^{1,1}_{u_0}(\Omega)$ such that $\ol{u_k}$ converges to $\ol{u}$ strictly in $\BV(\R^N)$ and such that we have
  \begin{gather}
    \lim_{k\to\infty}\int_\Omega\p(\,\cdot\,,\nabla u_k)\dx
    =|\D u|_\p(\Omega)+\int_{\partial\Omega}\p(\,\cdot\,,(u{-}u_0)\nu_\Omega)\,\d\H^{N-1}\,,
    \label{eq:recovery-BV-Wu0-p-strict}\\
    \lim_{k\to\infty}\int_\Omega u_k^\ast\,\d\mu_- 
    = \int_\Omega u^+\,\d\mu_-\,,
    \qq\qq
    \lim_{k\to\infty}\int_\Omega u_k^\ast\,\d\mu_+
    = \int_\Omega u^-\,\d\mu_+\,.
    \label{eq:recovery-BV-Wu0-mu}
  \end{gather}
  In particular, there hold\/ $\widehat\Phi[u_k]=\Phi[u_k]$ and
  \begin{equation} 
    \lim_{k\to\infty}\Phi[u_k]
    =\widehat\Phi[u]\,.
  \end{equation}
\end{thm}

\begin{proof}    
    \textit{Step 1.} We assume even $u_0\in\W^{1,1}(\R^N)\cap\L^\infty(\R^N)$. Then, we approximate $u\in\BV(\Omega)$ by a sequence $(w_k)_k$ in $\W^{1,1}(\Omega)$ as in Proposition \ref{prop:recovery-BV-W} and afterwards each $w_k\in\W^{1,1}(\Omega)$ by a sequence $(v_{k,\ell})_\ell$ in $\W^{1,1}_{u_0}(\Omega)$ as in Lemma \ref{lem:recovery-W-Wu0}. In summary, for $u_k\coleq v_{k,\ell_k}\in\W^{1,1}_{u_0}(\Omega)$ with $\ell_k$ chosen suitably large, we obtain that $(\ol{u_k})_k$ converges to $\ol{u}$ strictly in $\BV(\R^N)$ with exactly \eqref{eq:recovery-BV-Wu0-p-strict} and \eqref{eq:recovery-BV-Wu0-mu}.

    \textit{Step 2.} We consider a general $u_0\in\W^{1,1}(\R^N)$. Then we choose an arbitrary sequence $(u_{0,k})_k$ in $\W^{1,1}(\R^N)\cap\L^\infty(\R^N)$ such that $(u_{0,k})_k$ converges to $u_0$ in $\W^{1,1}(\R^N)$. For instance, we might again use truncations $u_{0,k}\coleq(u_0)^k$. In any case, for the sole relevant terms which involve the boundary datum, we exploit once more Theorem \ref{thm:trace_cont} and record
    \begin{gather*}
      \lim_{k\to\infty}\int_{\partial\Omega}|u{-}u_{0,k}|\,\d\H^{N-1}
      =\int_{\partial\Omega}|u{-}u_0|\,\d\H^{N-1}\,,\\
      \lim_{k\to\infty}\int_{\partial\Omega}\p(\,\cdot\,,(u{-}u_{0,k})\nu_\Omega)\,\d\H^{N-1}
      =\int_{\partial\Omega}\p(\,\cdot\,,(u{-}u_0)\nu_\Omega)\,\d\H^{N-1}\,.
    \end{gather*}
    Then, by applying the outcome of Step 1 for each fixed $u_{0,k}$ and once more by suitable choices of a second index, we obtain a sequence $(z_k)_k$ such that we have $z_k\in\W^{1,1}_{u_{0,k}}(\Omega)$ for all $k$, such that $(\ol{z_k}^{u_{0,k}})_k$ converges to $\ol{u}^{u_0}$ strictly in $\BV(\R^N)$ and such that \eqref{eq:recovery-BV-Wu0-p-strict} and \eqref{eq:recovery-BV-Wu0-mu} hold with $z_k$ in place of $u_k$. With this we have reached the point for applying Lemma \ref{lem:approx-u0}. The lemma in fact yields a sequence $(u_k)_k$ in $\W^{1,1}_{u_0}(\Omega)$ (now with the correct fixed boundary datum) such that $(\ol{u_k}^{u_0}{-}\ol{z_k}^{u_{0,k}})_k$ converges to $0$ in $\BV(\R^N)$ and \eqref{eq:approx-u0} applies with $\mu=\mu_\pm$. In conclusion, we get that $(\ol{u_k}^{u_0})_k$ converges to $\ol{u}^{u_0}$ strictly in $\BV(\R^N)$, and moreover \eqref{eq:approx-u0} perfectly fits for passing from \eqref{eq:recovery-BV-Wu0-p-strict}, \eqref{eq:recovery-BV-Wu0-mu} with $z_k$ to the claimed form of \eqref{eq:recovery-BV-Wu0-p-strict}, \eqref{eq:recovery-BV-Wu0-mu} with $u_k$. The final claim $\lim_{k\to\infty}\Phi[u_k]=\widehat\Phi[u]$ follows once more by writing out the functionals and employing the convergences of the single terms. The proof of Theorem \ref{thm:rec_seq} is complete.
\end{proof}
	
\bibliographystyle{amsplain}
\bibliography{TV_phi_mu}

\providecommand{\bysame}{\leavevmode\hbox to3em{\hrulefill}\thinspace}
\providecommand{\MR}{\relax\ifhmode\unskip\space\fi MR }
\providecommand{\MRhref}[2]{%
  \href{http://www.ams.org/mathscinet-getitem?mr=#1}{#2}
}
\providecommand{\href}[2]{#2}
\begin{thebibliography}{10}

\bibitem{AmaBel94}
L.~Amar and G.~Bellettini, \emph{{A notion of total variation depending on a
  metric with discontinuous coefficients}}, Ann. Inst. Henri Poincar\'e, Anal.
  Non Lin\'eaire \textbf{{\bf11}} (1994), 91--133.
  
\bibitem{AFP00}
L.~Ambrosio, N.~Fusco, and D.~Pallara, \emph{{Functions of Bounded Variation 
  and Free Discontinuity Problems}}, Oxford Science Publications (2000).

\bibitem{Anzellotti83a}
G.~Anzellotti, \emph{{Pairings between measures and bounded functions and
  compensated compactness}}, Ann. Mat. Pura Appl., IV. Ser. \textbf{{\bf135}}
  (1983), 293--318.

\bibitem{BecSch15}
L.~Beck and T.~Schmidt, \emph{{Convex duality and uniqueness for
  $\mathrm{BV}$-minimizers}}, J. Funct. Anal. \textbf{{\bf268}} (2015),
  3061--3107.
  
\bibitem{Bildhauer00}
M.~Bildhauer, \emph{{Convex Variational Problems}}, Lecture Notes in
  Mathematics, Springer (2000).

\bibitem{BoeDuzSch13}
V.~B{\"o}gelein, F.~Duzaar, and C.~Scheven, \emph{{Weak solutions to the heat
  flow for surfaces of prescribed mean curvature}}, Trans. Am. Math. Soc.
  \textbf{{\bf365}} (2013), 4633--4677.
  
\bibitem{BomGiu73}
E.~Bombieri, and E.~Giusti, \emph{{Local estimates for the gradient of
  non-parametric surfaces of prescribed mean curvature}}, Comm. Pure Appl.
  Math.,  \textbf{{\bf26}} (1973), 381--394.

\bibitem{CarDalLeaPas88}
M.~Carriero, G.~Dal~Maso, A.~Leaci, and E.~Pascali, \emph{{Relaxation of the
  non-parametric Plateau problem with an obstacle}}, J. Math. Pures Appl. (9)
  \textbf{{\bf67}} (1988), 359--396.

\bibitem{CarLeaPas85}
M.~Carriero, A.~Leaci, and E.~Pascali, \emph{{Integrals with respect to a Radon
  measure added to area type functionals: semi-continuity and relaxation}},
  Atti Accad. Naz. Lincei, VIII. Ser., Rend., Cl. Sci. Fis. Mat. Nat.
  \textbf{{\bf78}} (1985), 133--137.
  
\bibitem{CarLeaPas86}
M.~Carriero, A.~Leaci, and E.~Pascali, \emph{{Semicontinuity and relaxation for
  functionals that are sum of integrals of area type and of integrals with
  respect to a Radon measure}}, Rend. Accad. Naz. Sci. Detta XL, V. Ser., Mem.
  Mat. \textbf{{\bf10}} (1986), 1--31.

\bibitem{CarLeaPas87}
M.~Carriero, A.~Leaci, and E.~Pascali, \emph{{On the semicontinuity and the
  relaxation for integrals with respect to the Lebesgue measure added to
  integrals with respect to a Radon measure}}, Ann. Mat. Pura Appl. (4)
  \textbf{{\bf149}} (1987), 1--21.

\bibitem{ChaNov22}
A.~Chambolle and M.~Novaga, \emph{{Anisotropic and crystalline mean curvature
  flow of mean-convex sets}}, Ann. Sc. Norm. Super. Pisa, Cl. Sci. (5)
  \textbf{{\bf23}} (2022), 623--643.

\bibitem{CheFri99}
G.-Q.~Chen and H.~Frid, \emph{{Divergence-measure fields and hyperbolic
  conservation laws}}, Arch. Ration. Mech. Anal. \textbf{{\bf147}} (1999),
  89--118.

\bibitem{CicTro03}
M.~Cicalese and C.~Trombetti, \emph{{Asymptotic behaviour of solutions to
  $p$-Laplacian equation}}, Asymptotic Anal. \textbf{{\bf35}} (2003), 27--40.

\bibitem{ComLeo25}
G.E.~Comi and G.P.~Leonardi, \emph{{Measures in the dual of $\mathrm{BV}$:
  perimeter bounds and relations with divergence-measure fields}}, Nonlinear
  Anal., Theory Methods Appl., Ser. A \textbf{\bf251} (2025), Article ID 113686,
  28 pages.
  
\bibitem{DaiTruWan12}
Q.~Dai, N.S. Trudinger, and X.-J. Wang, \emph{{The mean curvature measure}},
  {J. Eur. Math. Soc.} \textbf{{\bf14}} (2012), 779--800.

\bibitem{DaiWanZho15}
Q.~Dai, X.-J. Wang, and B.~Zhou, \emph{{The signed mean curvature measure}},
  in: Feehan, Paul M. N. (ed.) et al., Analysis, complex geometry, and
  mathematical physics: in honor of Duong H. Phong. Contemporary Mathematics
  {\bf644} (2015) 23--32.

\bibitem{Demengel04}
F.~Demengel, \emph{{Functions locally almost $1$-harmonic}}, Appl. Anal.
  \textbf{{\bf83}} (2004), 865--896.

\bibitem{DuzSch15}
F.~Duzaar and C.~Scheven, \emph{{The evolution of $H$-surfaces with a Plateau
  boundary condition}}, Ann. Inst. Henri Poincar\'e, Anal. Non Lin\'eaire
  \textbf{{\bf32}} (2015), 109--157.

\bibitem{DuzSte96}
F.~Duzaar and K.~Steffen, \emph{{Existence of hypersurfaces with prescribed
  mean curvature in Riemannian manifolds}}, Indiana Univ. Math. J.
  \textbf{{\bf45}} (1996), 1045--1093.
  
\bibitem{EvGar15}
L.C.~Evans and R.F.~Gariepy, \emph{{Measure Theory and Fine Properties of
  Functions}}, Revised Edition, CRC Press (2015).

\bibitem{FedZie72}
H.~Federer and W.P.~Ziemer, \emph{{The Lebesgue set of a function whose
  distribution derivatives are $p$-th power summable}}, Indiana Univ. Math.
  J. \textbf{{\bf22}} (1972), 139--158.

\bibitem{FicSch25}
E.~Ficola and T.~Schmidt, \emph{{Existence theory for linear-growth
  variational integrals with signed measure data}}, Adv. Calc. Var. \textbf{19} (2026), 97--129.

\bibitem{Gerhardt74}
C.~Gerhardt, \emph{{Existence, regularity, and boundary behaviour of
  generalized surfaces of prescribed mean curvature}}, Math. Z.
  \textbf{{\bf139}} (1974), 173--198.

\bibitem{Giaquinta74b}
M.~{Giaquinta}, \emph{{On the Dirichlet problem for surfaces of prescribed mean
  curvature}}, {Manuscr. Math.} \textbf{{\bf12}} (1974), 73--86.

\bibitem{Giaquinta74a}
M.~{Giaquinta}, \emph{{Sul problema di Dirichlet per le superfici a curvatura
  media assegnata}}, {in: Symp. math. 14, Geom. simplett., Fis. mat., Teor.
  geom. Integr. Var. minim., Convegni 1973, 391--396} (1974).

\bibitem{Giusti78a}
E.~Giusti, \emph{{On the equation of surfaces of prescribed mean curvature.
  Existence and uniqueness without boundary conditions}}, Invent. Math.
  \textbf{{\bf46}} (1978), 111--137.

\bibitem{Hutchinson81}
J.E. Hutchinson, \emph{{On the relationship between Hausdorff measure and a
 measure of De Giorgi, Colombini and Piccinini}}, Boll. Unione Mat. Ital., V.
 Ser., B \textbf{{\bf18}} (1981), 619--628.

\bibitem{JerMorNac18}
R.L.~Jerrard, A.~Moradifam, and A.I. Nachman, \emph{{Existence and uniqueness of
  minimizers of general least gradient problems}}, J. Reine Angew. Math.
  \textbf{{\bf734}} (2018), 71--97.
  
\bibitem{Lahti17}
P.~Lahti, \emph{{Strict and pointwise convergence of $\BV$ functions in
  metric spaces}}, J. Math. Anal. Appl. \textbf{{\bf455}} (2017), 1005--1021.
  
\bibitem{LeoCom24}
G.P. Leonardi and G.E. Comi, \emph{{The prescribed mean curvature measure
  equation in non-parametric form}}, Preprint,
  \href{https://arxiv.org/abs/2302.10592v5}{arXiv:2302.10592v5} (2024), 62 pages.
  
\bibitem{Maggi12}
F. Maggi, \emph{{Sets of Finite Perimeter and Geometric Variational Problems,
  An Introduction to Geometric Measure Theory}}, Cambridge University Press (2012).

\bibitem{Mazon16}
J.M. Maz{\'o}n, \emph{{The Euler-Lagrange equation for the anisotropic least
  gradient problem}}, Nonlinear Anal., Real World Appl. \textbf{{\bf31}}
  (2016), 452--472.

\bibitem{MerSegTro08}
A.~Mercaldo, S.~Segura~de Le{\'o}n, and C.~Trombetti, \emph{{On the behaviour
  of the solutions to $p$-Laplacian equations as $p$ goes to $1$}}, {Publ.
  Mat., Barc.} \textbf{{\bf52}} (2008), 377--411.

\bibitem{MerSegTro09}
A.~Mercaldo, S.~Segura~de Le{\'o}n, and C.~Trombetti, \emph{{On the solutions
  to 1-Laplacian equation with $\L^1$ data}}, J. Funct. Anal. \textbf{{\bf256}}
  (2009), 2387--2416.

\bibitem{MeyZie77}
N.G. Meyers and W.P. Ziemer, \emph{{Integral inequalities of Poincar\'e and
  Wirtinger type for $\mathrm{BV}$ functions}}, Am. J. Math. \textbf{{\bf99}}
  (1977), 1345--1360.

\bibitem{Miranda74}
M.~Miranda, \emph{{Dirichlet problem with $\L^1$ data for the non-homogeneous
  minimal surface equation}}, Indiana Univ. Math. J. \textbf{{\bf24}} (1974),
  227--241.

\bibitem{Moll05}
S.~Moll, \emph{{The anisotropic total variation flow}}, Math. Ann.
  \textbf{{\bf332}} (2005), 177--218.

\bibitem{Pallara91}
D.~Pallara, \emph{{On the lower semicontinuity of certain integral
  functionals}}, Rend. Accad. Naz. Sci. XL, V. Ser., Mem. Mat. \textbf{{\bf15}}
  (1991), 57--70.

\bibitem{PhuTor08}
N.C. Phuc and M.~Torres, \emph{{Characterizations of the existence and
  removable singularities of divergence-measure vector fields}}, {Indiana Univ.
  Math. J.} \textbf{{\bf57}} (2008), 1573--1597.

\bibitem{PhuTor17}
N.C. Phuc and M.~Torres, \emph{{Characterizations of signed measures in the
  dual of $BV$ and related isometric isomorphisms}}, Ann. Sc. Norm. Super.
  Pisa, Cl. Sci. (5) \textbf{{\bf17}} (2017), 385--417.

\bibitem{SchSch15}
C.~Scheven and T.~Schmidt, \emph{{An Anzellotti type pairing for
  divergence-measure fields and a notion of weakly super-$1$-harmonic
  functions}}, Preprint, \href{https:/arxiv.org/abs/1701.02656}{arXiv:1701.02656} (2017), 7 pages.

\bibitem{SchSch16}
C.~Scheven and T.~Schmidt, \emph{{$\mathrm{BV}$\! supersolutions to
  equations of\/ $1$-Laplace and minimal surface type}}, J. Differ. Equations
  \textbf{{\bf261}} (2016), 1904--1932.

\bibitem{SchSch18}
C.~Scheven and T.~Schmidt, \emph{{On the dual formulation of obstacle
  problems for the total variation and the area functional}}, Ann. Inst. Henri
  Poincar{\'e}, Anal. Non Lin{\'e}aire \textbf{{\bf35}} (2018), 1175--1207.
  
\bibitem{Schmidt15}
T.~Schmidt, \emph{{Strict interior approximation of sets of finite perimeter
  and functions of bounded variation}}, Proc. Am. Math. Soc. \textbf{{\bf143}}
  (2015), 2069–2084.

\bibitem{Schmidt25}
T.~Schmidt, \emph{{Isoperimetric conditions, lower semicontinuity, and
  existence results for perimeter functionals with measure data}},
  Math. Ann. \textbf{{\bf391}} (2025), 5729--5807.

\bibitem{Steffen76a}
K.~Steffen, \emph{{Isoperimetric inequalities and the problem of Plateau}},
  Math. Ann. \textbf{{\bf222}} (1976), 97--144.

\bibitem{Steffen76b}
K.~Steffen, \emph{{On the existence of surfaces with prescribed mean curvature
  and boundary}}, Math. Z. \textbf{{\bf146}} (1976), 113--135.

\bibitem{Treinov21}
A.~Treinov, \emph{{The double obstacle problem for functionals with linear growth}}, 
PhD thesis, Hamburg (2021).

\bibitem{Ziemer95}
W.K. Ziemer, \emph{{The nonhomogeneous minimal surface equation involving a
  measure}}, Pac. J. Math. \textbf{{\bf167}} (1995), 183--200.

\end{thebibliography}
	
\end{document}